\documentclass[a4paper,12pt]{article}
\usepackage{amssymb}
\usepackage[latin1]{inputenc}
\usepackage{amsthm}
\usepackage{amsfonts}
\usepackage[dvips]{graphicx}
\usepackage{fancyhdr}
\usepackage{latexsym}
\usepackage{amsmath}
\usepackage{color}
\usepackage{setspace}

\oddsidemargin
=-0.0cm
\newtheorem{lem}{Lemma}
\newtheorem{thm}{Theorem}
\newtheorem{defn}{Definition}
\newtheorem{prop}{Proposition}
\newtheorem{cor}{Corollary}
\newtheorem{oss}{Remark}
\allowdisplaybreaks[4]

\numberwithin{equation}{section}
\begin{document}

\date{}
\title{{\large \textbf{ON NONLOCAL CAHN-HILLIARD-NAVIER-STOKES \\
SYSTEMS IN TWO DIMENSIONS}}}
\author{ \textsc{Sergio Frigeri} \\
Weierstrass Institute\\
10117 Berlin, Germany\\
\textit{SergioPietro.Frigeri@wias-berlin.de} \\
\\
\textsc{Ciprian G. Gal}\\
Department of Mathematics, Florida International University\\
Miami, FL 33199, USA\\
\textit{cgal@fiu.edu}\\
\\
\textsc{Maurizio Grasselli}\\
Dipartimento di Matematica, Politecnico di Milano\\
Milano 20133, Italy \\
\textit{maurizio.grasselli@polimi.it}}
\maketitle

\begin{abstract}
\noindent We consider a diffuse interface model which describes the motion
of an incompressible isothermal mixture of two immiscible fluids. This model
consists of the Navier-Stokes equations coupled with a convective nonlocal
Cahn-Hilliard equation. Several results were already proven by two of the
present authors. However, in the two-dimensional case, the uniqueness of
weak solutions was still open. Here we establish such a result even in the
case of degenerate mobility and singular potential. Moreover, we show the
weak-strong uniqueness in the case of viscosity depending on the order
parameter, provided that {\color{black} either the mobility is constant and the potential is
regular or the mobility is degenerate and the potential is singular}.
In the case of constant viscosity, on account of the uniqueness
results we can deduce the connectedness of the global attractor whose
existence was obtained in a previous paper. The uniqueness technique can be
adapted to show the validity of a smoothing property for the difference of
two trajectories which is crucial to establish the existence of an
exponential attractor. The latter is established even in the case of variable
viscosity, constant mobility and regular potential. \newline

\end{abstract}

\section{Introduction}

\setcounter{equation}{0} In a series of recent papers (see \cite%
{CFG,FG1,FG2,FGK,FGR}) the following nonlinear evolution system has been
analyzed {\color{black}
\begin{align}
&u_t-2\mbox{div}(\nu(\varphi)Du)+(u\cdot\nabla)u+\nabla\pi=\mu\nabla%
\varphi+h(t),  \label{sy3} \\
&\mbox{div}(u)=0,  \label{sy4} \\
& \varphi_t+u\cdot\nabla\varphi=\mbox{div}(m(\varphi)\nabla\mu),  \label{sy1}
\\
&\mu=a\varphi-J\ast\varphi+F^\prime(\varphi),  \label{sy2}
\end{align}
}on a bounded domain $\Omega\subset \mathbb{R}^d,$ $d=2,3$, for $t>0$. This
system describes the evolution of an isothermal mixture of two
incompressible and immiscible fluids through the (relative) concentration $%
\varphi$ of one species and the (averaged) velocity field $u$. Here $m$
denotes the mobility, $\mu$ is the so-called chemical potential, $J$ is a
spatial-dependent interaction kernel and $J\ast\varphi$ stands for spatial
convolution over $\Omega$, $a$ is defined as follows $a(x)=\int_\Omega
J(x-y)dy$, $F$ is a double well potential, $\nu$ is the viscosity and $h$ is
an external force acting on the mixture. The density is supposed to be
constant and equal to one (i.e., matched densities).

Such a system is the nonlocal version of the well-known
Cahn-Hilliard-Navier-Stokes system which has been the subject of a number of
papers (cf., e.g., \cite{A1, A2, B, CG, GG1, GG2, GG3, S, ZWH} and
references therein, see also the review \cite{Kim12} for modelling and
numerical simulation issues). We recall that the nonlocal term seems
physically more appropriate than its approximation, i.e., when in place of $%
a\varphi -J\ast \varphi $ there is $-\Delta \varphi $. For this issue, we
refer the reader to the basic papers \cite{GL0,GL1,GL2} (see also \cite%
{BH1,GG4,GZ,LP1,LP2}). However, from the mathematical viewpoint, the present
system is more challenging since the regularity of $\varphi $ is lower and
so the Korteweg force $\mu \nabla \varphi $ acting on the fluid can be less
regular than the convective term $(u\cdot \nabla )u$, even in dimension two
(cf. \cite[(3.7)]{CFG}). Therefore, it is not straightforward to extend some
of the results which holds for the Navier-Stokes equations as well as for
the standard Cahn-Hilliard-Navier-Stokes system. This is particularly
meaningful in dimension two. In fact, in dimension three, the only known
results are comparable with the standard ones for the Navier-Stokes
equations, namely, the existence of a global weak solution under various
assumptions on $m$ and $F$ and a generalized notion of attractor (cf. \cite%
{CFG,FG1,FG2,FGR}).

In dimension two, under reasonable assumptions on $F$ which ensure a
suitable regularity of $\varphi $, it is possible to prove that there exists
a weak solution which satisfies the energy identity. Therefore, such a
solution is strongly continuous in time (see \cite{CFG}). In addition,
taking advantage of the energy identity, it is also possible to prove the
existence of a the global attractor for the corresponding semiflow (cf. \cite%
{FG1,FG2,FGR}). More recently, in \cite{FGK}, assuming that $\nu $ and $m$ are
constant and taking a regular potential $F$, it has been shown the
existence of a (unique) strong solution and that any weak solution which
satisfies the energy identity regularizes in finite time. This entails some
smoothness for the global attractor. Also, the convergence of any weak to a
single equilibrium was established through the {\L }ojasiewicz-Simon
inequality approach. However, uniqueness of weak solutions was still an open
issue in \cite{CFG, FG1, FG2, FGR}.

The main goal of this paper is to prove the uniqueness of weak solutions
when $\nu $ is constant; while, when $\nu $ is non constant, we are able to
show {\color{black} the existence of a strong solution and then the weak-strong uniqueness}.
Uniqueness entails the connectedness of the
global attractor. In addition, modifying the uniqueness argument we can also
show the validity of a suitable smoothing property of the difference of two
trajectories (see \cite{EMZ,EZ}). This is the basic step to establish the
existence of an exponential attractor. The fractal dimension of the global
attractor is thus finite.


As in the previous contributions we take the following boundary and initial
conditions
\begin{align}
&\frac{\partial\mu}{\partial n}=0,\quad u=0\quad\mbox{on }
\partial\Omega\times (0,T)  \label{sy5} \\
&u(0)=u_0,\quad\varphi(0)=\varphi_0\quad\mbox{in }\Omega.  \label{sy6}
\end{align}

The plan of the paper is the following. In the next section we recall the
basic assumptions and the related existence of a weak solution. Section~\ref%
{Sec3} is devoted to the uniqueness of weak solutions for constant
viscosity. The weak-strong uniqueness is shown in Section~\ref{Sec4}. The
final Section~\ref{Sec5} is concerned with the connectedness of the global
attractor and the existence of an exponential attractor.

\section{Functional setup and preliminary results}

\label{Sec1}\setcounter{equation}{0}

Let us introduce the classical Hilbert spaces for the Navier-Stokes
equations with no-slip boundary condition (see, e.g., \cite{T})
\begin{equation*}
G_{div}:=\overline{\{u\in C^\infty_0(\Omega)^d:\mbox{ div}(u)=0\}}%
^{L^2(\Omega)^d},
\end{equation*}
and
\begin{equation*}
V_{div}:=\{u\in H_0^1(\Omega)^d:\mbox{ div}(u)=0\}.
\end{equation*}

We set $H:=L^2(\Omega)$, $V:=H^1(\Omega)$, and denote by $\|\cdot\|$ and $%
(\cdot,\cdot)$ the norm and the scalar product, respectively, on both $H$
and $G_{div}$.
The notation $\langle\cdot,\cdot\rangle$ will stand for
the duality pairing between a Banach space $X$ and its dual $X^{\prime}$. $%
V_{div}$ is endowed with the scalar product
\begin{equation*}
(u,v)_{V_{div}}=(\nabla u,\nabla v)=2(Du,Dv),\qquad\forall u,v\in V_{div},
\end{equation*}

where $D$ is the symmetric gradient, defined by $Du:=(\nabla u+(\nabla
u)^{tr})/2$.

The trilinear form $b$ which appears in the weak formulation of the
Navier-Stokes equations is defined as usual
\begin{equation*}
b(u,v,w)=\int_{\Omega}(u\cdot\nabla)v\cdot w,\qquad\forall u,v,w\in V_{div},
\end{equation*}
and the associated bilinear operator $\mathcal{B}$ from $V_{div}\times
V_{div}$ into $V_{div}^{\prime}$ is defined by $\langle\mathcal{B}%
(u,v),w\rangle:=b(u,v,w)$, for all $u,v,w\in V_{div}$. We recall that we
have $b(u,w,v)=-b(u,v,w)$, for all $u,v,w\in V_{div}$, and that the
following estimate holds in dimension two
\begin{align*}
&|b(u,v,w)|\leq c\|u\|^{1/2}\|\nabla u\|^{1/2}\|\nabla
v\|\|w\|^{1/2}\|\nabla w\|^{1/2},\quad\forall u,v,w\in V_{div}.
\end{align*}
In particular we have the following standard estimate in 2D which holds for
all $u\in V_{div}$
\begin{align}
&\|\mathcal{B}(u,u)\|_{V_{div}^{\prime}}\leq c\|u\|\|\nabla u\|.
\label{standest2D}
\end{align}

For every $f\in V^{\prime }$ we denote by $\overline{f}$ the average of $f$
over $\Omega $, i.e., $\overline{f}:=|\Omega |^{-1}\langle f,1\rangle $.
Here $|\Omega |$ is the Lebesgue measure of $\Omega $. We assume that $%
\partial \Omega $ is smooth enough (say of class $\mathcal{C}^{2}$).

We also need to introduce the Hilbert spaces
\begin{equation*}
V_0:=\{v\in V:\overline{v}=0\},\qquad V_0^{\prime}:=\{f\in V^{\prime}:%
\overline{f}=0\},
\end{equation*}
and the operator $A_N:V\to V^{\prime}$, $A_N\in\mathcal{L}(V,V^{\prime})$,
defined by
\begin{equation*}
\langle A_N u,v\rangle:=\int_{\Omega}\nabla u\cdot\nabla v\qquad\forall
u,v\in V.
\end{equation*}
We recall that $A_N$ maps $V$ onto $V_0^{\prime}$ and the restriction $B_N$
of $A_N$ to $V_0$ maps $V_0$ onto $V_0^{\prime}$ isomorphically. Further, we
denote by $B_N^{-1}:V_0^{\prime}\to V_0$ the inverse map. 
As is well known, for every $f\in V_0^{\prime}$, $B_N^{-1}f$ is the unique
solution with zero mean value of the Neumann problem
\begin{equation*}
\left\{%
\begin{array}{ll}
-\Delta u=f,\qquad\mbox{in }\Omega &  \\
\frac{\partial u}{\partial n}=0,\qquad\mbox{on }\partial\Omega. &
\end{array}%
\right.
\end{equation*}
In addition, we have
\begin{align*}
&\langle A_N u,B_N^{-1}f\rangle=\langle f,u\rangle,\qquad\forall u\in
V,\quad\forall f\in V_0^{\prime}, \\
&\langle f,B_N^{-1}g\rangle=\langle
g,B_N^{-1}f\rangle=\int_{\Omega}\nabla(B_N^{-1}f)
\cdot\nabla(B_N^{-1}g),\qquad\forall f,g\in V_0^{\prime}.
\end{align*}

Furthermore, $B_N$ can be also viewed as an unbounded linear operator on $H$
with domain $D(B_N)=\{v\in H^2(\Omega):\partial_n v=0\mbox{ on }%
\partial\Omega\}$.

If $X$ is a Banach space and $\tau\in\mathbb{R}$, we shall denote by $%
L^p_{tb}(\tau,\infty;X)$, $1\leq p<\infty$, the space of functions $f\in
L^p_{loc}([\tau,\infty);X)$ that are translation bounded in $%
L^p_{loc}([\tau,\infty);X)$, that is,
\begin{align*}
&\Vert f\Vert_{L^p_{tb}(\tau,\infty;X)}^p:=\sup_{t\geq\tau}\int_t^{t+1}\Vert
f(s)\Vert_X^p ds<\infty.
\end{align*}

We now recall the result on existence of weak solutions and on the validity
of the energy identity and of a dissipative estimate in dimension two for
the nonlocal Cahn-Hilliard-Navier-Stokes system in the case of constant
mobility, nonconstant viscosity and regular potential. This is the main case
we shall deal with in this paper.

Let us list the assumptions (see \cite{CFG}).

\begin{description}
\item[(H1)] $J\in W^{1,1}(\mathbb{R}^d),\quad J(x)=J(-x),\quad a\geq 0\quad%
\mbox{a.e. in } \Omega$.

\item[(H2)] The mobility $m(s)=1$ for all $s\in\mathbb{R}$, the viscosity $%
\nu$ is locally Lipschitz on $\mathbb{R}$ and there exist $\nu_1,\nu_2>0$
such that
\begin{equation*}
\nu_1\leq\nu(s)\leq \nu_2,\qquad\forall s\in\mathbb{R}.
\end{equation*}

\item[(H3)] $F\in C^{2,1}_{loc}(\mathbb{R})$ and there exists $c_0>0$ such
that
\begin{equation*}
F^{\prime\prime}(s)+a(x)\geq c_0,\qquad\forall s\in\mathbb{R},\quad%
\mbox{a.e. }x\in\Omega.
\end{equation*}

\item[(H4)] $F\in C^2(\mathbb{R})$ and there exist $c_1>0$, $c_2>0$ and $q>0
$ such that
\begin{equation*}
F^{\prime\prime}(s)+a(x)\geq c_1\vert s\vert^{2q} - c_2, \qquad\forall s\in%
\mathbb{R},\quad\mbox{a.e. }x\in\Omega.
\end{equation*}

\item[(H5)] There exist $c_3>0$, $c_4\geq0$ and $r\in(1,2]$ such that
\begin{equation*}
|F^\prime(s)|^r\leq c_3|F(s)|+c_4,\qquad \forall s\in\mathbb{R}.
\end{equation*}
\end{description}

\begin{oss}
{\upshape
Assumption $J\in W^{1,1}(\mathbb{R}^d)$ can be weakened. Indeed, it can be
replaced by $J\in W^{1,1}(B_{\delta})$, where $B_{\delta}:=\{z\in\mathbb{R}%
^d:|z|<\delta\}$ with $\delta:=\mbox{diam}(\Omega)$, or also by (see, e.g.,
\cite{BH1})
\begin{equation*}
\sup_{x\in\Omega}\int_{\Omega}\big(|J(x-y)|+|\nabla J(x-y)|\big)dy<\infty.
\end{equation*}
}
\end{oss}

\begin{oss}
{\upshape
Since $F$ is bounded from below, it is easy to see that (H5) implies that $F$
has polynomial growth of order {\color{black}$r^{\prime}$}, where {\color{black}$%
r^{\prime}\in[2,\infty)$} is the conjugate index to {\color{black}$r$.}
Namely, there exist $c_5>0$ and $c_6\geq 0$ such that
\begin{equation}  \label{growth}
|F(s)|\leq c_5|s|^{{\color{black}r^{\prime}}}+c_6,\qquad\forall s\in\mathbb{R}.
\end{equation}
Observe that assumption (H5) is fulfilled by a potential of arbitrary
polynomial growth. For example, (H3)--(H5) are satisfied for the case of the
well-known double well potential $F(s)=(s^2-1)^2$.}
\label{Fgrowth}
\end{oss}

The following result follows from \cite[Theorem~1, Corollaries~1 and~2]{CFG}.

\begin{thm}
\label{thm}{\color{black} Assume that (H1)--(H5) are satisfied. Let $u_0\in
G_{div}$, $\varphi_0\in H$ such that $F(\varphi_0)\in L^1(\Omega)$ and $h\in
L^2_{loc}([0,\infty);V^\prime_{div})$.} 
Then, for every given $T>0$, there exists a weak solution $[u,\varphi]$ to %
\eqref{sy1}--\eqref{sy6} 
such that
\begin{align}
&u\in L^{\infty}(0,T;G_{div})\cap L^2(0,T;V_{div}),\quad\varphi \in
L^\infty(0,T;L^{2+2q}(\Omega))\cap L^2(0,T;V),  \label{regpw1} \\
&u_t\in L^{4/3}(0,T;V_{div}^{\prime}),\quad\varphi_t\in
L^{4/3}(0,T;V^{\prime}),\qquad d=3,  \label{regpw2} \\
&u_t\in L^2(0,T;V_{div}^{\prime}),\quad d=2,  \label{regpw3} \\
&\varphi_t\in L^2(0,T;V^{\prime}),\quad d=2\quad\mbox{ or } \quad d=3
\mbox{
and } q\geq 1/2,  \label{regpw4}
\end{align}
and satisfying the energy inequality
\begin{equation}
\mathcal{E}(u(t),\varphi(t)) +\int_0^t\Big(2\|\sqrt{\nu(\varphi)}Du
\|^2+\|\nabla\mu \|^2\Big)d\tau \leq\mathcal{E}(u_0,\varphi_0)+\int_0^t%
\langle h(\tau),u \rangle d\tau,  \label{ei}
\end{equation}
for every $t>0$, where we have set
\begin{equation*}
\mathcal{E}(u(t),\varphi(t))=\frac{1}{2}\|u(t)\|^2+\frac{1}{4}
\int_{\Omega}\int_{\Omega}J(x-y)(\varphi(x,t)-\varphi(y,t))^2
dxdy+\int_{\Omega}F(\varphi(t)).
\end{equation*}
If $d=2$, then any weak solution satisfies the energy identity
\begin{equation}
\frac{d}{dt}\mathcal{E}(u,\varphi) +2\|\sqrt{\nu(\varphi)}Du
\|^2+\|\nabla\mu\|^2=\langle h(t),u\rangle,  \label{eniden}
\end{equation}
In particular we have $u\in C([0,\infty);G_{div})$, $\varphi\in
C([0,\infty);H)$ and $\int_\Omega F(\varphi)\in C([0,\infty))$. Furthermore,
if $d=2$ and $h\in L^2_{tb}(0,\infty;V_{div}^{\prime})$, 
then any weak solution satisfies also the dissipative estimate
\begin{equation}
\mathcal{E}(u(t),\varphi(t))\leq \mathcal{E}(u_0,\varphi_0)e^{-kt}+
F(m_0)|\Omega| + K, \qquad\forall t\geq 0,  \label{dissest}
\end{equation}
where $m_0=(\varphi_0,1)$ and $k$, $K$ are two positive constants which are
independent of the initial data, with $K$ depending on $\Omega$, $\nu$, $J$,
$F$ and $\|h\|_{L^2_{tb}(0,\infty;V_{div}^{\prime})}$.
\end{thm}

{\color{black} Henceforth we shall denote by $Q$ a continuous function
monotone increasing with respect to each of its arguments. As a consequence of
energy inequality \eqref{ei}
it is easy to deduce the following bound
\begin{align}
&\Vert u\Vert_{L^\infty(0,T;G_{div})\cap
L^2(0,T;V_{div})}+\Vert\varphi\Vert_{L^\infty(0,T;L^{2+2q}(\Omega))\cap
L^2(0,T;V)} +\Vert F(\varphi)\Vert_{L^\infty(0,T;L^1(\Omega))}  \notag \\
&\leq Q\big(\mathcal{E}(u_0,\varphi_0),\Vert
h\Vert_{L^2(0,T;V_{div}^{\prime})}\big),  \label{est36}
\end{align}
where $Q$ also depends on $F,J,\nu_1$ and $\Omega$.
In all the following sections we take $d=2$}.

\section{Uniqueness of weak solutions (constant viscosity)}

\label{Sec3}\setcounter{equation}{0}

{\color{black} Here} we prove that the weak solution of the nonlocal
Cahn-Hilliard-Navier-Stokes system {\color{black}with constant viscosity $\nu$}
is unique and {\color{black} we} provide a continuous dependence estimate. In Subsection \ref%
{regpot} we shall first address the case of constant mobility ($m=1$)
and regular potential $F$. Nevertheless, we shall see in Subsection \ref%
{singpotcm} and Subsection \ref{singpotdm} that the arguments used for this
case can also be applied to the cases of singular potential and constant or
degenerate mobility (see \cite{FG2} or \cite{FGR} for the existence).

\subsection{Regular potential and constant mobility}

\label{regpot}

The main result is the following.

\begin{thm}
\label{uniqthm} Let $d=2$ and suppose that assumptions (H1)--(H5) are
satisfied with {\color{black}$\nu$ constant}. {\color{black}Let $u_0\in G_{div}$%
, $\varphi_0\in H$ with $F(\varphi_0)\in L^1(\Omega)$ and $h\in
L^2_{loc}([0,\infty);V^\prime_{div})$.} Then, the weak solution $[u,\varphi]$
corresponding to $[u_0,\varphi_0]$ and given by Theorem \ref{thm} is unique.
Furthermore, {\color{black} let $z_i:=[u_i,\varphi_i]$ be two weak solutions
corresponding to two initial data $z_{0i}:=[u_{0i},\varphi_{0i}]$ {%
\color{black}and external forces $h_i$}, with $u_{0i}\in G_{div}$, $%
\varphi_{0i}\in H$ such that $F(\varphi_{0i})\in L^1(\Omega)$ and $h_i\in
L^2_{loc}([0,\infty);V^\prime_{div})$.}
Then the following continuous dependence estimate holds
\begin{align}
&\Vert
u_2(t)-u_1(t)\Vert^2+\Vert\varphi_2(t)-\varphi_1(t)\Vert_{V^{\prime}}^2
\notag \\
&+\int_0^t\Big(\frac{c_0}{2}\Vert\varphi_2(\tau)-\varphi_1(\tau)\Vert^2 +%
\frac{\nu}{4} \Vert\nabla\big(u_2(\tau)-u_1(\tau)\big)\Vert^2\Big)d\tau
\notag \\
&\leq\big(\Vert
u_2(0)-u_1(0)\Vert^2+\Vert\varphi_2(0)-\varphi_1(0)\Vert_{V^{\prime}}^2\big) %
\Lambda_0(t)  \notag \\
&+{\color{black}\big|\overline{\varphi_2(0)}-\overline{\varphi_1(0)}\big|
Q\big(\mathcal{E}(z_{01}),\mathcal{E}(z_{02}),\Vert
h_1\Vert_{L^2(0,t;V_{div}^{\prime})},\Vert
h_2\Vert_{L^2(0,t;V_{div}^{\prime})}\big)} \Lambda_1(t)  \notag \\
&+\Vert h_2-h_1\Vert_{L^2(0,T;V_{div}^{\prime})}^2\Lambda_2(t),
\label{estcontdip}
\end{align}
for all $t\in [0,T]$, where $\Lambda_0$, $\Lambda_1$ and $\Lambda_2$ are
continuous functions which depend on the norms of the two solutions. {%
\color{black} The functions $Q$ and $\Lambda_i$ also depend on $%
F,J,\nu$ and $\Omega$. } 
\end{thm}

\begin{proof}
Let us start by rewriting the Korteweg force by making explicit the
dependence on $\varphi$. Indeed, we have
\begin{align*}
& \mu\nabla\varphi=\big(a\varphi -J\ast\varphi+F^{\prime }(\varphi)\big)%
\nabla\varphi=\nabla\Big(F(\varphi)+a\frac{\varphi ^{2}}{2}\Big) -\nabla a%
\frac{\varphi ^{2}}{2}-(J\ast\varphi)\nabla\varphi.
\end{align*}
Hence we can write the Navier-Stokes equation with an extra-pressure $%
\widetilde{\pi}:=\pi-F(\varphi)+a\frac{\varphi^2}{2}$ as follows
\begin{align*}
& u_t-\nu\Delta u+(u\cdot\nabla)u+\nabla\widetilde{\pi} -h =-\nabla a\frac{%
\varphi^2}{2} -(J\ast\varphi)\nabla\varphi=:K(\varphi).
\end{align*}
Let us now consider two weak solutions $[u_i,\varphi_i]$ corresponding to
two initial data $[u_{0i},\varphi_{0i}]$ {\color{black}and two external forces
$h_i$}, with $u_{0i}\in G_{div}$, $\varphi_{i0}\in H$, $%
F(\varphi_{0i})\in L^1(\Omega)$ {\color{black} and $h_i\in
L^2_{loc}([0,\infty);V^\prime_{div})$}, $i=1,2$.  Set $u:=u_2-u_1$ and $%
\varphi:=\varphi_2-\varphi_1$. Then, the difference $[u,\varphi]$ satisfies
the system
\begin{align}
&\varphi_t=\Delta\widetilde{\mu}-u\cdot\nabla\varphi_1{\color{black}-u_2\cdot
\nabla\varphi},  \label{CHdiff} \\
&\widetilde{\mu}=a\varphi-J\ast\varphi+F^{\prime}(\varphi_2)-F^{\prime}(%
\varphi_1),  \label{chpotdiff} \\
&u_t-\nu\Delta u+ (u_2\cdot\nabla )u_2-(u_1\cdot \nabla)u_1+\nabla\widetilde{%
\pi}  \notag \\
&=-\varphi(\varphi_1+\varphi_2)\frac{\nabla a}{2}-(J\ast\varphi)\nabla
\varphi_2-(J\ast\varphi_1)\nabla\varphi +h ,  \label{NSdiff}
\end{align}
where $\widetilde{\pi}:=\widetilde{\pi}_2-\widetilde{\pi}_1$ and $h:=h_2-h_1$%
. 
We multiply \eqref{NSdiff} by $u$ in $G_{div}$. After standard calculations,
the following terms (cf. \eqref{NSdiff})
\begin{align*}
I_{1}=-\frac{1}{2}\left(\varphi \left( \varphi _{1}+\varphi
_{2}\right)\nabla a,u\right), \text{ }I_{2}=-\left( \left( J\ast \varphi
\right) \nabla \varphi _{2},u\right) ,\text{ }I_{3}=-\left( \left( J\ast
\varphi _{1}\right) \nabla \varphi ,u\right),
\end{align*}
can be estimated in this way
\begin{align}
&I_1\leq\big|\big(\varphi(\varphi_1+\varphi_2)\nabla a,u\big)\big| %
\leq\Vert\varphi\Vert\Vert\varphi_1+\varphi_2\Vert_{L^4}\Vert\nabla
a\Vert_{L^\infty}\Vert u\Vert_{L^4}  \notag \\
&\leq c\Vert\varphi\Vert\Vert\varphi_1+\varphi_2\Vert_{L^4}\Vert\nabla
a\Vert_{L^\infty}\Vert u\Vert^{1/2} \Vert\nabla u\Vert^{1/2}  \notag \\
&\leq \frac{c_0}{10}\Vert\varphi\Vert^2+c\Vert\varphi_1+\varphi_2%
\Vert_{L^4}^2 \Vert\nabla a\Vert_{L^\infty}^2\Vert u\Vert\Vert\nabla u\Vert
\notag \\
&\leq\frac{c_0}{10}\Vert\varphi\Vert^2+\frac{\nu}{6}\Vert\nabla u\Vert^2
+c\Vert\varphi_1+\varphi_2\Vert_{L^4}^4 \Vert\nabla a\Vert_{L^\infty}^4\Vert
u\Vert^2,  \label{est1} \\
& I_2\leq\big|\big({\color{black}\varphi_2},(\nabla J\ast\varphi)u\big)\big| %
\leq\Vert{\color{black}\varphi_2}\Vert_{L^4}\Vert\nabla J\ast\varphi\Vert\Vert
u\Vert_{L^4}  \notag \\
&\leq c\Vert{\color{black}\varphi_2}\Vert_{L^4}\Vert\nabla
J\Vert_{L^1}\Vert\varphi\Vert\Vert u\Vert^{1/2} \Vert\nabla u\Vert^{1/2}
\notag \\
&\leq\frac{c_0}{10}\Vert\varphi\Vert^2+c\Vert\nabla J\Vert_{L^1}^2\Vert{%
\color{black}\varphi_2}\Vert_{L^4}^2 \Vert u\Vert\Vert\nabla u\Vert  \notag \\
&\leq\frac{c_0}{10}\Vert\varphi\Vert^2+\frac{\nu}{6}\Vert\nabla u\Vert^2
+c\Vert\nabla J\Vert_{L^1}^4\Vert{\color{black}\varphi_2}\Vert_{L^4}^4 \Vert
u\Vert^2,  \label{est2} \\
&I_3\leq\big|\big((\nabla J\ast{\color{black}\varphi_1})\varphi,u\big)\big| %
\leq\Vert\nabla J\ast{\color{black}\varphi_1}\Vert_{L^4}\Vert\varphi\Vert\Vert
u\Vert_{L^4}  \notag \\
&\leq c\Vert\nabla J\Vert_{L^1}\Vert{\color{black}\varphi_1}%
\Vert_{L^4}\Vert\varphi\Vert\Vert u\Vert^{1/2}\Vert\nabla u\Vert^{1/2}
\notag \\
&\leq\frac{c_0}{10}\Vert\varphi\Vert^2+ c\Vert\nabla J\Vert_{L^1}^2\Vert{%
\color{black}\varphi_1}\Vert_{L^4}^2\Vert u\Vert\Vert\nabla u\Vert  \notag \\
&\leq\frac{c_0}{10}\Vert\varphi\Vert^2+\frac{\nu}{6}\Vert\nabla u\Vert^2
+c\Vert\nabla J\Vert_{L^1}^4\Vert{\color{black}\varphi_1}\Vert_{L^4}^4\Vert
u\Vert^2.  \label{est3}
\end{align}
Taking estimates \eqref{est1}--\eqref{est3} into account, it is easy see that from \eqref{NSdiff}
we are led to the following differential inequality
\begin{align}
\frac{1}{2}\frac{d}{dt}\Vert u\Vert^2+ \frac{\nu}{4} \Vert\nabla u\Vert^2
\leq\frac{3}{10}c_0\Vert\varphi\Vert^2+\alpha\Vert u\Vert^2 +\frac{1}{\nu}%
\Vert h\Vert_{V_{div}^{\prime}}^2,  \label{est4}
\end{align}
where the function $\alpha$ is given by
\begin{align}
&\alpha:=c\Vert\nabla J\Vert_{L^1}^4\big(\Vert\varphi_1\Vert_{L^4}^4+\Vert%
\varphi_2\Vert_{L^4}^4\big) +c\Vert\nabla u_2\Vert^2.  \label{defalpha}
\end{align}
Since $\varphi_1,\varphi_2\in L^\infty(0,T;H)\cap L^2(0,T,V)$ and $%
L^\infty(0,T;H)\cap L^2(0,T,V)\hookrightarrow L^4(0,T;L^4(\Omega))$,
thanks to the Gagliardo-Nirenberg inequality, we have $\alpha\in
L^1(0,T)$.

Let us now multiply \eqref{CHdiff} by $B_N^{-1}(\varphi-\overline{\varphi})$
(notice that we have $\overline{\varphi}=\overline{\varphi}_{01}-\overline{%
\varphi}_{02}$). We get
\begin{align}
& \frac{1}{2}\frac{d}{dt}\Vert B_{N}^{-1/2}(\varphi-\overline{\varphi}%
)\Vert^2+(a\varphi+F^{\prime}(\varphi_1)-F^{\prime}(\varphi_2),\varphi)=(J%
\ast\varphi,\varphi) +|\Omega|\overline{\varphi}\overline{\widetilde{\mu}}
+I_{4}+I_{5},  \label{est13}
\end{align}
where {\color{black}
\begin{equation*}
I_{4}=-\left( u\cdot \nabla \varphi _1,B_{N}^{-1}(\varphi-\overline{\varphi}%
) \right) ,\quad I_{5}=-\left( u_2\cdot \nabla\varphi ,B_{N}^{-1}(\varphi-%
\overline{\varphi})\right).
\end{equation*}%
} By using assumption (H3), we find
\begin{equation}
\frac{1}{2}\frac{d}{dt}\Vert B_{N}^{-1/2}(\varphi-\overline{\varphi}%
)\Vert^2+c_{0}\Vert\varphi\Vert^2\leq|(J\ast\varphi,\varphi)| +|\Omega|%
\overline{\varphi}\overline{\widetilde{\mu}} +I_{4}+I_{5}.  \label{est}
\end{equation}

The first term on the right-hand side of \eqref{est} can be controlled as
follows
\begin{align}
&\big|(J\ast\varphi,\varphi-\overline{\varphi})\big|+|(J\ast\varphi,%
\overline{\varphi})| {\color{black}=|\big(B_N^{1/2}(J\ast\varphi-\overline{%
J\ast\varphi}),B_N^{-1/2}(\varphi-\overline{\varphi})\big)\big|%
+|(J\ast\varphi,\overline{\varphi})|}  \notag \\
&\leq\frac{c_0}{10}\Vert\varphi\Vert^2+c\Vert B_N^{-1/2}(\varphi-\overline{%
\varphi})\Vert^2 +\frac{c_0}{4}\Vert\varphi\Vert^2+c\overline{\varphi}^2,
\label{est5}
\end{align}
{\color{black}where we have used the fact that $\Vert B_N^{1/2}
u\Vert^2=(B_N u,u)=\Vert\nabla u\Vert^2$, for all $u\in D(B_N)$ and hence $%
\Vert B_N^{1/2} u\Vert=\Vert\nabla u\Vert$, which also holds, by density,
for all $u\in D(B_N^{1/2})=V_0$. 
The terms $I_4$ and $I_5$ can be estimated in this way}
\begin{align}
&I_4\leq\big|\big(u\cdot\nabla B_N^{-1}(\varphi-\overline{\varphi}),{%
\color{black}\varphi_1}\big)\big|\leq \Vert u\Vert_{L^4}\Vert\nabla
B_N^{-1}(\varphi-\overline{\varphi})\Vert\Vert{\color{black}\varphi_1}%
\Vert_{L^4}  \notag \\
&\leq \frac{\nu}{8} \Vert\nabla u\Vert^2 +c\Vert{\color{black}\varphi_1}%
\Vert_{L^4}^2\Vert{\color{black}B_N^{-1/2}}(\varphi-\overline{\varphi})\Vert^2,
\label{est6}\\
&I_5\leq\big|(u_2\cdot\nabla B_N^{-1}(\varphi-\overline{\varphi}),\varphi)%
\big| \leq \Vert\varphi\Vert\Vert u_2\Vert_{L^4}\Vert\nabla B_N^{-1}(\varphi-%
\overline{\varphi})\Vert_{L^4}  \notag \\
&\leq \frac{c_0}{20}\Vert\varphi\Vert^2+c\Vert u_2\Vert_{L^4}^2\Vert\nabla
B_N^{-1}(\varphi-\overline{\varphi})\Vert_{L^4}^2  \notag \\
&\leq \frac{c_0}{20}\Vert\varphi\Vert^2 +c\Vert u_2\Vert_{L^4}^2\Vert\nabla
B_N^{-1}(\varphi-\overline{\varphi})\Vert \Vert\nabla B_N^{-1}(\varphi-%
\overline{\varphi})\Vert_{H^1}.  \label{est7}
\end{align}
Observe that the $H^2$-norm of $\phi$ on $D(B_N)$ is equivalent to the $L^2$%
-norm of $B_N\phi+\phi$ (recall that $\phi:=B_N^{-1}(\varphi-\overline{%
\varphi})\in D(B_N)$). Thus we have
\begin{align}
&\Vert\nabla B_N^{-1}(\varphi-\overline{\varphi})\Vert_{H^1}\leq\Vert
B_N^{-1}(\varphi-\overline{\varphi})\Vert_{H^2}\leq c\Vert
(B_N+I)B_N^{-1}(\varphi-\overline{\varphi})\Vert \leq c\Vert\varphi-%
\overline{\varphi}\Vert.  \notag
\end{align}
Therefore, from \eqref{est7} we get
\begin{align}
I_5\leq\frac{c_0}{10}\Vert\varphi\Vert^2+c\Vert u_2\Vert_{L^4}^4\Vert
B_N^{-1/2}(\varphi-\overline{\varphi})\Vert^2 +|\Omega|\overline{\varphi}^2.
\label{est8}
\end{align}

{\color{black} Recalling estimate \eqref{est4} and plugging estimates  \eqref{est5}--\eqref{est8}
into \eqref{est}}, we deduce the differential inequality
\begin{align}
&\frac{1}{2}\frac{d}{dt}\Big(\Vert u\Vert^2+\Vert B_N^{-1/2}(\varphi-%
\overline{\varphi})\Vert^2\Big)+ \frac{c_0}{4}\Vert\varphi\Vert^2+\frac{\nu}{%
8} \Vert\nabla u\Vert^2  \notag \\
&\leq\beta\Big(\Vert u\Vert^2+\Vert B_N^{-1/2}(\varphi-\overline{\varphi}%
)\Vert^2\Big) +c\overline{\varphi}^2+|\Omega|\overline{\varphi}\overline{%
\widetilde{\mu}} +\frac{1}{\nu}\Vert h\Vert_{V_{div}^{\prime}}^2,
\label{est9}
\end{align}
where $\beta$ is given by
\begin{align}
&\beta:=\alpha+c(1+\Vert{\color{black}\varphi_1}\Vert_{L^4}^2+\Vert{\color{black}
u_2}\Vert_{L^4}^4)\in L^1(0,T).  \notag
\end{align}

If we consider two weak solutions corresponding to the same initial data and
to the same external force, then we have $\overline{\varphi}=0$ and $h=0$.
Therefore, by using Gronwall's lemma, from \eqref{est9}  we get $u=0$ and $%
\varphi=0$ on $[0,T]$ and this proves uniqueness.

If the two weak solutions correspond to different initial data and to
different external forces, we have
\begin{align}
&|\Omega||\overline{\widetilde{\mu}}|\leq \int_\Omega\big(%
|F^{\prime}(\varphi_2)|+|F^{\prime}(\varphi_1)|\big) \leq{\color{black} c}%
\int_\Omega\big(|F(\varphi_2)|+|F(\varphi_1)|\big)+c  \notag \\
& {\color{black}\leq Q\big(\mathcal{E}(z_{01}),\mathcal{E}%
(z_{02}),\Vert h_1\Vert_{L^2(0,T;V_{div}^{\prime})},\Vert
h_2\Vert_{L^2(0,T;V_{div}^{\prime})}\big)}, \qquad\forall t\geq 0,
\label{est35}
\end{align}
where we have used (H5) (which implies that $|F^{\prime}(s)|\leq c F(s)+c$,
for all $s\in\mathbb{R}$) {\color{black} and \eqref{est36}.
Therefore \eqref{est9} can be rewritten as}
\begin{align}
&\frac{d}{dt}\Big(\Vert u\Vert^2+\Vert B_N^{-1/2}(\varphi-\overline{\varphi}%
)\Vert^2\Big)+ \frac{c_0}{2}\Vert\varphi\Vert^2+ \frac{\nu}{4}\Vert\nabla
u\Vert^2  \notag \\
& \leq\beta\Big(\Vert u\Vert^2+\Vert B_N^{-1/2}(\varphi-\overline{\varphi}%
)\Vert^2\Big) +|\overline{\varphi}| {\color{black}Q\big(\mathcal{E}%
(z_{01}),\mathcal{E}(z_{02}),\Vert
h_1\Vert_{L^2(0,T;V_{div}^{\prime})},\Vert
h_2\Vert_{L^2(0,T;V_{div}^{\prime})}\big)}  \notag \\
&+\frac{2}{\nu}\Vert h\Vert_{V_{div}^{\prime}}^2.  \label{est10}
\end{align}
By using Gronwall's lemma once more, we deduce from \eqref{est10} that
\begin{align}
& \Vert u(t)\Vert^2+\Vert B_N^{-1/2}(\varphi(t)-\overline{\varphi})\Vert^2
\leq\big(\Vert u(0)\Vert^2+\Vert B_N^{-1/2}(\varphi(0)-\overline{\varphi}%
)\Vert^2\big) \Gamma_0(t)  \notag \\
& +|\overline{\varphi}|{\color{black}Q\big(\mathcal{E}(z_{01}),%
\mathcal{E}(z_{02}),\Vert h_1\Vert_{L^2(0,T;V_{div}^{\prime})},\Vert
h_2\Vert_{L^2(0,T;V_{div}^{\prime})}\big)} \Gamma_1(t) +\frac{2}{\nu}%
\Gamma_0(t)\Vert h\Vert_{L^2(0,T;V_{div}^{\prime})}^2,  \label{est11}
\end{align}
where $\Gamma_0(t):=e^{\int_0^t\beta(s)ds}$ and $\Gamma_1(t):=\int_0^t
e^{\int_s^t\beta(\tau)d\tau} ds$. By integrating \eqref{est10} between $0$
and $t$ and using \eqref{est11}, we find
\begin{align}
&\Vert u(t)\Vert^2+\Vert B_N^{-1/2}(\varphi(t)-\overline{\varphi})\Vert^2
+\int_0^t\Big(\frac{c_0}{2}\Vert\varphi\Vert^2+ \frac{\nu}{4}\Vert\nabla
u\Vert^2\Big)d\tau  \notag \\
& \leq\big(\Vert u(0)\Vert^2+\Vert B_N^{-1/2}(\varphi(0)-\overline{\varphi}%
)\Vert^2\big) \Gamma_2(t)  \notag \\
&+|\overline{\varphi}|{\color{black}Q\big(\mathcal{E}(z_{01}),%
\mathcal{E}(z_{02}),\Vert h_1\Vert_{L^2(0,T;V_{div}^{\prime})},\Vert
h_2\Vert_{L^2(0,T;V_{div}^{\prime})}\big)} \Gamma_3(t)  \notag \\
&+\frac{2}{\nu}\Gamma_0(t)\Vert h\Vert_{L^2(0,T;V_{div}^{\prime})}^2,
\label{est12}
\end{align}
for all $t\in [0,T]$, where $\Gamma_2(t):=1+\int_0^t\beta(s)\Gamma_0(s)ds$ and $%
\Gamma_3(t):=\int_0^t\beta(s)\Gamma_1(s)ds+T$. Finally, by suitably defining the functions $\Lambda_0$,
$\Lambda_1$ in terms of $\Gamma_0$, $\Gamma_2$ and $\Gamma_ 3$,
we deduce \eqref{estcontdip} from \eqref{est12}.
\end{proof}



\subsection{Singular potential and constant mobility}

\label{singpotcm} The proof of existence of a weak solution with initial
data $u_0\in G_{div}$ and $\varphi_0\in L^\infty(\Omega)$ with $%
F(\varphi_0)\in L^1(\Omega)$ is given in \cite{FG2}, where also a
nonconstant viscosity is considered. We recall that in this case the
assumption $|\overline{\varphi}_0|<1$ is needed in order to control the
average of the chemical potential. For the assumptions on the singular
potential $F$ we refer the reader to \cite{FG2}. We recall, in particular,
the physically relevant case of the so-called logarithmic potential, that
is,
\begin{equation}
F(s)=-\frac{\theta_c}{2}s^2+\frac{\theta}{2}\big((1+s)\log(1+s)+(1-s)%
\log(1-s)\big),  \label{potlog}
\end{equation}
where $0<\theta<\theta_c$, $\theta$ being the absolute temperature and $%
\theta_c$ a given critical temperature below which the phase separation
takes place.

It is easy to see that, assuming the viscosity $\nu$ constant and $d=2$, the
uniqueness argument can also be applied to the present case. Indeed,
estimates \eqref{est1}-\eqref{est4} obviously still hold. Moreover,
considering \eqref{est13} we immediately see that \eqref{est} still follows
from \eqref{est13}, since in the case of singular potential we have
\begin{equation*}
F^{\prime\prime}(s)+a(x)\geq c_0,\qquad\forall s\in(-1,1),\qquad c_0>0.
\end{equation*}
In particular, this assumption is ensured by \cite[(A6)]{FG2}. Therefore,
uniqueness {\color{black} follows from \eqref{est9} on account of the
fact that in this inequality we have $\overline{\varphi}=0$ (and $h=0$). }

{\color{black} Concerning the proof of the continuous dependence estimate \eqref{estcontdip},
we have to be a bit more careful since estimate \eqref{est35} cannot be
applied in the present situation.

On the other hand, recalling \cite[Proof of Theorem 1]{FG2}, we have
\begin{align*}
&\Vert F^{\prime}(\varphi_i)\Vert_{L^2(0,T;L^1(\Omega))} \leq Q\big(%
\overline{\varphi}_{0i},\mathcal{E}(z_{0i}),\Vert
h_i\Vert_{L^2(0,T;V_{div}^{\prime})}\big),\qquad i=1,2.
\end{align*}
By applying these last estimates we see that the term $|\Omega|\overline{%
\varphi}\overline{\widetilde{\mu}}$ on the right-hand side of \eqref{est9}
can be written in the form $\overline{\varphi}\Gamma_4$ with a function $%
\Gamma_4$ such that 
\begin{align*}
&\Vert\Gamma_4\Vert_{L^2(0,T)}\leq Q\big(\eta,\mathcal{E}(z_{01}),%
\mathcal{E}(z_{02}),\Vert h_1\Vert_{L^2(0,T;V_{div}^{\prime})}, \Vert
h_2\Vert_{L^2(0,T;V_{div}^{\prime})}\big),
\end{align*}
where $\eta\in[0,1)$ is such that $|\overline{\varphi}_{0i}|\leq\eta$, $i=1,2
$. Starting now from \eqref{est9} and using Gronwall's lemma like in the proof
of Theorem \ref{uniqthm}, we find a continuous dependence estimate of
the same form as \eqref{estcontdip} where now the function $Q$
depends also on $\eta$. We can therefore state the following}  

\begin{thm}
\label{uniqthmsing} {\color{black} Let $d=2$ and suppose that assumptions
(A1)--(A8) of \cite{FG2} are satisfied with $\nu$ constant. Let $u_0\in
G_{div}$, $\varphi_0\in L^\infty(\Omega)$ with $F(\varphi_0)\in L^1(\Omega)$%
, $|\overline{\varphi}_0|<1$ and $h\in L^2_{loc}([0,\infty);V_{div}^{\prime})
$.}  Then, the weak solution $[u,\varphi]$, corresponding to $[u_0,\varphi_0]$
and given by \cite[Theorem 1]{FG2}, is unique. Furthermore, {\color{black}
Let $z_i:=[u_i,\varphi_i]$ be two weak solutions} corresponding to two
initial data $z_{0i}:=[u_{0i},\varphi_{0i}]$ {\color{black} and two external
forces $h_i$}, with $u_{0i}\in G_{div}$, $\varphi_{0i}\in L^\infty(\Omega)$
such that $F(\varphi_{0i})\in L^1(\Omega)$, $|\overline{\varphi}%
_{0i}|\leq\eta$ for some constant $\eta\in[0,1)$ {\color{black} and $h_i\in
L^2_{loc}([0,\infty);V_{div}^{\prime})$}, $i=1,2$. Then estimate %
\eqref{estcontdip} {\color{black} holds with $Q$ also depending on $\eta$}.
\end{thm}

\subsection{Singular potential and degenerate mobility}

\label{singpotdm} This physically relevant case was addressed in \cite{FGR} {%
\color{black} 
from which we recall all the assumptions on the degenerate mobility $m$ and
on the singular potential $F$ as well as the weak formulation.
We assume that the mobility $m$ is degenerate at $\pm1$ and that the
double well potential $F$ is singular (e.g. logarithmic like) and defined in
$(-1,1)$. More precisely, we assume that $m\in C^1([-1,1])$, $m\geq 0$, that
$m(s)=0$ if and only if $s=-1$ or $s=1$, and that there exists $\epsilon_0>0$
such that $m$ is non-increasing in $[1-\epsilon_0,1]$ and non-decreasing in $%
[-1,-1+\epsilon_0]$. Furthermore, we suppose that $m$ and $F$ fulfill the
condition

\begin{description}
\item[(A1)] $F\in C^2(-1,1)$ and $mF^{\prime\prime}\in C([-1,1])$.
\end{description}

We point out that (A1) is a typical condition which arises in the
Cahn-Hilliard equation with degenerate mobility (see \cite{EG,GL1,GL2,GZ}).

As far as $F$ is concerned
we assume that it can be written in the following form
\begin{equation*}
F=F_1+F_2,
\end{equation*}
where the singular component $F_1$ and the regular component $F_2\in
C^2([-1,1])$ satisfy the following assumptions.

\begin{description}
\item[(A2)] There exist $\kappa>4(a^\ast-a_\ast-b_\ast)$, where $%
b_\ast:=\min_{[-1,1]}F_2^{\prime\prime}$,
and $\epsilon_0>0$ such that
\begin{align*}
&F_1^{\prime\prime}(s)\geq \kappa,\qquad\forall
s\in(-1,-1+\epsilon_0]\cup[1-\epsilon_0,1).
\end{align*}


\item[(A3)] There exists $\epsilon_0>0$ such that $F_1^{\prime\prime}$ is
non-decreasing in $[1-\epsilon_0,1)$ and non-increasing in $%
(-1,-1+\epsilon_0]$.

\item[(A4)] There exists $c_0>0$ such that
\begin{align*}
&F^{\prime\prime}(s)+a(x)\geq c_0,\qquad\forall s\in(-1,1),\qquad\mbox{a.e. }%
x\in\Omega.
\end{align*}
\end{description}

The constants $a^\ast$ and $a_\ast$ in (A2) are given by
\begin{align*}
&a^\ast:=\sup_{x\in\Omega}\int_\Omega|J(x-y)|dy<\infty, \qquad
a_\ast:=\inf_{x\in\Omega}\int_\Omega J(x-y)dy.
\end{align*}
Moreover, we denote by $\epsilon_0$ a positive constant the value of which
may possibly vary from line to line.

It is worth recalling that a typical situation is $m(s)=k_1(1-s^2)$ and $F$
given by \eqref{potlog}. We also recall that in \cite{FGR} the viscosity $\nu$ was
assumed to be constant just to avoid technicalities, but the results therein
also hold for a nonconstant viscosity satisfying (H2).

As far as the weak formulation is concerned, we point out that,
if the mobility degenerates then the gradient of the chemical
potential $\mu$ is not controlled  in some $L^p$ space. For this reason, and
also in order to pass to the limit to prove existence of a weak solution, a
suitable reformulation of the definition of weak solution should be
introduced in such a way that $\mu$ does not appear explicitly (cf. \cite{EG}, see also  \cite{FGR}).

\begin{defn}
Let $u_0\in G_{div}$, $\varphi_0\in H$ with $F(\varphi_0)\in L^1(\Omega)$, $%
h\in L^2(0,T;V_{div}^{\prime})$  and $0<T<+\infty$ be given. A couple $%
[u,\varphi]$ is a weak solution to \eqref{sy1}-\eqref{sy6} on $[0,T]$
corresponding to $[u_0,\varphi_0]$ if

\begin{itemize}
\item $u$, $\varphi$ satisfy
\begin{align*}
&u\in L^{\infty}(0,T;G_{div})\cap L^2(0,T;V_{div}), \\
&u_t\in L^{4/3}(0,T;V_{div}^{\prime}),\qquad\mbox{if}\quad d=3, \\
&u_t\in L^2(0,T;V_{div}^{\prime}),\qquad\mbox{if}\quad d=2, \\
&\varphi\in L^{\infty}(0,T;H)\cap L^2(0,T;V), \\
&\varphi_t\in L^2(0,T;V^{\prime}),
\end{align*}
and
\begin{align*}
&\varphi\in L^{\infty}(Q_T),\qquad|\varphi(x,t)|\leq 1\quad\mbox{a.e. }%
(x,t)\in Q_T:=\Omega\times(0,T);
\end{align*}

\item for every $\psi\in V$, every $v\in V_{div}$ and for almost any $%
t\in(0,T)$ we have
\begin{align*}
&\langle\varphi_t,\psi\rangle+\int_\Omega
m(\varphi)F^{\prime\prime}(\varphi)\nabla\varphi\cdot\nabla\psi+ \int_\Omega
m(\varphi) a \nabla\varphi\cdot\nabla\psi  \notag \\
&+\int_\Omega m(\varphi)(\varphi\nabla a-\nabla
J\ast\varphi)\cdot\nabla\psi=(u\varphi,\nabla\psi), \\
&\langle u_t,v\rangle+\nu(\nabla u,\nabla v)+b(u,u,v)=\big(%
(a\varphi-J\ast\varphi)\nabla\varphi,v\big)+\langle h,v\rangle;
\end{align*}

\item the initial conditions $u(0)=u_0$, $\varphi(0)=\varphi_0$ hold.
\end{itemize}
\end{defn}

Recall also that from the regularity properties of the weak solution we have
$u\in C_w([0,T];G_{div})$ and $\varphi\in C_w([0,T];H)$. Therefore, the
initial conditions $u(0)=u_0$, $\varphi(0)=\varphi_0$ make sense.

In \cite[Theorem 2]{FGR} the existence of a weak solution was
established with initial data $u_0\in G_{div}$ and $\varphi_0\in
L^\infty(\Omega)$ with $F(\varphi_0)\in L^1(\Omega)$ and $M(\varphi_0)\in
L^1(\Omega)$, where $M\in C^2(-1,1)$ is defined by $m(s)M^{\prime\prime}(s)=1
$ for all $s\in(-1,1)$ and $M(0)=M^{\prime}(0)=0$.}
Furthermore, in \cite[Proposition 4]{FGR} uniqueness of the weak solution
was proven for the convective nonlocal Cahn-Hilliard equation with
degenerate mobility for a given velocity $u\in
L^2_{loc}([0,\infty);V_{div}\cap L^\infty(\Omega)^d)$ ($d=2,3$). {\color{black}
To this purpose, the following additional conditions were assumed.

\begin{description}
\item[(A5)] There exists $\rho\in[0,1)$ such that
\begin{align*}
&\rho F_1^{\prime\prime}(s)+F_2^{\prime\prime}(s)+a(x)\geq 0,\qquad\forall
s\in(-1,1),\quad\mbox{a.e. in }\Omega.
\end{align*}

\item[(A6)] There exists $\alpha_0>0$ such that
\begin{align*}
&m(s) F_1^{\prime\prime}(s)\geq\alpha_0,\qquad\forall s\in [-1,1].
\end{align*}
\end{description}
}

By combining the proof of \cite[Proposition 4]{FGR} with the arguments of
Theorem \ref{uniqthm} we can now prove uniqueness of weak solutions for
the nonlocal Cahn-Hilliard-Navier-Stokes system with singular potential and
degenerate mobility. Indeed we have

\begin{thm}
\label{uniqthmdeg} {\color{black}Let $d=2$ and suppose that assumptions
(A1)--(A6) 
are satisfied with $\nu$ constant. Let $u_0\in G_{div}$, $\varphi_0\in
L^\infty(\Omega)$ with $F(\varphi_0)\in L^1(\Omega)$, $M(\varphi_0)\in
L^1(\Omega)$ and $h\in L^2_{loc}([0,\infty);V_{div}^{\prime})$.} Then, the
weak solution to system \eqref{sy1}-\eqref{sy6}
is unique. {\color{black} Moreover,  let $z_i:=[u_i,\varphi_i]$ be two weak solutions
corresponding to two initial data $z_{0i}:=[u_{0i},\varphi_{0i}]$
and external forces $h_i$, with $u_{0i}\in G_{div}$, $%
\varphi_{0i}\in L^\infty(\Omega)$ such that $F(\varphi_{0i})\in
L^1(\Omega)$, $M(\varphi_{0i})\in L^1(\Omega)$ and $h_i\in
L^2_{loc}([0,\infty);V^\prime_{div})$.
Then the following continuous dependence estimate holds
\begin{align}
&\Vert
u_2(t)-u_1(t)\Vert^2+\Vert\varphi_2(t)-\varphi_1(t)\Vert_{V^{\prime}}^2
\notag \\
&+\int_0^t\Big((1-\rho)\alpha_0\Vert\varphi_2(\tau)-\varphi_1(\tau)\Vert^2 +%
\frac{\nu}{2} \Vert\nabla\big(u_2(\tau)-u_1(\tau)\big)\Vert^2\Big)d\tau
\notag \\
&\leq\big(\Vert
u_2(0)-u_1(0)\Vert^2+\Vert\varphi_2(0)-\varphi_1(0)\Vert_{V^{\prime}}^2\big) %
\Lambda_0(t)+ \big|\overline{\varphi_2(0)}-\overline{\varphi_1(0)}\big|^2\Lambda_1(t)  \notag \\
&+\Vert h_2-h_1\Vert_{L^2(0,T;V_{div}^{\prime})}^2\Lambda_2(t),
\label{estcontdeg}
\end{align}
for all $t\in [0,T]$, where $\Lambda_0$, $\Lambda_1$ and $\Lambda_2$ are
continuous functions which depend on the norms of the two solutions. The
functions $Q$ and $\Lambda_i$ also depend on $F,J,\nu$ and $\Omega$%
. }
\end{thm}

\begin{proof}
Arguing as in the first part of the proof of Theorem \ref{uniqthm} we can
obtain \eqref{est4} that we now write in the following form
\begin{align}
\frac{1}{2}\frac{d}{dt}\Vert u\Vert^2+\frac{\nu}{2}\Vert\nabla u\Vert^2 \leq%
\frac{1}{4}(1-\rho)\alpha_0\Vert\varphi\Vert^2+\alpha\Vert u\Vert^2{%
\color{black}+\frac{1}{\nu}\Vert h\Vert_{V_{div}^{\prime}}^2,}  \label{nlocCH7}
\end{align}
{\color{black} where the function $\alpha$ is still given by \eqref{est4} and
we have set $\varphi:=\varphi_2-\varphi_1$, $u:=u_2-u_1$, $h:=h_2-h_1$.
}

Regarding the estimates for the difference of the nonlocal Cahn-Hilliard,
let us first recall the approach used in the proof of \cite[Proposition 4]{FGR}.

Following \cite{GL2}, we introduce
\begin{align}
&\widetilde{\Lambda}_1(s):=\int_0^s
m(\sigma)F_1^{\prime\prime}(\sigma)d\sigma,\qquad\widetilde{\Lambda}_2(s):=\int_0^s
m(\sigma)F_2^{\prime\prime}(\sigma)d\sigma,\qquad\Gamma(s):=\int_0^s
m(\sigma)d\sigma,  \notag
\end{align}
for all $s\in[-1,1]$, and see that the assumptions on $m$ and on $F$ imply
that $\widetilde{\Lambda}_1\in C^1([-1,1])$ and $0<\alpha_0\leq
\widetilde{\Lambda}_1^{\prime}(s)\leq\alpha_1$ for some positive constant $\alpha_1$.
The weak formulation of the convective nonlocal Cahn-Hilliard equation {%
\color{black} with degenerate mobility 
can then be rewritten} as follows
\begin{align}
&\langle\varphi_t,\psi\rangle+\big(\nabla\Lambda(\cdot,\varphi),\nabla\psi%
\big) -\big(\Gamma(\varphi)\nabla a,\nabla\psi\big) +\big( %
m(\varphi)(\varphi\nabla a-\nabla J\ast\varphi),\nabla\psi\big) =\big(%
u\varphi,\nabla\psi\big),  \label{nlocCH2}
\end{align}
for all $\psi\in V$, where $\Lambda(x,s):=\widetilde{\Lambda}_1(s)+\widetilde{\Lambda}_2(s)+a(x)\Gamma(s)$ for
all $s\in [0,T]$ and almost any $x\in\Omega$.

{\color{black} Consider now two weak solutions $[u_1,\varphi_1]$, $%
[u_2,\varphi_2]$
and take the difference between the two identities \eqref{nlocCH2}
corresponding to each solution. Then, choose $\psi=B_N^{-1}(\varphi-%
\overline{\varphi})$ as test function in the resulting identity.
This yields
\begin{align}
&\frac{1}{2}\frac{d}{dt}\Vert B_N^{-1/2}(\varphi-\overline{\varphi})\Vert^2 +%
\big(\Lambda(\cdot,\varphi_2)-\Lambda(\cdot,\varphi_1),\varphi\big) -\big(%
(\Gamma(\varphi_2)-\Gamma(\varphi_1))\nabla a,\nabla B_N^{-1}(\varphi-%
\overline{\varphi})\big)  \notag \\
&+\big((m(\varphi_2)-m(\varphi_1))(\varphi_2\nabla a-\nabla
J\ast\varphi_2),\nabla B_N^{-1}(\varphi-\overline{\varphi})\big)  \notag \\
&+\big(m(\varphi_1)(\varphi\nabla a -\nabla J\ast\varphi),\nabla
B_N^{-1}(\varphi-\overline{\varphi})\big)  \notag \\
&=\big(\Lambda(\cdot,\varphi_2)-\Lambda(\cdot,\varphi_1),\overline{\varphi}%
\big)+\big(u\varphi_1,\nabla B_N^{-1}(\varphi-\overline{\varphi})\big) +\big(%
u_2\varphi,\nabla B_N^{-1}(\varphi-\overline{\varphi})\big).  \label{nlocCH3}
\end{align}
Observe first that, thanks to (A5) and(A6), we have
\begin{align*}
&\partial_s\Lambda(x,s)=m(s)(F^{\prime\prime}(s)+a(x))\geq
(1-\rho)\alpha_0,\qquad\forall s\in[-1,1],\quad\mbox{a.e. }x\in\Omega,
\end{align*}
and also
\begin{align*}
&|\Lambda(x,s_2)-\Lambda(x,s_1)|\leq k|s_2-s_1|,\qquad\forall s_1,s_2\in[-1,1%
],\quad\mbox{a.e. }x\in\Omega,
\end{align*}
where $k=\Vert mF^{\prime\prime}\Vert_{C([-1,1])}+\Vert m\Vert_{C([-1,1])}\Vert
a\Vert_{L^\infty(\Omega)}$. Hence we have
\begin{align*}
&\big(\Lambda(\cdot,\varphi_2)-\Lambda(\cdot,\varphi_1),\varphi\big)%
\geq(1-\rho)\alpha_0\Vert\varphi\Vert^2,
\end{align*}
and also
\begin{align*}
&\big(\Lambda(\cdot,\varphi_2)-\Lambda(\cdot,\varphi_1),\overline{\varphi}%
\big)\leq k|\Omega|^{1/2}\Vert\varphi\Vert\overline{\varphi} \leq \frac{1}{8}%
(1-\rho)\alpha_0\Vert\varphi\Vert^2+c \overline{\varphi}^2.
\end{align*}
Concerning the third, fourth and fifth term on the left-hand side of %
\eqref{nlocCH3}, it is easy to see that they can be estimated by
\begin{align*}
&\frac{1}{8}(1-\rho)\alpha_0\Vert\varphi\Vert^2+c\Vert B_N^{-1/2}(\varphi-%
\overline{\varphi})\Vert^2.
\end{align*}
Finally,
the last two terms on the right-hand side of \eqref{nlocCH3} can be controlled
in this way
\begin{align*}
|\big(u\varphi_1,\nabla B_N^{-1}(\varphi-\overline{\varphi})\big)|&\leq\Vert
u\Vert_{L^4}\Vert\varphi_1\Vert_{L^4}\Vert\nabla B_N^{-1}(\varphi-\overline{%
\varphi})\Vert  \notag \\
&\leq\frac{\nu}{4}\Vert\nabla u\Vert^2+c\Vert\varphi_1\Vert_{L^4}^2\Vert
B_N^{-1/2}(\varphi-\overline{\varphi})\Vert^2, \\
|\big(u_2\varphi,\nabla B_N^{-1}(\varphi-\overline{\varphi})\big)|&\leq\Vert
u_2\Vert_{L^4}\Vert\varphi\Vert\Vert\nabla B_N^{-1}(\varphi-\overline{\varphi%
})\Vert_{L^4}  \notag \\
& \leq\frac{1}{8}(1-\rho)\alpha_0\Vert\varphi\Vert^2+c\Vert
u_2\Vert_{L^4}^2\Vert\nabla B_N^{-1}(\varphi-\overline{\varphi})\Vert_{L^4}^2
\notag \\
&\leq \frac{1}{8}(1-\rho)\alpha_0\Vert\varphi\Vert^2+c\Vert
u_2\Vert_{L^4}^2\Vert\nabla B_N^{-1}(\varphi-\overline{\varphi}%
)\Vert\Vert\nabla B_N^{-1}(\varphi-\overline{\varphi})\Vert_{H^1}  \notag \\
&\leq \frac{1}{8}(1-\rho)\alpha_0\Vert\varphi\Vert^2+c\Vert
u_2\Vert_{L^4}^2\Vert B_N^{-1/2}(\varphi-\overline{\varphi}%
)\Vert\Vert\varphi-\overline{\varphi}\Vert  \notag \\
& \leq\frac{1}{4}(1-\rho)\alpha_0\Vert\varphi\Vert^2+c\Vert
u_2\Vert_{L^4}^4\Vert B_N^{-1/2}(\varphi-\overline{\varphi})\Vert^2+c%
\overline{\varphi}^2.
\end{align*}
Therefore, using the above estimates,
we deduce from \eqref{nlocCH3} the following differential inequality
\begin{align}
&\frac{1}{2}\frac{d}{dt}\Vert B_N^{-1/2}(\varphi-\overline{\varphi})\Vert^2+%
\frac{3}{4}(1-\rho)\alpha_0\Vert\varphi\Vert^2 \leq \frac{\nu}{4}\Vert\nabla
u\Vert^2+\zeta\Vert B_N^{-1/2}(\varphi-\overline{\varphi})\Vert^2+c\overline{%
\varphi}^2,  \label{nlocCH6}
\end{align}
where $\zeta\in L^1(0,T)$ is given by $\zeta:=c(1+\Vert%
\varphi_1\Vert_{L^4}^2+\Vert u_2\Vert_{L^4}^4)$.
Inequalities \eqref{nlocCH7} and \eqref{nlocCH6} finally give
\begin{align}
&\frac{d}{dt}\Big(\Vert u\Vert^2+\Vert B_N^{-1/2}(\varphi-\overline{\varphi}%
)\Vert^2\Big)+ (1-\rho)\alpha_0\Vert\varphi\Vert^2+\frac{\nu}{2}\Vert\nabla
u\Vert^2  \notag \\
&\leq\theta\Big(\Vert u\Vert^2+\Vert B_N^{-1/2}(\varphi-\overline{\varphi}%
)\Vert^2\Big)+c\overline{\varphi}^2+ +\frac{2}{\nu}\Vert
h\Vert_{V_{div}^{\prime}}^2,  \label{est37}
\end{align}
where $\theta=2(\alpha+\zeta)\in L^1(0,T)$. Inequality \eqref{est37} has the same
form as \eqref{est9} without the term containing $\widetilde{\mu}$.
Therefore, arguing as in the proof of Theorem \ref{uniqthm} and using the standard Gronwall's lemma,
we find \eqref{estcontdeg}.}
\end{proof}



\section{Weak-strong uniqueness (nonconstant viscosity)}

\label{Sec4}\setcounter{equation}{0}

Here we consider system \eqref{sy1}-\eqref{sy5} in dimension two with
constant mobility, regular potential and nonconstant viscosity $
\nu=\nu(\varphi)$. In this case we are not able to prove the uniqueness of
weak solutions, due to the poor regularity of $\varphi$ which makes
difficult to estimate the difference of the dissipation term in the
Navier-Stokes equations. However, we can prove a weak-strong uniqueness
result. This means that, given a {\color{black}weak solution $[u_1,\varphi_1]$
and a strong solution $[u_2,\varphi_2]$} both corresponding to the same
initial datum {\color{black}$[u_{0},\varphi_{0}]\in G_{div}\times L^\infty(\Omega)$},
then these two solutions coincide.

Before proving this result, let us first show that a global strong solution
exists. Indeed, we observe that, while the existence of a weak solution with
nonconstant viscosity easily follows easily from the same result for the
constant viscosity case (see \cite{CFG}), this does not occur as far as
strong solutions are concerned. The difficulty essentially lies in the fact
that the classical results for the Navier-Stokes equations in two dimensions
with constant viscosity (see, e.g., \cite{T}) cannot be used as in \cite{FGK}
to exploit the improved regularity for the convective term in the nonlocal
Cahn-Hilliard equation.

{\color{black} The regularity result requires a slightly stronger assumption on the interaction kernel $J$. Thus,}
before stating the main results of this section we recall the definition of
admissible kernel (see \cite[Definition 1]{BRB}).

\begin{defn}
A kernel $J\in W^{1,1}_{loc}(\mathbb{R}^2)$ is admissible if the following
conditions are satisfied:

\begin{description}
\item[(A1)] $J\in C^3(\mathbb{R}^2 \setminus \{0\})$;

\item[(A2)] $J$ is radially symmetric, $J(x) = \tilde J(\vert x\vert)$ and $%
\tilde J$ is non-increasing;

\item[(A3)] $\tilde J^{\prime\prime}(r)$ and $\tilde J^\prime(r)/r$ are
monotone on $(0,r_0)$ for some $r_0>0$;

\item[(A4)] $\vert D^3 J(x) \vert \leq C_\sharp\vert x\vert^{-3}$ for some $C_\sharp>0$.
\end{description}
\end{defn}

We recall that the Newtonian and Bessel potentials are admissible.
Moreover, we report the following (cf. \cite[Lemma 2]{BRB}).

\begin{lem}
\label{admiss} Let $J$ be admissible and $v=\nabla J * \psi$. Then, for all $%
p\in (1,\infty)$, there exists $C_p>0$ such that
\begin{equation*}
\Vert\nabla v\Vert_p \leq C_{L^p} \Vert \psi \Vert_{L^p}.
\end{equation*}
\end{lem}

{\color{black} We also recall the following proposition for an inhomogeneous Stokes system
in non-divergence form:%
\begin{equation}
\left\{
\begin{array}{ll}
-\varpi \left( x\right) \Delta u+\nabla \pi =f\left( x\right) , & \text{in }%
\Omega , \\
\text{div}\left( u\right) =0\text{,} & \text{in }\Omega , \\
u=0, & \text{on }\partial \Omega .%
\end{array}%
\right.  \label{Stokes}
\end{equation}

\begin{prop}
\label{Stok-non}\cite[Proposition 2.1]{ZY} Let $f\in {\color{black} L^{2}\left(
\Omega \right) ^{2}}$ and $\varpi \in C^{\delta }\left( \overline{%
\Omega }\right) ,$ for some $\delta \in \left( 0,1\right) $, such that $%
0<\lambda_0 \leq \varpi \left( x\right) \leq \lambda_1 <\infty $ for all
$x\in\overline{\Omega}$. Then any solution
${\color{black}\left[ u,\pi \right]} \in H^{2}\left( \Omega \right)^2 \times
H^{1}\left( \Omega \right) $ of (\ref{Stokes}) satisfies the estimate%
\begin{equation*}
\left\Vert u\right\Vert _{H^{2}\left( \Omega \right) }+\left\Vert \pi
\right\Vert _{H^{1}\left( \Omega \right) }\leq C\left( \left\Vert
f\right\Vert _{L^{2}}+\left\Vert \pi \right\Vert _{L^{2}}\right) ,
\end{equation*}%
for some constant $C=C(\lambda_0 ,\lambda_1 ,\Omega ,\left\Vert \varpi
\right\Vert _{C^{\delta }\left( \overline{\Omega }\right) })>0.$
\end{prop}

We first show a result which generalizes \cite[Lemma 2.11]{GG4} for the
nonlocal Cahn-Hilliard equation with convection in two space dimensions.

\begin{lem}
\label{holder} Let $d=2$ and assume (H1) and (H3). Let $u\in L^{\infty}(T^\prime,T;G_{div})\cap L^{2}(T^\prime,T;V_{div}),$
for some $%
T>T^\prime\geq 0$ and let $\varphi \in L^{\infty }(T^\prime,T;L^{\infty }\left( \Omega \right) )$ be a bounded generalized (weak)\
solution of%
\begin{equation}
\left\{
\begin{array}{ll}
\partial _{t}\varphi =\text{div}\left( c\left( x,\varphi ,\nabla \varphi
\right) \right) -\text{div}\left( u\varphi \right) , & \text{in }\Omega
\times (T^\prime,T), \\
c\left( x,\varphi ,\nabla \varphi \right) \cdot n=0, & \text{on }\Gamma
\times (T^\prime,T),%
\end{array}%
\right.  \label{CH-c}
\end{equation}%
where $c\left( x,\varphi ,\nabla \varphi \right) :=(a\left( x\right)
+F^{\prime\prime}\left( \varphi \right) )\nabla \varphi +\nabla
a\varphi -\nabla J\ast \varphi $. There exist constants $C>0,$ $\alpha \in
\left( 0,1\right) ,$ depending on the $L^{\infty }(T^\prime,T;L^{\infty }\left( \Omega \right) )$-norm of $\varphi $ and $%
L^{4}(T^\prime,T; L^{4}\left( \Omega \right)^{2})$-norm
of $u$, respectively, such that%
\begin{equation}
\left\vert \varphi \left( x,t\right) -\varphi \left( y,s\right) \right\vert
\leq C(\left\vert x-y\right\vert ^{\alpha }+\left\vert t-s\right\vert
^{\alpha /2}),  \label{holder-est}
\end{equation}%
for every $\left( x,t\right) ,\left( y,s\right) \in
Q_{T^\prime,T}:=[T^\prime,T]\times \overline{\Omega }$.
\end{lem}

\begin{proof}
The proof is inspired by \cite[Theorem 3.7]{NU} (cf. also \cite[Lemma 3.2]%
{ZY}) where it was observed that a Hölder continuous estimate holds for a
similar parabolic equation with drift term $u\cdot \nabla \varphi $ whenever
the vector field $u$ is divergent free and belongs to the critical space $%
L^{4}\left(0,T; L^{4}(\Omega)\right) $. We begin by assuming that $\left\Vert
\varphi \right\Vert _{L^{\infty }(T^\prime,T;L^{\infty }\left( \Omega
\right) )}\leq R$, for some $R>0$ and observe that%
\begin{equation*}
L^{\infty }(T^\prime,T;G_{div})\cap L^{2}(T^\prime,T;V_{div})\hookrightarrow
L^{4}(T^{\prime},T;L^{4}\left( \Omega\right) ^{2}).
\end{equation*}%
Following \cite{LSU}, we let $k\in \left[ 0,R\right] $ and $%
\eta =\eta \left( x,t\right) \in \left[ 0,1\right] $ be a continuous
piecewise-smooth function which is supported on the space-time cylinders
$Q_{t_{0},t_{0}+\tau }\left( \rho \right) :=B_{\rho }\left( x_{0}\right)
\times \left( t_{0},t_{0}+\tau \right)$,
where $B_{\rho }\left( x_{0}\right) $ denotes the ball centered at $x_{0}$
of radius $\rho >0$. As usual for the interior Hölder regularity in (\ref
{holder-est}) one takes $x_{0}\in \Omega $, while $x_{0}\in \partial \Omega $
for the corresponding boundary estimate in (\ref{holder-est}) and then
exploit a standard compactness argument in which $\overline{\Omega }$ may be
covered by a finite number of such balls. We thus multiply the first
equation of (\ref{CH-c}) by $\eta ^{2}\varphi _{k}^{+}$, where $\varphi
_{k}^{+}:=\max \left\{ 0,\varphi -k\right\} ,$ integrate the resulting
identity over $Q_{t_{0},t}:=\left( t_{0},t\right) \times \Omega $, where $
T^\prime\leq t_{0}<t<t_{0}+\tau \leq T$, to deduce%
\begin{align}
& \int_{Q_{t_{0},t}}\partial _{t}\varphi \eta ^{2}\varphi
_{k}^{+}dxdt+\int_{Q_{t_{0},t}}(a\left( x\right) +F^{\prime\prime}\left( \varphi \right) )\nabla \varphi _{k}^{+}\cdot \nabla \left( \eta
^{2}\varphi _{k}^{+}\right) dxdt  \label{hold1} \\
& =\int_{Q_{t_{0},t}}u\varphi \cdot \nabla \left( \eta ^{2}\varphi
_{k}^{+}\right) dxdt+\int_{Q_{t_{0},t}}l\left( x,t\right) \cdot \nabla
\left( \eta ^{2}\varphi _{k}^{+}\right) dxdt,  \notag
\end{align}%
owing to the boundary condition of (\ref{CH-c}) and the fact that $u\in
L^{2}(T^\prime,T;V_{div})$. Here, we have set $l
=-\varphi \nabla a+\nabla J\ast \varphi $ for the sake of simplicity. Also we
notice that $\nabla \varphi _{k}^{+}\equiv \nabla \varphi $ only on the sets
where $\left\{ \varphi \left( x,t\right) >k\right\} $ while $\nabla \varphi
_{k}^{+}\equiv 0$ elsewhere. In addition, if $J\in W^{1,1}\left(
\mathbb{R}^{2}\right) $ then $l\in L^{\infty }(T^\prime,T;L^{\infty }\left( \Omega \right)^{2})$, since $\varphi $
is bounded and $a\in W^{1,\infty }\left( \Omega \right) {\color{black} \hookrightarrow} C\left(
\overline{\Omega }\right) $. From (\ref{hold1}) and assumption (H3), we
obtain%
\begin{align}
\frac{1}{2}\sup_{t\in \left( t_{0},t\right) }& \int_{\Omega }\left( \eta
\varphi _{k}^{+}\right) ^{2}\left( t\right)
dx+c_{0}\int_{Q_{t_{0},t}}\left\vert \nabla \left( \eta \varphi
_{k}^{+}\right) \right\vert ^{2}dxdt  \label{hold2} \\
& \leq \frac{1}{2}\int_{\Omega }\left( \eta \varphi _{k}^{+}\right)
^{2}\left( t_{0}\right) dx+\int_{Q_{t_{0},t}}\left( \varphi _{k}^{+}\right)
^{2}\left\vert \eta \partial _{t}\eta \right\vert dxdt  \notag \\
& +L\left( R\right) \int_{Q_{t_{0},t}}\left( \varphi _{k}^{+}\right)
^{2}\left\vert \nabla \eta \right\vert ^{2}dxdt+\int_{Q_{t_{0},t}}u\varphi
\cdot \nabla \left( \eta ^{2}\varphi _{k}^{+}\right) dxdt  \notag \\
& +\int_{Q_{t_{0},t}}l\left( x,t\right) \cdot \nabla \left( \eta ^{2}\varphi
_{k}^{+}\right) dxdt,  \notag
\end{align}%
for some function $L>0$ such that $|a\left( x\right) +F^{\prime \prime
}\left( \varphi \right) |\leq L\left( R\right) $. Indeed, we have%
\begin{equation*}
\nabla \varphi _{k}^{+}\cdot \nabla \left( \eta ^{2}\varphi _{k}^{+}\right)
=\left\vert \nabla \left( \eta \varphi _{k}^{+}\right) \right\vert
^{2}-\left\vert \nabla \eta \right\vert ^{2}\left( \varphi _{k}^{+}\right)
^{2}, \quad \text{a.e. in }Q_{t_{0},t}.
\end{equation*}%
To estimate the fourth term on the right-hand side of (\ref{hold1}) we use
the fact that $u\in L^{4}(T^\prime,T;L^{4}\left( \Omega \right)
^{2})$ is also divergent free, we argue by elementary H\"{o}lder's and
Young's inequalities as in the proof of \cite[Lemma 3.2]{ZY} to find%
\begin{align}
& \left\vert \int_{Q_{t_{0},t}}u\varphi \cdot \nabla \left( \eta ^{2}\varphi
_{k}^{+}\right) dxdt\right\vert  \label{hold3} \\
& \leq \frac{1}{4}\left\Vert \eta \varphi _{k}^{+}\right\Vert
_{L^{2}(Q_{t_{0},t})}^{2}+\frac{c_{0}}{4}\left\Vert \nabla \left( \eta
\varphi _{k}^{+}\right) \right\Vert
_{L^{2}(Q_{t_{0},t})}^{2}+C_{0}\left\Vert \nabla \eta \varphi
_{k}^{+}\right\Vert _{L^{2}(Q_{t_{0},t})}^{2},  \notag
\end{align}%
where $C_{0}>0$ depends on $c_{0}>0$ and the $L^{4}(T^\prime,T;
L^{4}\left( \Omega \right)^{2})$-norm of $u$ only. For the final
term on the right-hand side of (\ref{hold2}), we employ H\"{o}lder's and Young's
inequalities again to deduce%
\begin{align}
\left\vert \int_{Q_{t_{0},t}}l\left( x,t\right) \cdot \nabla \left( \eta
^{2}\varphi _{k}^{+}\right) dxdt\right\vert & =\left\vert
\int_{Q_{t_{0},t}}\left( l\left( x,t\right) \cdot \nabla \eta \varphi
_{k}^{+}\eta +\eta l\left( x,t\right) \cdot \nabla \left( \eta \varphi
_{k}^{+}\right) \right) dxdt\right\vert  \label{hold4} \\
& \leq C_{1}\int_{Q_{t_{0},t}}\left\vert \eta \right\vert ^{2}dxdt+\frac{1}{2%
}\int_{Q_{t_{0},t}}\left( \varphi _{k}^{+}\right) ^{2}\left\vert \nabla \eta
\right\vert ^{2}dxdt  \notag \\
& +\frac{c_{0}}{4}\int_{Q_{t_{0},t}}\left\vert \nabla \left( \eta \varphi
_{k}^{+}\right) \right\vert ^{2}dxdt,  \notag
\end{align}%
where $C_{1}>0$ depends only on $c_{0}>0$ and the
$L^{\infty }(T^\prime,T;L^{\infty }\left( \Omega \right)^{2})$-norm of $l$,
and hence on $R>0$. Inserting estimates (\ref{hold3})-(\ref%
{hold4}) into the right-hand side of (\ref{hold2}), we infer the existence
of a constant $C_{2}=C_{2}\left( C_{0},C_{1}\right) >0$ such that%
\begin{align}
\frac{1}{2}\sup_{t\in \left( t_{0},t\right) }& \int_{\Omega }\left( \eta
\varphi _{k}^{+}\right) ^{2}\left( t\right)
dx+c_{0}\int_{Q_{t_{0},t}}\left\vert \nabla \left( \eta \varphi
_{k}^{+}\right) \right\vert ^{2}dxdt  \label{hold5} \\
& \leq \frac{1}{2}\int_{\Omega }\left( \eta \varphi _{k}^{+}\right)
^{2}\left( t_{0}\right) dx  \notag \\
& +C_{2}\left( \int_{Q_{t_{0},t}}\left( \varphi _{k}^{+}\right)
^{2}\left\vert \eta \partial _{t}\eta \right\vert
dxdt+\int_{Q_{t_{0},t}}\left( \varphi _{k}^{+}\right) ^{2}\left\vert \nabla
\eta \right\vert ^{2}dxdt+\int_{Q_{t_{0},t}}\left\vert \eta \right\vert
^{2}dxdt\right) .  \notag
\end{align}%
Arguing in a similar fashion, inequality (\ref{hold5}) also holds with $%
\varphi $ replaced by $-\varphi .$ In particular, such inequalities imply
that the generalized solution $\varphi $ of (\ref{CH-c}) is an element of $%
\mathcal{B}_{2}(Q_{T^\prime,T},R,\gamma ,\omega ,0,\varkappa )$ in the
sense of \cite[Chapter II, Section 7 ]{LSU}, for some $\gamma =\gamma \left(
c_{0},R\right) $ and $\omega ,\varkappa >0$ (cf., in particular, the
inequalities in \cite[Section V, (1.12)-(1.13)]{LSU}). Therefore, on account
of \cite[Chapter V, Theorem 1.1]{LSU}, the H\"{o}lder continuity (\ref%
{holder-est}) of the solution of (\ref{CH-c}) follows in a standard way.
This ends the proof.
\end{proof}

\begin{cor}
Let $d=2$. If $\left[ u,\varphi \right]$ is any weak solution to
problem \eqref{sy1}--\eqref{sy6} in the sense of Theorem \ref{thm} then, for every $%
\tau >0,$ we have
\begin{equation*}
\left\Vert \varphi \right\Vert _{C^{\delta /2,\delta }\left( [\tau ,\infty
)\times \overline{\Omega }\right) }\leq C_{\tau },
\end{equation*}%
for some $C_{\tau }\sim \tau ^{-\gamma },$ $\gamma >0$, depending only on $
\mathcal{E}(u_{0},\varphi _{0})$ and on the other parameters of the
problem.
\end{cor}

\begin{proof}
The claim follows from the statement of Theorem \ref{thm} and the
application of Lemma \ref{holder} and \cite[Lemma 2.10]{GG4}.
\end{proof}
}

The following result on the existence of a strong solution generalizes \cite[Theorem 2]{FGK}
to the case of nonconstant viscosity.

\begin{thm}
\label{thmncvisc} {\color{black} Let $d=2$ and suppose that (H1)--(H5) are
satisfied with either $J\in W^{2,1}(B_{\delta })$ or $J$ admissible}. Assume
that $u_{0}\in V_{div}$, {\color{black}$\varphi_0\in V\cap C^\beta(\overline{\Omega})$, for
some $\beta>0$},
 and $h\in L_{loc}^{2}(\mathbb{R}^{+};G_{div})$. Then, for
every $T>0$, there exists a solution {\color{black} $[u,\varphi]$} to \eqref{sy1}--\eqref{sy6} such that
\begin{align}
& u\in L^{\infty }(0,T;V_{div})\cap L^{2}(0,T;H^{2}(\Omega )^{2}),\text{ }%
u_{t}\in L^{2}(0,T;G_{div})  \label{regps1} \\
& 
{\color{black}\varphi,\mu\in L^\infty(0,T;V)\cap L^2(0,T;H^2(\Omega))}\cap L^{\infty
}(\Omega \times (0,T)),\label{regps1bis}\\
&
\varphi_{t},\mu_t\in L^{2}(0,T;H).  \label{regps2}
\end{align}%
Furthermore, suppose in addition that $F\in C^{3}(\mathbb{R})$ and
$\varphi _{0}\in H^{2}(\Omega )$. Then, system \eqref{sy1}-\eqref{sy6} admits
a strong solution on $[0,T]$ satisfying {\color{black}\eqref{regps1}} and
\begin{align}
& \varphi{\color{black},\mu} \in L^{\infty }(0,T;H^{2}(\Omega )),  \label{regps3} \\
& \varphi _{t}{\color{black},\mu_t}\in L^{\infty }(0,T;H)\cap L^{2}(0,T;V).  \label{regps4}
\end{align}
\end{thm}

\begin{oss}
{\upshape
\label{weak-H2} Assumption (H2) in the statement of Theorem \ref%
{thmncvisc} (and subsequent Theorem \ref{weak-uniq}) can be replaced by a
more general one, i.e., it suffices to assume that $\nu $ is locally
Lipschitz on $\mathbb{R}$ and the existence of $\nu _{1}>0$ such that%
\begin{equation}
\nu (s)\geq \nu _{1},\qquad \forall\,s\in \mathbb{R}.  \label{H2bis}
\end{equation}%
Indeed, {an upper bound for }$\nu \left( \varphi \right) $ (and $%
\nu^{\prime}\left( \varphi \right) ,$ respectively) in $L^{\infty }(\Omega
\times (0,T))$ can be easily produced on account of the fact that {$\Vert
\varphi \Vert _{L^{\infty }(\Omega \times (0,T))}\leq C_{R},$ for any
\thinspace }$R>0$ such that $\left\Vert \varphi _{0}\right\Vert _{L^{\infty
}}\leq R.$ }
\end{oss}

{\color{black}
\begin{proof}
\emph{Step 1.} We first need to establish the $L^{\infty}(0,T;V)$-regularity
for $\mu $ and $\varphi $. The argument used here differs from the one
devised in \cite{FGK}. Indeed, we  cannot easily exploit the regularity $%
u\in L^{2}(0,T;H^{2})$ as it happens for the constant viscosity case.
Let us consider equation (\ref{sy1}) whose generalized
(weak) solution also satisfies (\ref{CH-c}). First we recall that $\varphi $
is bounded (see \cite[Lemma 2.10]{GG4}, cf. also \cite[Theorem 2]{FGK}) and
thus, by Lemma \ref{holder}, we infer that $\varphi \in C^{\delta /2,\delta
}\left( \left[ 0,T\right] \times \overline{\Omega }\right)$ for some  $0<
\delta \leq \min \left\{ \alpha ,\beta\right\}$.
By assumption (H2), $\nu \left( \varphi \right) \in
C^{\delta /2,\delta }\left( \left[ 0,T\right] \times \overline{\Omega }%
\right) $ since $\nu $ is a (locally) Lipschitz function on $\mathbb{R}$;
moreover, there exists a positive constant $\nu _{2}=\nu _{2}\left( R\right)
>0$ such that $\nu _{2}\geq \nu (\varphi )\geq \nu _{1}$, almost everywhere in $\left(
0,T\right) \times \Omega $, owing once again to the boundedness of $\varphi$
(cf. Remark \ref{weak-H2}). In the same fashion, we define
$b\left( x,t,\varphi \right) =a\left( x\right) +F^{\prime \prime}\left(
\varphi \right)$ and observe that it is measurable and bounded (i.e., $c_{0}\leq b\leq
b_{0}=b_{0}\left( R,\left\Vert a\right\Vert _{L^{\infty }}\right) $) for all
$\left( x,t,\varphi \right) $, in light of $a\in W^{1,\infty }\left( \Omega
\right) \hookrightarrow C\left( \overline{\Omega }\right) $ and the fact that $%
F^{\prime \prime }\left( \varphi \right) \in C\left( \left[ 0,T\right]
\times \overline{\Omega }\right)$. In fact, as a function of $\left(
x,t\right) \in Q_{0,T}$,  $b\left(\cdot,\cdot,\varphi(\cdot,\cdot) \right) $
is also continuous due to the H\"{o}lder continuity of $\varphi $.
Henceforth we shall denote by
$R$ a constant such that $\Vert \varphi \Vert _{L^{\infty
}(\Omega \times (0,T))}\leq R$.

We now test the nonlocal Cahn-Hilliard equation by $\mu _{t}=\big(%
a+F^{\prime \prime }(\varphi )\big)\varphi _{t}-J\ast \varphi _{t}$ in $H$
to deduce
\begin{align}
& \int_{\Omega }\varphi _{t}\mu _{t}+\int_{\Omega }(u\cdot \nabla \varphi
)\mu _{t}+\frac{1}{2}\frac{d}{dt}\Vert \nabla \mu \Vert ^{2}  \notag \\
& =\int_{\Omega }(a+F^{\prime \prime }(\varphi ))\varphi _{t}^{2}-(\varphi
_{t},J\ast \varphi _{t})+\int_{\Omega }(u\cdot \nabla \varphi )\mu _{t}+%
\frac{1}{2}\frac{d}{dt}\Vert \nabla \mu \Vert ^{2}=0.  \label{diffid1}
\end{align}%
This identity was considered in \cite{FGK}, but now we cannot
use the $H^{2}$-norm of $u$ to estimate the convective term
(i.e., the third term in the second line of (\ref{diffid1})).
Here we exploit the identity
\begin{equation}
u\cdot \nabla \varphi =b^{-1}u\cdot \nabla \mu +b^{-1}u\cdot \left( \nabla
J\ast \varphi -\nabla a\varphi \right) \label{idd}
\end{equation}%
and we find
\begin{align}
& \Big|\int_{\Omega }(u\cdot \nabla \varphi )\mu _{t}\Big|=\Big|\int_{\Omega
}\left( b^{-1}u\cdot \nabla \mu \right) \mu _{t}+\int_{\Omega }b^{-1}\left[
u\cdot \left( \nabla J\ast \varphi -\nabla a\varphi \right) \right] \mu _{t}%
\Big|  \label{est22} \\
& \leq c_{0}^{-1}\left( \Vert u\cdot \nabla \mu \Vert \Vert \mu _{t}\Vert
+\Vert u\cdot \left( \nabla J\ast \varphi -\nabla a\varphi \right) \Vert
\Vert \mu _{t}\Vert \right)  \notag \\
& \leq Q_{J,c_{0}}(R)\Vert \varphi _{t}\Vert \left( \Vert u\cdot
\nabla \mu \Vert +\Vert u\Vert \right)  \notag \\
& \leq \frac{c_{0}}{4}\Vert \varphi _{t}\Vert ^{2}+Q%
_{c_{0},J}(R)\left( \Vert u\Vert _{L^{4}}^{2}\Vert \nabla \mu \Vert
_{L^{4}}^{2}+\left\Vert u\right\Vert ^{2}\right)  \notag \\
& \leq \frac{c_{0}}{4}\Vert \varphi _{t}\Vert ^{2}+Q%
_{c_{0},J}(R)\Vert u\Vert \Vert \nabla u\Vert \Vert \nabla \mu \Vert \Vert
\mu \Vert _{H^{2}}+Q_{c_{0},J}(R)\left\Vert u\right\Vert ^{2}
\notag \\
& \leq \frac{c_{0}}{4}\Vert \varphi _{t}\Vert ^{2}+Q%
_{c_{0},J,\epsilon }(R)\big(\Vert u\Vert ^{2}\Vert \nabla u\Vert ^{2}\big)%
\Vert \nabla \mu \Vert ^{2}  \notag \\
& +Q_{c_{0},J}(R)\left\Vert u\right\Vert ^{2}+\epsilon \left( \Vert
B_{N}\mu \Vert ^{2}+\left\Vert \mu \right\Vert ^{2}\right) ,  \notag
\end{align}%
for any $\epsilon >0$. Furthermore, we have
\begin{align}
& |(\varphi _{t},J\ast \varphi _{t})|\leq \Vert \varphi _{t}\Vert
_{V^{\prime }}\Vert J\ast \varphi _{t}\Vert _{V}\leq \Vert \varphi _{t}\Vert
_{V^{\prime }}\Vert J\Vert _{W^{1,1}}\Vert \varphi _{t}\Vert  \notag \\
& \leq \frac{c_{0}}{4}\Vert \varphi _{t}\Vert ^{2}+c\Vert J\Vert
_{W^{1,1}}^{2}\Vert \varphi _{t}\Vert _{V^{\prime }}^{2}.  \label{est23}
\end{align}%
Inserting \eqref{est22}, \eqref{est23} into \eqref{diffid1}, and keeping $%
\epsilon >0$ arbitrary, we get the following differential inequality
\begin{align}
& \frac{d}{dt}\Vert \nabla \mu \Vert ^{2}+c_{0}\Vert \varphi _{t}\Vert ^{2}
\label{est25} \\
& \leq Q_{c_{0},J,\epsilon }(R)\big(\Vert u\Vert ^{2}\Vert \nabla
u\Vert ^{2}\big)\Vert \nabla \mu \Vert ^{2}+c\Vert J\Vert
_{W^{1,1}}^{2}\Vert \varphi _{t}\Vert _{V^{\prime }}^{2}  \notag \\
& +Q_{c_{0},J}(R)\left\Vert u\right\Vert ^{2}+\epsilon \left( \Vert
B_{N}\mu \Vert ^{2}+\left\Vert \mu \right\Vert ^{2}\right) .  \notag
\end{align}%
Moreover, observing that $\varphi_t =-B_{N}\mu -u\cdot \nabla \varphi $, we have
\begin{equation}
\left\Vert \varphi _{t}\right\Vert ^{2}\geq \frac{1}{2}\left\Vert B_{N}\mu
\right\Vert ^{2}-\left\Vert u\cdot \nabla \varphi \right\Vert ^{2},
\label{idd2}
\end{equation}%
owing to the basic inequality $\left( a-b\right) ^{2}\geq \left( 1/2\right)
a^{2}-b^{2}$. We can estimate the last term using (\ref{idd}).
Thus, recalling (\ref{est22}), we obtain
\begin{align*}
\left\Vert u\cdot \nabla \varphi \right\Vert ^{2}& \leq 2c_{0}^{-2}\left(
\Vert u\cdot \nabla \mu \Vert ^{2}+\Vert u\cdot \left( \nabla J\ast \varphi
-\nabla a\varphi \right) \Vert ^{2}\right) \\
& \leq Q_{c_{0},J,\epsilon }(R)\big(\Vert u\Vert ^{2}\Vert \nabla
u\Vert ^{2}\big)\Vert \nabla \mu \Vert ^{2}+Q_{c_{0},J}(R)\left%
\Vert u\right\Vert ^{2} \\
& +\epsilon \left( \Vert B_{N}\mu \Vert ^{2}+\left\Vert \mu \right\Vert
^{2}\right) .
\end{align*}%
Thus, from (\ref{est25}) by virtue of (\ref{idd2}) we further derive%
\begin{align}
& \frac{d}{dt}\Vert \nabla \mu \Vert ^{2}+\frac{c_{0}}{2}\left( \Vert
\varphi _{t}\Vert ^{2}+\frac{1}{2}\left\Vert B_{N}\mu \right\Vert ^{2}\right)
\label{idd3} \\
& \leq Q_{c_{0},J,\epsilon }(R)\big(\Vert u\Vert ^{2}\Vert \nabla
u\Vert ^{2}\big)\Vert \nabla \mu \Vert ^{2}+c\Vert J\Vert
_{W^{1,1}}^{2}\Vert \varphi _{t}\Vert _{V^{\prime }}^{2}  \notag \\
& +Q_{c_{0},J}(R)\left\Vert u\right\Vert ^{2}+2\epsilon \left(
\Vert B_{N}\mu \Vert _{H^{2}}^{2}+\left\Vert \mu \right\Vert ^{2}\right) ,
\notag
\end{align}%
for any $\epsilon >0$. Let us now choose a sufficiently small $\epsilon \leq
c_{0}/8$ in order to absorb the $L^{2}$-norm of $B_{N}\mu $ into the
left-hand side and observe that $\mu \in L^{\infty }\left( \Omega\times (0,T)\right)$
since $\varphi $ is bounded. Thus, we find
\begin{align}
\varphi & \in L^{\infty }(0,T;V),\text{ }\varphi _{t}\in L^{2}(0,T;H),\qquad
\label{est25bis} \\
\mu & \in L^{\infty }(0,T;V)\cap L^{2}\left( 0,T;H^{2}\left( \Omega \right)
\right), \notag
\end{align}%
by means of Gronwall's inequality (cf. also Lemma \ref{admiss}), using the
initial condition $\varphi _{0}\in V\cap L^{\infty }\left( \Omega \right) $
(which implies $\mu _{0}\in V$), the regularity properties of the weak
solution given by the first of \eqref{regpw1} and by \eqref{regpw4}, and the
fact that%
\begin{equation*}
c_{0}\Vert \nabla \varphi \Vert ^{2}-Q(R)\leq \Vert \nabla \mu
\Vert ^{2}\leq Q(R)\big(\Vert \nabla \varphi \Vert ^{2}+1\big).
\end{equation*}%
We now control $\nabla \varphi $ in terms of $\nabla \mu $ in $L^{p}$.
In order to do that we take the gradient of $\mu =a\varphi -J\ast \varphi +F^{\prime }(\varphi
)$, multiply it by $\nabla \varphi |\nabla \varphi |^{p-2}$ and integrate
the resulting identity on $\Omega $. This gives
\begin{equation*}
\int_{\Omega }\nabla \varphi |\nabla \varphi |^{p-2}\cdot \nabla \mu
=\int_{\Omega }(a+F^{\prime \prime }(\varphi ))|\nabla \varphi
|^{p}+\int_{\Omega }(\varphi \nabla a-\nabla J\ast \varphi )\cdot \nabla
\varphi |\nabla \varphi |^{p-2}.
\end{equation*}%
So that, by (H3), we find
\begin{align*}
c_{0}\Vert \nabla \varphi \Vert _{L^{p}}^{p}& \leq \Vert \nabla \varphi
\Vert _{L^{p}}^{p-1}\Vert \nabla \mu \Vert _{L^{p}}+(\Vert \nabla a\Vert
_{L^{\infty }}+\Vert \nabla J\Vert _{L^{1}})\Vert \varphi \Vert
_{L^{p}}\Vert \nabla \varphi \Vert _{L^{p}}^{p-1} \\
& \leq \frac{c_{0}}{2}\Vert \nabla \varphi \Vert _{L^{p}}^{p}+c\Vert \nabla
\mu \Vert _{L^{p}}^{p}+Q(R)(\Vert \nabla a\Vert _{L^{\infty
}}+\Vert \nabla J\Vert _{L^{1}})^{p},
\end{align*}%
which yields
\begin{equation}
\Vert \nabla \varphi \Vert _{L^{p}}\leq c\Vert \nabla \mu \Vert _{L^{p}}+%
Q(R).\label{est38}
\end{equation}%
This estimate implies in particular
\begin{equation}
\varphi \in L^{4}\left( 0,T;W^{1,4}\left( \Omega \right) \right) ,
\label{est28bb}
\end{equation}%
owing to the second of (\ref{est25bis}). We now control the $H^{2}$-norm of $%
\varphi $ (or at least the $L^{2}$-norm of the second derivatives $\partial
_{ij}^{2}\varphi :=\frac{\partial ^{2}\varphi }{\partial x_{i}\partial x_{j}}
$) in terms of the $H^{2}$-norm of $\mu $ and (\ref{est28bb}). To this aim
apply the second derivative operator $\partial _{ij}^{2}$ to \eqref{sy2},
multiply the resulting identity by $\partial _{ij}^{2}\varphi $ and
integrate on $\Omega $. 
This entails
\begin{align}
& \int_{\Omega }\partial _{ij}^{2}\mu \partial _{ij}^{2}\varphi
=\int_{\Omega }(a+F^{\prime \prime }(\varphi ))(\partial _{ij}^{2}\varphi
)^{2}+\int_{\Omega }(\partial _{i}a\partial _{j}\varphi +\partial
_{j}a\partial _{i}\varphi )\partial _{ij}^{2}\varphi  \notag \\
& +\int_{\Omega }(\varphi \partial _{ij}^{2}a-\partial _{i}(\partial
_{j}J\ast \varphi )\partial _{ij}^{2}\varphi +\int_{\Omega }F^{\prime \prime
\prime }(\varphi )\partial _{i}\varphi \partial _{j}\varphi \partial
_{ij}^{2}\varphi ,\qquad i,j=1,2.  \notag
\end{align}%
From this identity, thanks to (H3), we obtain
\begin{align}
& c_{0}\Vert \partial _{ij}^{2}\varphi \Vert ^{2}\leq c\Vert \partial
_{ij}^{2}\mu \Vert ^{2}  \label{est21t} \\
& +c\big(\Vert \nabla a\Vert _{L^{\infty }}^{2}+Q(R)\big)\Vert
\nabla \varphi \Vert ^{2}+Q(R)\Vert \partial _{ij}^{2}a\Vert ^{2}
\notag \\
& +\Vert \partial _{i}(\partial _{j}J\ast \varphi )\Vert ^{2}+Q%
\left( R\right) \left\Vert \nabla \varphi \right\Vert _{L^{4}}^{4},  \notag
\end{align}%
and an estimate like this still holds if $\Vert \partial _{ij}^{2}\varphi
\Vert $ and $\Vert \partial _{ij}^{2}\mu \Vert $ are replaced by $\Vert
\varphi \Vert _{H^{2}}$ and $\Vert \mu \Vert _{H^{2}}$, respectively.
Thus, recalling (\ref{est25bis}), (\ref{est28bb}),  and using the fact that
$J\in W^{2,1}(B_{\delta })$ or $J$ is admissible, from (\ref{est21t}) we easily get
\begin{equation}
\varphi \in L^{2}\left( 0,T;H^{2}\left( \Omega \right) \right).
\label{est21b}
\end{equation}%

\emph{Step 2}. We now establish the $L^{\infty }(0,T;V_{div})\cap
L^{2}\left(0,T;H^{2}\left( \Omega \right)^2\right)$-regularity for $u$. To
this end, let us test the Navier-Stokes equations by $u_{t}$ in $G_{div}$
to deduce the identity
\begin{equation}
\Vert u_{t}\Vert ^{2}+2\int_{\Omega }\nu (\varphi )\left( Du:Du_{t}\right)
dx+b(u,u,u_{t})=(l,u_{t}),  \label{est14}
\end{equation}%
where the function $l$ is given by
\begin{equation*}
l:=-\frac{\varphi ^{2}}{2}\nabla a-(J\ast \varphi )\nabla \varphi +h.
\end{equation*}%
Notice that, due to the assumption on the external force $h$ and to the
regularity of $\varphi $, we have $%
l\in L^{2}(0,T;L^2(\Omega)^2)$. From \eqref{est14} we obtain
\begin{equation}
\frac{1}{2}\Vert u_{t}\Vert ^{2}+\frac{d}{dt}\int_{\Omega }\nu (\varphi
)|Du|^{2}+b(u,u,u_{t})\leq \frac{1}{2}\Vert l\Vert ^{2}+\int_{\Omega
}|Du|^{2}\nu ^{\prime }(\varphi )\varphi _{t}.  \label{est15}
\end{equation}%
Observe that
\begin{align}
& \Big|\int_{\Omega }|Du|^{2}\nu ^{\prime }(\varphi )\varphi _{t}\Big|\leq
\Vert \nu ^{\prime }(\varphi )\Vert _{L^{\infty }}\Vert \varphi _{t}\Vert
\Vert Du\Vert _{L^{4}}^{2}  \notag \\
& \leq Q(R)\Vert \varphi _{t}\Vert \Vert Du\Vert \Vert u\Vert
_{H^{2}}  \notag \\
& \leq \delta \Vert u\Vert _{H^{2}}^{2}+Q_{\delta }(R)\Vert Du\Vert
^{2}\Vert \varphi _{t}\Vert ^{2}.  \label{est16}
\end{align}%
Furthermore, we have
\begin{align}
& |b(u,u,u_{t})|\leq \frac{1}{4}\Vert u_{t}\Vert ^{2}+\Vert u\cdot \nabla
u\Vert ^{2}  \notag \\
& \leq \frac{1}{4}\Vert u_{t}\Vert ^{2}+2\Vert u\Vert _{L^{4}}^{2}\Vert
\nabla u\Vert _{L^{4}}^{2}  \notag \\
& \leq \frac{1}{4}\Vert u_{t}\Vert ^{2}+c\Vert u\Vert \Vert \nabla u\Vert
\Vert \nabla u\Vert \Vert u\Vert _{H^{2}}  \notag \\
& \leq \frac{1}{4}\Vert u_{t}\Vert ^{2}+\delta \Vert u\Vert
_{H^{2}}^{2}+c_{\delta }\big(\Vert u\Vert ^{2}\Vert \nabla u\Vert ^{2}\big)%
\Vert \nabla u\Vert ^{2}.  \label{est17}
\end{align}%
Plugging \eqref{est16} and \eqref{est17} into \eqref{est15}, we get
\begin{align}
& \frac{1}{4}\Vert u_{t}\Vert ^{2}+\frac{d}{dt}\int_{\Omega }\nu (\varphi
)|Du|^{2}  \notag \\
& \leq \frac{1}{2}\Vert l\Vert ^{2}+2\delta \Vert u\Vert
_{H^{2}}^{2}+c_{\delta }\big(\Vert u\Vert ^{2}\Vert \nabla u\Vert ^{2}\big)%
\Vert Du\Vert ^{2}  \notag \\
& +Q_{\delta }(R)\Vert Du\Vert ^{2}\Vert \varphi _{t}\Vert ^{2},
\label{est31}
\end{align}%
for any $\delta >0$ that will be fixed later.

It remains to absorb the term $2\delta \Vert u\Vert _{H^{2}}^{2}$ into
the left-hand side of inequality (\ref{est31}). This can be done essentially
by controlling it with $2\delta \left\Vert u_{t}\right\Vert ^{2}$
plus some lower-order (bounded) perturbation. To achieve this we first
rewrite the Navier-Stokes equations as an inhomogeneous elliptic system in divergence form,
namely,
\begin{equation}
\left\{
\begin{array}{ll}
-\text{div}(2\nu (\varphi )Du)+\nabla \pi =\widetilde{h}, & \text{ in }\Omega
\times \left( 0,T\right) , \\
\text{div}\left( u\right) =0\text{,} & \text{ in }\Omega \times \left(
0,T\right) , \\
u=0\text{,} & \text{ on }\partial \Omega \times \left( 0,T\right) ,%
\end{array}%
\right.   \label{St-ee1}
\end{equation}%
where%
\begin{equation}
\label{tildeh}
\widetilde{h}:=\mu \nabla \varphi +h(t)- (u\cdot \nabla )u-u_{t}.
\end{equation}%
Since $\varphi $ is bounded on $\Omega \times \left( 0,T\right) $ (and
therefore, $\nu \left( \varphi \right) $ is bounded by (H3)), by the
application of Lax-Milgram lemma, we can infer that every solution $\left[
u,\pi \right] \in V_{div}\times L^{2}\left( \Omega \right)$ to
(\ref{St-ee1}) such that $\overline{\pi }=0$  satisfies the bound
\begin{equation}
\left\Vert Du\right\Vert +\left\Vert \pi \right\Vert \leq C||\widetilde{h}%
||_{V^{\prime }},  \label{est18b}
\end{equation}%
for some $C>0$ which depends on $\Omega $ and $R>0$ only. On the other hand,
we can also rewrite (\ref{St-ee1}) as an inhomogeneous elliptic system in
non-divergence form, that is,
\begin{equation}
\left\{
\begin{array}{ll}
-\nu (\varphi )\Delta u+\nabla \pi =\widehat{h}, & \text{ in }\Omega \times
\left( 0,T\right) , \\
\text{div}\left( u\right) =0\text{,} & \text{ in }\Omega \times \left(
0,T\right) , \\
u=0\text{,} & \text{ on }\partial \Omega \times \left( 0,T\right) ,%
\end{array}%
\right.   \label{St-ee2}
\end{equation}%
where
\begin{equation*}
\widehat{h}:=\widetilde{h}+2\nu^{\prime }\left( \varphi \right) \nabla
\varphi \cdot Du.
\end{equation*}%
We can then apply Proposition \ref{Stok-non} to (\ref{St-ee2}) since
$\nu \left( \varphi \right) \in C^{\delta /2,\delta }\left( %
\left[ 0,T\right] \times \overline{\Omega }\right) $. Thus we obtain the bound
(cf. also (\ref{est18b}))
\begin{align}
\left\Vert u\right\Vert _{H^{2}}+\left\Vert \pi \right\Vert _{H^{1}}& \leq
C\left( ||\widehat{h}||+\left\Vert \pi \right\Vert \right) \leq C\left( ||%
\widehat{h}||+||\widetilde{h}||_{V^{^{\prime }}}\right)   \label{St-ee3} \\
& \leq C\left( ||\widetilde{h}||+\left\Vert \nabla \varphi \cdot
Du\right\Vert \right) ,  \notag
\end{align}%
where  $C=C\left( \nu _{1},\nu _{2},R,T,\Omega \right) >0$.
Recalling \eqref{tildeh}, we deduce
\begin{align}
\left\Vert u\right\Vert _{H^{2}}& \leq C\left\Vert u_{t}\right\Vert +C\left(
\left\Vert h\right\Vert +\left\Vert u\cdot \nabla u\right\Vert +\left\Vert
\mu \nabla \varphi \right\Vert +\left\Vert \nabla \varphi \cdot
Du\right\Vert \right)   \label{eest-18t} \\
& \leq C\left\Vert u_{t}\right\Vert +C\left( \left\Vert h\right\Vert
+\left\Vert u\right\Vert _{L^{4}}\left\Vert \nabla u\right\Vert
_{L^{4}}+\left\Vert \mu \right\Vert _{L^{\infty }}\left\Vert \nabla \varphi
\right\Vert +\left\Vert \nabla \varphi \right\Vert _{L^{4}}\left\Vert
Du\right\Vert _{L^{4}}\right)   \notag \\
& \leq C\left\Vert u_{t}\right\Vert +C\left( \left\Vert h\right\Vert
+\left\Vert u\right\Vert ^{1/2}\left\Vert Du\right\Vert \left\Vert
u\right\Vert _{H^{2}}^{1/2}\right)   \notag \\
& +C\left( \left\Vert \mu \right\Vert _{L^{\infty }}\left\Vert \nabla
\varphi \right\Vert +\left\Vert \nabla \varphi \right\Vert
_{L^{4}}\left\Vert Du\right\Vert ^{1/2}\left\Vert u\right\Vert
_{H^{2}}^{1/2}\right)   \notag \\
& \leq C\left\Vert u_{t}\right\Vert +C_{\epsilon }\left( \left\Vert
h\right\Vert +\left( \left\Vert u\right\Vert \left\Vert \nabla u\right\Vert
\right) \left\Vert Du\right\Vert \right)   \notag \\
& +C_{\epsilon }\left( \left\Vert \mu \right\Vert _{L^{\infty }}\left\Vert
\nabla \varphi \right\Vert +\left\Vert \nabla \varphi \right\Vert
_{L^{4}}^{2}\right) \left\Vert Du\right\Vert +2\epsilon \left\Vert
u\right\Vert _{H^{2}},  \notag
\end{align}%
for any $\epsilon >0$. Thus, for $\epsilon \in(0,\frac{1}{2})$ we can absorb the small
term on the left-hand side and infer
\begin{align}
\left\Vert u\right\Vert _{H^{2}}^{2}& \leq C\left\Vert u_{t}\right\Vert
^{2}+C\left( \left\Vert h\right\Vert ^{2}+\left\Vert \mu \right\Vert
_{L^{\infty }}^{2}\left\Vert \nabla \varphi \right\Vert ^{2}\right)
\label{est19} \\
& +C\left( \left\Vert u\right\Vert ^{2}\left\Vert \nabla u\right\Vert
^{2}+\left\Vert \nabla \varphi \right\Vert _{L^{4}}^{4}\right) \left\Vert
Du\right\Vert ^{2}.  \notag
\end{align}%
We can now insert the bound (\ref{est19}) into (\ref{est31}), take $\delta >0
$ small enough and obtain the differential inequality
\begin{align}
& \frac{d}{dt}\int_{\Omega }\nu (\varphi )|Du|^{2}+\frac{1}{8}\Vert
u_{t}\Vert ^{2}  \label{est32b} \\
& \leq C\left( \Vert l\Vert ^{2}+\left\Vert h\right\Vert ^{2}+\left\Vert \mu
\right\Vert _{L^{\infty }}^{2}\left\Vert \nabla \varphi \right\Vert
^{2}\right)   \notag \\
& +C\left( R\right) (\Vert u\Vert ^{2}\Vert \nabla u\Vert ^{2}+\left\Vert
\nabla \varphi \right\Vert _{L^{4}}^{4}+\left\Vert \varphi _{t}\right\Vert
^{2})\Vert Du\Vert ^{2}.  \notag
\end{align}%
From \eqref{est32b}, on account of (H2) and of the improved regularity for
{\color{black}$\left[ \varphi ,\mu \right] $} given by (\ref{est25bis}) and (\ref{est28bb}),
by means of Gronwall's inequality, we obtain
\begin{equation}
u\in L^{\infty }(0,T;V_{div})\cap L^{2}(0,T;H^{2}(\Omega )^{2}),\text{ }%
u_{t}\in L^{2}(0,T;G_{div}).  \label{est34}
\end{equation}%
Moreover, owing to (\ref{St-ee3}), we have $\pi \in L^{2}\left(
0,T;H^{1}\left( \Omega \right) \right) $. With these regularity properties
for $u$ at disposal we can now argue exactly as in the second step of the
proof of \cite[Theorem 2]{FGK} by differentiating \eqref{sy1} with respect
to time, multiplying the resulting identity by $\mu _{t}$ in $H$ and using
the assumptions $F\in C^{3}(\mathbb{R})$ and $\varphi _{0}\in
H^{2}(\Omega )$ (this last assumption
ensures that $\varphi_t(0)\in H$, see Lemma \ref{admiss})
to deduce
\begin{equation}
\varphi _{t}\in L^{\infty }(0,T;H)\cap L^{2}(0,T,V).  \label{est27bb}
\end{equation}%
Furthermore, using \eqref{sy1}, we find
\begin{align}
& \Vert \nabla \mu \Vert _{L^{p}}\leq c\Vert \nabla \mu \Vert ^{2/p}\Vert
\nabla \mu \Vert _{H^{1}}^{1-2/p}  \label{est28b} \\
& \leq c\Vert \nabla \mu \Vert ^{2/p}\Vert \mu \Vert _{H^{2}}^{1-2/p}  \notag
\\
& \leq c\Vert \nabla \mu \Vert ^{2/p}(\Vert \Delta \mu \Vert ^{1-2/p}+\Vert
\mu \Vert ^{1-2/p})  \notag \\
& \leq Q(R,\Vert \varphi _{0}\Vert _{V},\Vert u_{0}\Vert )\big(%
\Vert \varphi _{t}\Vert ^{1-2/p}+\Vert u\cdot \nabla \varphi \Vert ^{1-2/p}+1%
\big)  \notag \\
& \leq Q(R,\Vert \varphi _{0}\Vert _{V},\Vert u_{0}\Vert )\big(%
\Vert \varphi _{t}\Vert ^{1-2/p}+\Vert u\Vert _{L^{q}}^{1-2/p}\Vert \nabla
\varphi \Vert _{L^{p}}^{1-2/p}+1\big).  \notag
\end{align}%
Here we have used the fact that the $H^{2}-$norm of $\mu $ is equivalent to
the $L^{2}-$ norm of $\left( B_{N}+I\right) \mu $ (cf. \eqref{sy5}) and we
have taken into account the improved regularity for $\mu $ given by the
third of \eqref{est25bis}. By combining \eqref{est25bis} with \eqref{est28b}
we therefore get
\begin{align}
& \Vert \nabla \varphi \Vert _{L^{p}}\leq Q(R,\Vert \varphi
_{0}\Vert _{V},\Vert u_{0}\Vert )\big(\Vert \varphi _{t}\Vert ^{1-2/p}+\Vert
u\Vert _{L^{2p/(p-2)}}^{(p-2)/2}+1\big)  \label{est27} \\
& \leq Q(R,\Vert \varphi _{0}\Vert _{V},\Vert u_{0}\Vert )\big(%
\Vert \varphi _{t}\Vert ^{1-2/p}+\Vert u\Vert ^{(p-2)^{2}/2p}\Vert \nabla
u\Vert ^{1-2/p}+1\big).  \notag
\end{align}%
Thanks to this property, on account  of \eqref{est34}$_1$, (\ref{est27bb}) and
(\ref{est28b})-(\ref{est27}), we have
\begin{equation}
\varphi \in L^{\infty }(0,T;W^{1,p}(\Omega )).  \label{est29}
\end{equation}%
Finally, by comparison in \eqref{sy1} (cf. \cite{FGK}) we also get
$\mu \in L^{\infty }(0,T;H^{2}(\Omega))$.
This fact, thanks to \eqref{est21t} and using once more the regularity
assumption on $J$, implies
\begin{equation}
\varphi \in L^{\infty }\left( 0,T;H^{2}(\Omega )\right) .  \label{est29b}
\end{equation}%
\emph{Step 3}. We shall briefly explain the details of the approximation
schemes which can be used to derive the estimates in Steps 1 and 2. Regarding
estimates (\ref{est25bis}), (\ref{est28bb}), (\ref{est21b}),
it suffices to employ the usual
Faedo-Galerkin truncation method as in \cite[Theorem 1]{CFG} since $u\in
L^{\infty }\left( 0,T;G_{div}\right) \cap L^{2}\left( 0,T;V_{div}\right) $
is a weak solution to (\ref{sy3})-(\ref{sy4}). Weak solutions are also enough
to deduce \eqref{holder-est}.
To deduce the higher-order estimate for $u\in L^{\infty }\left(
0,T;V_{div}\right) $ in Step 2, we can no longer exploit the usual Galerkin
scheme in a standard fashion but we need to rely on a different scheme.
We first mollify the Navier-Stokes equation in the
following fashion: recall that $\varphi \in L^{\infty }\left( 0,T;V\right)
\cap L^{2}\left( 0,T;H^{2}\left( \Omega \right) \right) $ is such that $\varphi
\in C^{\delta /2}\left( \left[ 0,T\right] ;C^{\delta }\left( \overline{%
\Omega }\right) \right) $ as provided by the Step 1 and that $\partial
\Omega$ is of class $\mathcal{C}^{2}$. Let $\widetilde{\varphi }%
=E\varphi ,$ where $E:W^{2,p}\left( \Omega \right) \rightarrow W^{2,p}\left(
\mathbb{R}^{2}\right) $ is an extension operator for any $p\in \lbrack
1,\infty )$. Then set $\widetilde{\varphi }%
_{\varepsilon }=\eta _{\varepsilon }\ast \widetilde{\varphi }$ where
$\eta _{\varepsilon }\in C^{\infty }\left( \mathbb{%
R}^{2}\right) $ is the usual Friedrich mollifier such that
$\eta _{\varepsilon }\geq 0$ and $\int_{\mathbb{R}%
^{2}}\eta _{\varepsilon }dx=1$. Defining $\varphi _{\epsilon }=R\widetilde{%
\varphi }_{\varepsilon }$, where $R:W^{2,p}\left( \mathbb{R}^{2}\right)
\rightarrow W^{2,p}\left( \Omega \right) $ is the restriction
operator, it is clear that $\widetilde{\varphi }_{\varepsilon }\left(
x,\cdot \right) $ is of class $C^{\infty }$ in a neighborhood of $\overline{%
\Omega }$. Moreover $\varphi _{\epsilon }$ satisfies,
for any $k\in \left\{ 0,1,2\right\} $ and $p\in \lbrack 1,\infty )$, the bounds
\begin{equation*}
\left\Vert \varphi _{\varepsilon }\left( t\right) \right\Vert _{W^{k,p}}\leq
\left\Vert \varphi \left( t\right) \right\Vert _{W^{k,p}},\left\Vert \varphi
_{\varepsilon }\left( t\right) \right\Vert _{W^{k+1,p}}\leq
C_{k,p,\varepsilon }\left\Vert \varphi \left( t\right) \right\Vert _{W^{k,p}}
\end{equation*}%
and $\varphi _{\varepsilon }\left( t\right) \rightarrow \varphi \left(
t\right) $ strongly in $W^{k,p}\left( \Omega \right) $ for almost any $t\in
\left( 0,T\right)$ (see, e.g., \cite[Chapter V]{EE}).
We also have
\begin{equation}
\varphi _{\varepsilon }\in L^{\infty }\left( 0,T;H^{2}\left( \Omega \right)
\right) \cap C^{\delta /2}\left( \left[ 0,T\right] ;C^{\delta }\left(
\overline{\Omega }\right) \right).  \label{reg-order}
\end{equation}
We now consider the following
mollified version of the original Navier-Stokes equations
\begin{align}
&u_{t}-2\mbox{div}(\nu (\varphi _{\varepsilon })Du)+(u\cdot \nabla )u+\nabla
\pi =\mu \nabla \varphi +h(t),  \label{molliNS}\\
&\mbox{div}\left( u_{\varepsilon }\right) =0 \label{freediv}
\end{align}
in $\Omega\times(0,T)$ with initial condition $u_{\varepsilon \mid t=0}=u_{0}$ and no-slip
boundary condition. Here $\mu$ and $\varphi$ are as regular as specified in Step 1. Let us observe that
(\ref{reg-order}) together with standard interpolation results in Sobolev spaces
imply that $\varphi _{\varepsilon }\in BUC\left( \left[ 0,T\right]
;W^{1,q}\left( \Omega \right) \right) $ for any $q>2$ (i.e., $\varphi
_{\varepsilon }$ is bounded and uniformly continuous with values in $W^{1,q}(\Omega)$
with $\left\Vert \varphi _{\varepsilon }\right\Vert _{BUC(0,T;W^{1,q})}\leq
C_{\varepsilon },$ for some $C_{\varepsilon }\rightarrow \infty $ as $%
\varepsilon \rightarrow 0^{+}$). Thus, thanks to a result contained in the proof of
\cite[Theorem 8]{A1}, we can find a sufficiently small time $T_{\varepsilon }\leq T$,
a function $u_\varepsilon$ such that
\begin{equation}
u_{\varepsilon }\in H^{1}\left( 0,T_{\varepsilon };G_{div}\right) \cap
L^{2}(0,T_{\varepsilon };H^{2}\left( \Omega \right) ^{2})\cap L^{\infty
}\left( 0,T_{\varepsilon };V_{div}\right)   \label{reg-vel}
\end{equation}%
and the associated pressure $\pi _{\varepsilon }\in L^{2}\left(
0,T_{\varepsilon };H^{1}\left( \Omega \right) /\mathbb{R}\right) $ such that $u_\varepsilon$ is a
strong solution to (\ref{molliNS})-\eqref{freediv}, provided that $u_{0}\in V_{div}$ and $%
h\in L_{\text{loc}}^{2}(\mathbb{R}_{+};G_{div})$ and $\mu \nabla \varphi \in
L_{\text{loc}}^{2}(\mathbb{R}_{+};L^{2}\left( \Omega \right) ^{2})$ (for the
latter see Step 1). The regularity (\ref{reg-vel}) is enough to perform
all the estimates of Step 2 on the fluid velocity rigorously. In particular,
estimates \eqref{est19}-\eqref{est32b} entail that $u_{\varepsilon }$ can be extended to any interval $%
\left( 0,T\right) $, for any given $T>0$. Moreover, $u_\varepsilon$ is bounded in the spaces (\ref{reg-vel})
uniformly with respect to $\varepsilon$ (and $\pi _{\varepsilon }$ is bounded in $L^{2}\left( 0,T;H^{1}\left( \Omega
\right) \right) $ uniformly with respect to $\varepsilon$). Thus, usual
compactness arguments allows to pass to the limit as $\varepsilon \rightarrow
0$ in (\ref{molliNS})-\eqref{freediv}, owing to the strong convergence $\varphi _{\varepsilon
}\left( t\right) \rightarrow \varphi \left( t\right) $ in $V$ for almost any
$t\in \left( 0,T\right)$. This gives a strong solution $\tilde u$ to the same problem solved by
the weak solution found in Step 1. Then uniqueness applied to the NS equations with given viscosity
implies that $u=\tilde u$. We can now perform estimates (\ref{est27bb})-(\ref{est27})
to show that $\varphi$ satisfies (\ref{est29}) and (\ref{est29b}).
This ends the proof.
\end{proof}
}

{\color{black}
\begin{oss}{\upshape
Assuming that
 $u_{0}\in V_{div}$, $\varphi_0\in V\cap C^\beta(\overline{\Omega})$, for
some $\beta>0$,
and $h\in L_{loc}^{2}(\mathbb{R}^{+};G_{div})$, we can see that the strong
solution $[u,\varphi]$ of Theorem \ref{thmncvisc} also satisfies the following
strong time continuity properties
\begin{align}
& u\in C([0,T];V_{div}),\qquad \varphi, \mu\in C([0,T];V).\label{str-time-cont}
\end{align}
Indeed, in order to prove \eqref{str-time-cont}$_1$ we first observe that,
as a consequence of \eqref{regps1}, we have $u\in C_w([0,T];V_{div})$.
We recall that $C_w([0,T];X)$ stands for the space of weakly continuous
functions from $[0,T]$ with values in a Banach space $X$.
Moreover, multiplying \eqref{sy3} by
$-\Delta u$ in $L^2(\Omega)^2$ we get
\begin{align}
\frac{1}{2}\frac{d}{dt}\Vert\nabla u\Vert^2&=-\big(\nu(\varphi)\Delta u,\Delta u\big)
-2\big(\nu'(\varphi)\nabla\varphi\cdot Du,\Delta u\big)+\big((u\cdot\nabla)u,\Delta u\big)
\nonumber\\
&+(\nabla\pi,\Delta u)-(\mu\nabla\varphi,\Delta u)-(h,\Delta u).\nonumber
\end{align}
On account of \eqref{regps1} and \eqref{regps1bis}, we can see that all the terms
on the right-hand side of this differential identity belong to $L^1(0,T)$.
Indeed, recall that $\pi\in L^2(0,T;H^1(\Omega))$ and observe that
\begin{align}
&2\big|\big(\nu'(\varphi)\nabla\varphi\cdot Du,\Delta u\big)\big|\leq C_\nu(R)\Vert\nabla\varphi\Vert_{L^4}\Vert Du\Vert_{L^4}\Vert\Delta u\Vert
\nonumber\\
&\leq C_\nu(R)\Vert\nabla\varphi\Vert^{1/2}\Vert\varphi\Vert_{H^2}^{1/2}\Vert Du\Vert^{1/2}\Vert u\Vert_{H^2}^{1/2}\Vert\Delta u\Vert,
\nonumber
\end{align}
where $R>$ is such that $\Vert \varphi \Vert _{L^{\infty
}(\Omega \times (0,T))}\leq R$.
We therefore deduce the absolute continuity of $\Vert\nabla u(\cdot)\Vert^2$ on $[0,T]$.
This, together with the weak continuity of $u$ in $V_{div}$, yields \eqref{str-time-cont}$_1$.
As far as \eqref{str-time-cont}$_2$ is concerned, from
the differential identity \eqref{diffid1}, recalling \eqref{regps1} and \eqref{regps1bis},
we infer the absolute continuity of $\Vert\nabla\mu(\cdot)\Vert^2$ on $[0,T]$.
Since we also have $\mu\in C_w([0,T];V)$, then we immediately get
$\mu\in C([0,T];V)$. Consider now the identity
$\nabla\varphi=b^{-1}\nabla\mu+b^{-1}(\nabla J\ast\varphi-\varphi\nabla a)$,
where $b=a(x)+F''(\varphi)$. Observe that $b^{-1}\in C([0,T];C(\overline{\Omega}))$
owing to $\varphi \in C^{\delta /2,\delta}\left( \left[ 0,T\right] \times \overline{\Omega }\right)$.
Thus we conclude that $\varphi\in C([0,T];V)$.
If, in addition, $F\in C^3(\mathbb{R})$ and
$\varphi_0\in H^2(\Omega)$ then, by arguing exactly as in \cite[Remark 5]{FGK},
we can also prove the following properties
\begin{align}
&\varphi,\mu\in C([0,T];H^2(\Omega)),\qquad \varphi_t,\mu_t\in C([0,T];H).\nonumber
\end{align}
}
\end{oss}
}
{\color{black}
\begin{oss}{\upshape
We point out that, if the assumptions
$u_{0}\in V_{div}$, $\varphi_0\in V\cap C^\beta(\overline{\Omega})$, for
some $\beta>0$,
 and $h\in L_{loc}^{2}(\mathbb{R}^{+};G_{div})$ hold, the additional requirements
 on $J$ (cf. Theorem \ref{thmncvisc}) are not needed to prove the following regularity properties
\begin{align}
& u\in L^{\infty }(0,T;V_{div})\cap L^{2}(0,T;H^{2}(\Omega )^{2}), \quad
u_{t}\in L^{2}(0,T;G_{div}) \nonumber \\
& 
\varphi,\mu\in L^\infty(0,T;V)\cap L^{\infty
}(\Omega \times (0,T)),\quad\mu\in L^2(0,T;H^2(\Omega))\nonumber\\
&
\varphi_{t},\mu_t\in L^{2}(0,T;H),  \nonumber
\end{align}%
which imply, thanks to \eqref{est38},
\begin{align}
&\varphi\in L^2(0,T;W^{1,p}(\Omega))\cap L^q(0,T;W^{1,2q/(q-2)}(\Omega)),\nonumber
\end{align}
for every $2<p<\infty$ and every $2<q\leq\infty$.
The extra assumption on $J$ is needed only to prove that $\varphi\in L^2(0,T;H^2(\Omega))$
and, provided that $F\in C^3$ and $\varphi_0\in H^2(\Omega)$ hold as well, to
deduce the additional regularity properties \eqref{regps3} and \eqref{regps4}.
}
\end{oss}
}

We can now state the weak-strong uniqueness result for the nonconstant
viscosity case.

{\color{black} }

\begin{thm}
\label{weak-uniq} Let $d=2$ and assume that (H1)--(H5) are satisfied. Let $%
u_{0}\in G_{div}$, $\varphi _{0}\in L^{\infty }(\Omega )$ and let $%
[u_{1},\varphi _{1}]$ be a weak solution and $[u_{2},\varphi _{2}]$ a strong
solution satisfying {\color{black}\eqref{regps1} and
\eqref{regps1bis}} both corresponding to $[u_{0},\varphi _{0}]$ and to the same
external force $h\in L^2(0,T;V_{div}^{\prime})$.
Then $u_{1}=u_{2}$ and $\varphi _{1}=\varphi _{2}$.
\end{thm}

\begin{proof}
Taking the difference between the variational formulation of {\color{black}%
\eqref{sy3} and \eqref{sy1}} written for each solution and setting {%
\color{black} $u:=u_{2}-u_{1}$, $\varphi :=\varphi _{2}-\varphi _{1}$,} we get
\begin{align}
& \langle u_{t},v\rangle +2\big((\nu (\varphi _{2})-\nu (\varphi
_{1}))Du_{2},Dv\big)+2\big(\nu (\varphi _{1})Du,Dv\big)%
+b(u_{2},u_{2},v)-b(u_{1},u_{1},v)  \notag \\
& =-\frac{1}{2}\big(\varphi (\varphi _{1}+\varphi _{2})\nabla a,v\big)-\big(%
(J\ast \varphi )\nabla \varphi _{2},v\big)-\big((J\ast \varphi _{1})\nabla
\varphi ,v\big),  \label{vv1} \\
& \langle \varphi _{t},\psi \rangle +(\nabla \mu ,\nabla \psi )=-(u\cdot
\nabla \varphi _{2},\psi ){\color{black}-}(u_{1}\cdot \nabla \varphi ,\psi ),
\label{vv2}
\end{align}%
for all $v\in V_{div}$ and $\psi \in V$, where {\color{black}$\mu =\mu
_{2}-\mu _{1}=a\varphi -J\ast \varphi +F^{\prime }(\varphi _{2})-F^{\prime
}(\varphi _{1})$.} Let us choose $v=u$ and $\psi =\varphi $ as test
functions in \eqref{vv1} and \eqref{vv2}, respectively, and {\color{black}add}
the resulting identities. Notice that the contribution from the second term
on the right-hand side of \eqref{vv2} vanishes due to the incompressibility
condition. Hence, we get
\begin{align}
& \frac{1}{2}\frac{d}{dt}\big(\Vert u\Vert ^{2}+\Vert \varphi \Vert ^{2}\big)%
+2\big((\nu (\varphi _{2})-\nu (\varphi _{1}))Du_{2},Du\big)+2\big(\nu
(\varphi _{1})Du,Du\big)+b(u,u_{1},u)  \notag \\
& +(\nabla \mu ,\nabla \varphi )=I_{1}+I_{2}+I_{3}+I_{4},  \label{vv3}
\end{align}%
where $I_{1},I_{2},I_{3}$ are given again by
\begin{equation*}
I_{1}=-\frac{1}{2}\big(\varphi (\varphi _{1}+\varphi _{2})\nabla a,u\big)%
,\quad I_{2}=-\big((J\ast \varphi )\nabla \varphi _{2},u\big),\quad I_{3}=-%
\big((J\ast \varphi _{1})\nabla \varphi ,u\big),
\end{equation*}%
while $I_{4}$ is given by
\begin{equation*}
I_{4}=-(u\cdot \nabla \varphi _{2},\varphi ).
\end{equation*}%
Let us first estimate the terms in \eqref{vv3} coming from the Navier-Stokes
equations. Due to assumption (H2) we have
\begin{align}
& 2\big|\big((\nu (\varphi _{2})-\nu (\varphi _{1}))Du_{2},Du\big)\big|\leq
C\Vert \varphi \Vert _{L^{4}}\Vert Du_{2}\Vert _{L^{4}}\Vert \nabla u\Vert
\notag \\
& \leq C\Vert \varphi \Vert ^{1/2}\Vert \varphi \Vert _{V}^{1/2}\Vert
Du_{2}\Vert ^{1/2}\Vert Du_{2}\Vert _{H^{1}}^{1/2}\Vert \nabla u\Vert  \notag
\\
& \leq \frac{\nu _{1}}{12}\Vert \nabla u\Vert ^{2}+C\Vert \nabla u_{2}\Vert
\Vert u_{2}\Vert _{H^{2}}\Vert \varphi \Vert ^{2}+C\Vert \nabla u_{2}\Vert
\Vert u_{2}\Vert _{H^{2}}\Vert \varphi \Vert \Vert \nabla \varphi \Vert
\notag \\
& \leq \frac{\nu _{1}}{12}\Vert \nabla u\Vert ^{2}+\frac{c_{0}}{4}\Vert
\nabla \varphi \Vert ^{2}+C(1+\Vert \nabla u_{2}\Vert ^{2}\Vert u_{2}\Vert
_{H^{2}}^{2})\Vert \varphi \Vert ^{2},  \label{k}
\end{align}%
\begin{equation*}
2\big(\nu (\varphi _{1})Du,Du\big)\geq \nu _{1}\Vert \nabla u\Vert ^{2},
\end{equation*}%
where henceforth in this proof $C$ will denote a constant which depends on $%
\Vert \varphi _{0}\Vert _{L^{\infty }}$,
and on $\Vert u_{0}\Vert $. 
Indeed, recall that, since $\varphi _{0}\in L^{\infty }(\Omega )$, then we
have $\Vert \varphi _{i}\Vert _{L^{\infty }(\Omega \times (0,T))}\leq
C_{i}=C_{i}\big(\Vert \varphi _{0}\Vert _{L^{\infty }},\Vert u_{0}\Vert
\big)
$, for $i=1,2$.

The term in the trilinear form is standard
\begin{equation*}
|b(u,u_{1},u)|\leq c\Vert u\Vert \Vert \nabla u\Vert \Vert \nabla u_{1}\Vert
\leq \frac{\nu _{1}}{12}\Vert \nabla u\Vert ^{2}+c\Vert \nabla u_{1}\Vert
^{2}\Vert u\Vert ^{2},
\end{equation*}%
while the terms $I_{1},I_{2},I_{3}$ can now be estimated more easily in this
way
\begin{align*}
& I_{1}\leq \Vert \varphi \Vert \Vert \varphi _{1}+\varphi _{2}\Vert
_{L^{4}}\Vert \nabla a\Vert _{L^{\infty }}\Vert u\Vert _{L^{4}}  \notag \\
& \leq \frac{\nu _{1}}{12}\Vert \nabla u\Vert ^{2}+c\big(\Vert \varphi
_{1}\Vert _{L^{4}}^{2}+\Vert \varphi _{2}\Vert _{L^{4}}^{2}\big)\Vert
\varphi \Vert ^{2}, \\
& I_{2}\leq \Vert \varphi _{2}\Vert _{L^{4}}\Vert \nabla J\Vert
_{L^{1}}\Vert \varphi \Vert \Vert u\Vert _{L^{4}}  \notag \\
& \leq \frac{\nu _{1}}{12}\Vert \nabla u\Vert ^{2}+c\Vert \varphi _{2}\Vert
_{L^{4}}^{2}\Vert \varphi \Vert ^{2}, \\
& I_{3}\leq \frac{\nu _{1}}{12}\Vert \nabla u\Vert ^{2}+c\Vert \varphi
_{1}\Vert _{L^{4}}^{2}\Vert \varphi \Vert ^{2}.
\end{align*}%
Regarding the terms coming from the nonlocal Cahn-Hilliard equation we have
\begin{align}
& (\nabla \mu ,\nabla \varphi )=\big((a+F^{\prime \prime }(\varphi
_{1}))\nabla \varphi ,\nabla \varphi \big)+\big(\varphi \nabla a-\nabla
J\ast \varphi ,\nabla \varphi \big)  \notag \\
& +\big((F^{\prime \prime }(\varphi _{2})-F^{\prime \prime }(\varphi
_{1}))\nabla \varphi _{2},\nabla \varphi \big),  \label{diffF''}
\end{align}
and the last term on the right-hand side of this identity can be estimated
as
\begin{align*}
& \big|\big((F^{\prime \prime }(\varphi _{2})-F^{\prime \prime }(\varphi
_{1}))\nabla \varphi _{2},\nabla \varphi \big)\big|\leq \Vert F^{\prime
\prime }(\varphi _{2})-F^{\prime \prime }(\varphi _{1})\Vert _{L^{4}}\Vert
\nabla \varphi _{2}\Vert _{L^{4}}\Vert \nabla \varphi \Vert  \notag \\
& \leq C\Vert \varphi \Vert _{L^{4}}\Vert \nabla \varphi _{2}\Vert
_{L^{4}}\Vert \nabla \varphi \Vert \leq C(\Vert \varphi \Vert +\Vert \varphi
\Vert ^{1/2}\Vert \nabla \varphi \Vert ^{1/2}) {\color{black} \Vert \nabla
\varphi _{2}\Vert _{L^{4}}} \Vert \nabla \varphi \Vert  \notag \\
& \leq \frac{c_{0}}{4}\Vert \nabla \varphi \Vert ^{2}+C(1+ {\color{black}\Vert
\nabla \varphi _{2}\Vert _{L^{4}}^4} )\Vert \varphi \Vert ^{2}.  \notag
\end{align*}

Hence, by means of assumption (H3),
we get
\begin{align*}
& (\nabla \mu ,\nabla \varphi )\geq c_{0}\Vert \nabla \varphi \Vert
^{2}-2\Vert \nabla J\Vert _{L^{1}}\Vert \varphi \Vert \Vert \nabla \varphi
\Vert -\frac{c_{0}}{4}\Vert \nabla \varphi \Vert ^{2}-C(1+ {\color{black}\Vert
\nabla \varphi _{2}\Vert _{L^{4}}^4})\Vert \varphi \Vert ^{2}  \notag \\
& \geq \frac{c_{0}}{2}\Vert \nabla \varphi \Vert ^{2}-C(1+ {\color{black}\Vert
\nabla \varphi _{2}\Vert _{L^{4}}^4})\Vert \varphi \Vert ^{2}.
\end{align*}%
Finally, the last term in \eqref{vv3} coming from the nonlocal Cahn-Hilliard
equation can be controlled as follows
\begin{equation}
I_{4}\leq \Vert u\Vert _{L^{4}}\Vert \nabla \varphi _{2}\Vert _{L^{4}}\Vert
\varphi \Vert \leq \frac{\nu _{1}}{12}\Vert \nabla u\Vert ^{2}+c {\color{black}
\Vert \nabla \varphi _{2}\Vert _{L^{4}}^2} \Vert \varphi \Vert ^{2}.
\label{vv11}
\end{equation}

By plugging estimates \eqref{k}--\eqref{vv11} into \eqref{vv3} we are
led to the following differential inequality
\begin{equation}
\frac{1}{2}\frac{d}{dt}\big(\Vert u\Vert ^{2}+\Vert \varphi \Vert ^{2}\big)+%
\frac{\nu _{1}}{2}\Vert \nabla u\Vert ^{2}+\frac{c_{0}}{4}\Vert \nabla
\varphi \Vert ^{2}\leq {\color{black}\Pi} \big(\Vert u\Vert ^{2}+\Vert \varphi \Vert
^{2} \big),  \label{vv12}
\end{equation}%
where the function ${\color{black}\Pi}$ is given by
\begin{equation*}
{\color{black} \Pi} =c\big(1+\Vert \nabla u_{2}\Vert ^{2}\Vert u_{2}\Vert
_{H^{2}}^{2}+\Vert \nabla u_{1}\Vert ^{2}+\Vert \varphi _{1}\Vert
_{L^{4}}^{2}+\Vert \varphi _{2}\Vert _{L^{4}}^{2}+ {\color{black} \Vert \nabla
\varphi _{2}\Vert _{L^{4}}^2+\Vert \nabla \varphi _{2}\Vert _{L^{4}}^4} \big),
\end{equation*}%
and due to the regularity properties of the weak solution $[u_{1},\varphi
_{1}]$ and of the strong solution $[u_{2},\varphi _{2}]$ we have $\Pi \in
L^{1}(0,T)$. Weak-strong uniqueness follows by applying Gronwall's lemma to %
\eqref{vv12}. In addition, a continuous dependence estimate in $L^{2}(\Omega)^2$
can also be deduced by considering two solutions with different
initial data {\color{black} and external forces}.
\end{proof}

{\color{black} Before concluding this section let us make some remarks on
weak-strong uniqueness in the case of nonconstant viscosity and singular
potential. We first observe that, if the potential is singular and the
mobility is constant, then weak-strong uniqueness does not seem to be easy to
prove. The reason is in the way the term $2\big((\nu (\varphi _{2})-\nu
(\varphi _{1}))Du_{2},Du\big)$ in \eqref{vv1} can be estimated (even
assuming higher regularity for the strong solution $[u_2,\varphi_2]$). This
forces to choose $\varphi$ (instead of $B_N^{-1}\varphi$) as test function
in \eqref{vv2}. With this choice we have the term $(\nabla\mu,\nabla\varphi)$
(instead of $(\mu,\varphi)$) on the left-hand side of %
\eqref{vv2}. Therefore we are led to deal with the
difference $F^{\prime\prime}(\varphi_2)-F^{\prime\prime}(\varphi_1)$ (cf. %
\eqref{diffF''}) which we do not know how to handle.

However, if the potential is singular and the mobility is degenerate,
thanks to the particular weak formulation of the convective nonlocal
Cahn-Hilliard (cf. \eqref{nlocCH2}), the weak-strong
uniqueness can be proven as stated in the next theorem. In order to
do that, we just need to strengthen (A1) slightly, namely,

\begin{description}
\item[(A7)] $mF^{\prime\prime}\in C^1([-1,1])$.
\end{description}

We point out that in the case of singular potential, degenerate mobility and
constant (or nonconstant) viscosity, existence of strong solutions in 2D for
the nonlocal Cahn-Hilliard-Navier-Stokes system has not been proven yet.
This result, which actually can be established, will be presented in a
forthcoming paper.

\begin{thm}
Let $d=2$ and suppose that assumptions (A1)--(A7) and (H2)
are satisfied. Let $u_0\in G_{div}$, $\varphi_0\in L^\infty(\Omega)$ with $%
F(\varphi_0)\in L^1(\Omega)$, $M(\varphi_0)\in L^1(\Omega)$ and  let $%
[u_{1},\varphi _{1}]$ be a weak solution and $[u_{2},\varphi _{2}]$ a strong
solution to \eqref{sy3}--\eqref{sy6} satisfying  \eqref{regps1} and
\eqref{regps2} both corresponding to $[u_{0},\varphi _{0}]$ and to the same
external force $h\in L^2(0,T;V_{div}^{\prime})$.
Then $u_{1}=u_{2}$ and $\varphi _{1}=\varphi _{2}$.
\end{thm}

\begin{proof}
Let us write the variational formulation of \eqref{sy3}--\eqref{sy4} and \eqref{nlocCH2}
for each solution and take the difference, setting $u:=u_2-u_1$, $%
\varphi:=\varphi_2-\varphi_1$. Then we choose $v=u$ as test function in the
first identity \eqref{vv1}  and $\psi\zeta=\varphi$ as test function in the
second. Concerning the first identity, we can argue exactly as in the proof
of Theorem \ref{weak-uniq} and get
\begin{align}
& \frac{1}{2}\frac{d}{dt}\Vert u\Vert ^{2}+2\big((\nu (\varphi _{2})-\nu
(\varphi _{1}))Du_{2},Du\big)+2\big(\nu (\varphi _{1})Du,Du\big)%
+b(u,u_{1},u),  \notag \\
&=I_{1}+I_{2}+I_{3}. \label{diffid2}
\end{align}%
Then, by similarly estimating the terms in \eqref{diffid2}, we find
\begin{align}
&\frac{1}{2}\frac{d}{dt}\Vert u\Vert^2+\frac{\nu_1}{2}\Vert\nabla
u\Vert^2\leq\frac{1}{4}(1-\rho)\alpha_0\Vert\nabla\varphi\Vert^2  \notag \\
& +C(1+\Vert \nabla u_{2}\Vert ^{2}\Vert u_{2}\Vert _{H^{2}}^{2}+\Vert
\varphi _{1}\Vert _{L^{4}}^{2}+\Vert \varphi _{2}\Vert
_{L^{4}}^{2})\Vert\varphi\Vert^2+C\Vert \nabla u_{1}\Vert ^{2}\Vert u\Vert^2.
\label{diffineq1}
\end{align}
As far as the identity resulting from the difference in the Cahn-Hilliard is
concerned, if we set
\begin{align*}
&b(x,s):=\partial_s\Lambda(x,s)=m(s)(F^{\prime\prime}(s)+a(x)),\qquad\forall
s\in[-1,1],\quad\mbox{a.e. }x\in\Omega,
\end{align*}
this identity reads as follows
\begin{align}
&\frac{1}{2}\frac{d}{dt}\Vert \varphi\Vert^2 +\big(b(\cdot,\varphi_1)\nabla%
\varphi,\nabla\varphi\big) +\big((b(\cdot,\varphi_2)-b(\cdot,\varphi_1))%
\nabla\varphi_2,\nabla\varphi\big)  \notag \\
&+\big((m(\varphi_2)-m(\varphi_1))(\varphi_2\nabla a-\nabla
J\ast\varphi_2),\nabla\varphi\big)  \notag \\
&+\big(m(\varphi_1)(\varphi\nabla a -\nabla J\ast\varphi),\nabla \varphi\big)
\notag \\
&=\big(u\varphi_2,\nabla\varphi\big).  \label{nlocCH8}
\end{align}
Observe now that, thanks to assumptions (A5), (A6) and (A7), we have
\begin{align*}
&b(x,s)\geq (1-\rho)\alpha_0,\qquad |b(x,s_2)-b(x,s_1)|\leq
k^{\prime}|s_2-s_1|,
\end{align*}
for all $s,s_1,s_2\in[-1,1]$ and for almost every $x\in\Omega$. Here
$k^{\prime}=\Vert(mF^{\prime\prime})^{\prime}\Vert_{C([-1,1])}+\Vert
m^{\prime}\Vert_{C([-1,1])}\Vert a\Vert_{L^\infty(\Omega)}$. Let us now
estimate the terms in \eqref{nlocCH8}, taking the bounds $%
|\varphi_i|\leq 1$, $i=1,2$, into account. 
The second and third term on the left-hand side 
can be estimated in the following way
\begin{align*}
\big(b(\cdot,\varphi_1)\nabla\varphi,\nabla\varphi\big)&\geq(1-\rho)\alpha_0%
\Vert\nabla\varphi\Vert^2, \\
\big((b(\cdot,\varphi_2)-b(\cdot,\varphi_1))\nabla\varphi_2,\nabla\varphi%
\big)&\leq
k^{\prime}\Vert\varphi\Vert_{L^4}\Vert\nabla\varphi_2\Vert_{L^4}\Vert\nabla%
\varphi\Vert  \notag \\
&\leq
c\Vert\varphi\Vert^{1/2}\Vert\varphi\Vert_V^{1/2}\Vert\nabla\varphi_2%
\Vert_{L^4}\Vert\nabla\varphi\Vert  \notag \\
&\leq \frac{1}{32}(1-\rho)\alpha_0\Vert\nabla\varphi\Vert^2+c\Vert\varphi%
\Vert\Vert\varphi\Vert_V\Vert\nabla\varphi_2\Vert_{L^4}^2  \notag \\
&\leq\frac{1}{16}(1-\rho)\alpha_0\Vert\nabla\varphi\Vert^2+c\big(%
1+\Vert\nabla\varphi_2\Vert_{L^4}^4\big)\Vert\varphi\Vert^2.
\end{align*}
Furthermore, it is immediate to see that the last two terms on the left-hand
side of \eqref{nlocCH8} can be controlled in this way
\begin{align*}
& c\Vert\varphi\Vert\Vert\nabla\varphi\Vert\leq\frac{1}{16}%
(1-\rho)\alpha_0\Vert\nabla\varphi\Vert^2+c\Vert\varphi\Vert^2,
\end{align*}
and, finally, the term on the right-hand side can be controlled by
\begin{align*}
& c\Vert u\Vert\Vert\nabla\varphi\Vert\leq\frac{1}{16}(1-\rho)\alpha_0\Vert%
\nabla\varphi\Vert^2+c\Vert u\Vert^2.
\end{align*}
From \eqref{nlocCH8}, using the estimates above, we are therefore led to the
following differential inequality
\begin{align}
&\frac{1}{2}\frac{d}{dt}\Vert \varphi\Vert^2 +\frac{3}{4}(1-\rho)\alpha_0%
\Vert\nabla\varphi\Vert^2\leq c\big(1+\Vert\nabla\varphi_2\Vert_{L^4}^4\big)%
\Vert\varphi\Vert^2+c\Vert u\Vert^2.  \label{diffineq2}
\end{align}
Thus, from \eqref{diffineq1} and \eqref{diffineq2} we deduce
\begin{equation}
\frac{1}{2}\frac{d}{dt}\big(\Vert u\Vert ^{2}+\Vert \varphi \Vert ^{2}\big)+%
\frac{\nu _{1}}{2}\Vert \nabla u\Vert ^{2}+\frac{1}{2}(1-\rho)\alpha_0\Vert
\nabla \varphi \Vert ^{2}\leq \gamma \big(\Vert u\Vert ^{2}+\Vert \varphi
\Vert ^{2} \big),  \label{diffineq3}
\end{equation}%
where $\gamma \in L^1(0,T)$ has the same form as given at the
end of the proof of Theorem \ref{weak-uniq}. We conclude again by applying
Gronwall's lemma to \eqref{diffineq3}. Moreover, a continuous dependence
estimate in $L^2(\Omega)^2$ can be deduced in the present situation as well
by considering two solutions with different data.
\end{proof}

\begin{oss}
\label{openuniq} {\upshape
Uniqueness of weak solutions for the nonlocal Cahn-Hilliard-Navier-Stokes
system in 2D with nonconstant viscosity is an open issue. The difficulty
essentially comes from the term $2\big((\nu (\varphi
_{2})-\nu (\varphi _{1}))Du_{2},Du\big)$ in \eqref{diffid2}, which forces to
assume that one solution (e.g., $[u_2,\varphi_2]$) is strong. }
\end{oss}
}

\section{Global and exponential attractors}

\label{Sec5}\setcounter{equation}{0}

In this section we prove two results concerning the asymptotic behavior of
the dynamical system generated by \eqref{sy1}--\eqref{sy5} in dimension two.

The first result is related to the property of connectedness of the global
attractor whose existence was established in \cite{FG1} for nonconstant
viscosity, constant mobility and regular potential (see Remark \ref{singconn}
below, however).

The second result is the existence of an exponential attractor. This will be
proven in details when mobility and viscosity are constant and the potential
is regular. This kind of result relies on a regularization argument devised
in \cite{FGK} and on an abstract theorem (see \cite{EZ}) which generalizes a
well known result on the existence of exponential attractors in Banach
spaces (cf. \cite{EMZ}). A similar argument will be carried out in the
nonconstant viscosity case albeit we will work with strong solutions.

Let us define the dynamical system in the autonomous case. Take $d=2$ and $%
h\in V_{div}^{\prime }$. Then, as a consequence of Theorem \ref{uniqthm}, we
have that for every fixed $\eta \geq 0$ system \eqref{sy1}--\eqref{sy5}
generates a semigroup $\{S_{\eta }(t)\}_{t\geq 0}$ of {\itshape closed}
operators (see \cite{PZ}) on the metric space $\mathcal{X}_{\eta }$ given by
\begin{equation}  \label{phasespace}
\mathcal{X}_{\eta }:=G_{div}\times \mathcal{Y}_{\eta },
\end{equation}%
where
\begin{equation*}
\mathcal{Y}_{\eta }:=\{\varphi \in H:F(\varphi )\in L^{1}(\Omega ),\text{ }|%
\overline{\varphi }|\leq \eta \}.
\end{equation*}%
It is convenient to endow the space $\mathcal{X}_{\eta }$ with the following
metric
\begin{equation*}
\rho _{\mathcal{X}_{\eta }}(z_{2},z_{1})=\Vert u_{2}-u_{1}\Vert +\Vert
\varphi _{2}-\varphi _{1}\Vert +\Big|\int_{\Omega }F(\varphi
_{2})-\int_{\Omega }F(\varphi _{1})\Big|,\quad \forall z_{i}:=[u_{i},\varphi
_{i}]\in \mathcal{X}_{\eta },\; i=1,2.
\end{equation*}%
Notice that this metric is slightly different from the one which is
naturally associated to the energy $\mathcal{E}$ (the difference is in the
exponent in the third term, see \cite{FG1}).

A first noteworthy consequence of the uniqueness result for weak solutions
is the following

\begin{thm}
\label{connected} {\color{black}Let $d=2$ and let (H1)--(H5) be satisfied with
{\color{black}$\nu$ constant}. Assume also that that $h\in V_{div}^{\prime}$.}
Then, the global attractor in $\mathcal{X}_\eta$ for the semigroup $S_\eta(t)
$ is connected.
\end{thm}

\begin{proof}
The conclusion follows immediately by applying \cite[Corollary 4.3]{Ba}.
Indeed, the space $\mathcal{X}_\eta$ is (arcwise) connected, thanks to the
fact that $F$ is a quadratic perturbation of a convex function. Moreover, we
have the strong time continuity of each trajectory $z=[u,\varphi]$ from $%
[0,\infty)$ to the metric space $\mathcal{X}_\eta$ (see Theorem \ref{thm}).
Thus Kneser's property is satisfied thanks to uniqueness.
\end{proof}


\begin{oss}
\label{singconn} {\upshape Theorem \ref{connected} also holds in the case of
constant (or degenerate) mobility and singular potential on account of
Theorem \ref{uniqthmsing} and \cite[Proposition 4]{FG2} (or Theorem \ref%
{uniqthmdeg} and \cite[Proposition 3]{FGR}). The argument is similar. {%
\color{black} On the other hand, if the viscosity is nonconstant, then the
connectedness of the global attractor is an open issue (cf. Remark %
\ref{openuniq}). } }
\end{oss}

The second result is the existence of an exponential attractor. We first
recall its definition.

\begin{defn}
A compact set $\mathcal{M}_\eta\subset\mathcal{X}_\eta$ is an {\itshape %
exponential attractor} for the dynamical system $(\mathcal{X}%
_\eta,S_\eta(t)) $ if the following properties are satisfied

\begin{itemize}
\item[(i)] positive invariance: $S_{\eta }(t)\mathcal{M}_{\eta }\subseteq
\mathcal{M}_{\eta }$ for all $t\geq 0$;

\item[(ii)] finite dimensionality: dim$_F(\mathcal{M}_\eta,\mathcal{X}%
_\eta)<\infty$;

\item[(iii)] exponential attraction: $\exists$ $Q:\mathbb{R}%
^{+}\rightarrow \mathbb{R}^{+}$ increasing and $\kappa >0$ such that, for
all $R>0$ and for all $\mathcal{B}\subset \mathcal{X}_{\eta }$ with $%
\sup_{z\in \mathcal{B}}\rho _{\mathcal{X}_{\eta }}(z,0)\leq R$ there holds
\begin{equation*}
dist_{\mathcal{X}_{\eta }}(S_{\eta }(t)\mathcal{B},\mathcal{M}_{\eta })\leq
Q(R)e^{-\kappa t},\qquad \forall t\geq 0.
\end{equation*}
\end{itemize}
\end{defn}

\begin{thm}
\label{expattthm} Let $d=2$. Assume that (H1)--(H5) are satisfied with {%
\color{black}$\nu$ constant}. Then the dynamical system $(\mathcal{X}%
_\eta,S_\eta(t))$ possesses an exponential attractor $\mathcal{M}_\eta$
which is bounded in $V_{div}\times W^{1,p}(\Omega)$, $2<p<\infty$.
\end{thm}

The proof of Theorem \ref{expattthm} is based on four lemmas. These lemmas
allow us to apply the abstract result in \cite{EZ}. For their proof we shall
need the following regularization result which is an easy consequence of
\cite[Theorem 2 and Proposition 1]{FGK} and has an independent interest. In
the statement and proof of this result we shall denote by $%
\Gamma_\tau=\Gamma_\tau\big(\mathcal{E}(z_0),\eta\big)$ a positive constant
depending on a positive time $\tau$, on the energy $\mathcal{E}(z_0)$ of the
initial datum $z_0:=[u_0,\varphi_0]$ of a weak solution, and on $\eta$,
where $\eta\geq 0$ is such that $|\overline{\varphi}_0|\leq\eta$ ($%
\Gamma_\tau$ may of course depend also on $h$, $F$, $J$, $\nu$ and $\Omega$%
). The value of $\Gamma_\tau$ may change even on the same line.

\begin{prop}
\label{regthmtau} Let $d=2$ and $h\in L_{tb}^{2}(0,\infty ;G_{div})$. Assume
that (H1)--(H5) are satisfied with {\color{black}$\nu$ constant}, and suppose $%
F\in C^{3}(\mathbb{R})$. Let $u_{0}\in G_{div}$, $\varphi _{0}\in H$ with $%
F(\varphi _{0})\in L^{1}(\Omega )$ and let $[u,\varphi ]$ be the weak
solution on $(0,\infty )$ to system \eqref{sy1}--\eqref{sy6} corresponding
to $[u_{0},\varphi _{0}]$. Then, for every $\tau >0$ there exists $\Gamma
_{\tau }>0$ such that we have
\begin{align}
& u\in L^{\infty }(\tau ,\infty ;V_{div})\cap L_{tb}^{2}\big(\tau ,\infty
;H^{2}(\Omega )^{2}\big),\qquad u_{t}\in L_{tb}^{2}\big(\tau ,\infty ;G_{div}%
\big),  \label{pregtau1} \\
& \varphi \in L^{\infty }\big(\tau ,\infty ;W^{1,p}(\Omega )\big),\qquad
2<p<\infty ,\qquad \varphi _{t}\in L^{\infty }(\tau ,\infty ;H)\cap
L_{tb}^{2}(\tau ,\infty ;V),  \label{pregtau2}
\end{align}%
with norms controlled by $\Gamma _{\tau }$.
In addition, {\color{black} for every
initial data $z_{0}:=[u_{0},\varphi _{0}]\in G_{div}\times H$ with $%
F(\varphi _{0})\in L^{1}(\Omega )$ and $|\overline{\varphi }_{0}|\leq \eta $,
there exists a constant $\Lambda =\Lambda (\eta )>0$ depending
only on $\eta $ (and on $F$, $J$, $\nu $ and $\Omega $) and
a time $t^{\ast }=t^{\ast }\big(\mathcal{E}(z_{0})\big)\geq 0$}
starting from which the weak solution corresponding to $z_{0}$ regularizes,
that is,
\begin{equation}
\Vert \nabla u(t)\Vert +\Vert \varphi (t)\Vert _{W^{1,p}(\Omega
)}+\int_{t}^{t+1}\Vert u(s)\Vert _{H^{2}(\Omega )^{2}}^{2}ds\leq \Lambda
(\eta ),\qquad \forall t\geq t^{\ast }.  \label{pregtau3}
\end{equation}
\end{prop}

\begin{oss}
{\upshape
Notice that, differently from \cite[Theorem 2]{FGK}, in Proposition \ref%
{regthmtau} we do not require any further regularity assumption on $J$ in
addition to (H1). }
\end{oss}

\begin{proof}
{\color{black} Recalling the proof of \cite[Lemma 2.10]{GG4} and the dissipative estimate %
\eqref{dissest}, observe first that, if $z_{0}\in \mathcal{X}_{\eta }$, then
for every $\tau >0$ there exists $\Gamma _{\tau }=\Gamma _{\tau }\big(%
\mathcal{E}(z_{0}),\eta \big)$ such that
\begin{equation}
\Vert \varphi (t)\Vert _{L^{\infty }(\Omega )}\leq \Gamma _{\tau },\qquad
\forall t\geq \tau .  \label{sup-norm}
\end{equation}%
This implies that $\Vert \mu (t)\Vert _{L^{\infty }(\Omega )}\leq \Gamma
_{\tau }$ for all $t\geq \tau $, and hence that the Korteweg term $\mu
\nabla \varphi \in L^{2}(\tau ,T;L^{2}(\Omega )^{2})$. By Lemma \ref{holder}%
, there also holds%
\begin{equation}
\sup_{t\geq \tau }\left\Vert \varphi \right\Vert _{C^{\delta /2,\delta
}\left( \left[ t,t+1\right] \times \overline{\Omega }\right) }\leq \Gamma
_{\tau }\text{, }\forall t\geq \tau .  \label{holdi}
\end{equation}
}
We can now repeat exactly the same argument in the proof of \cite[Theorem 2]%
{FGK}, by writing the same estimates which now hold starting from a positive
time, say for $t\geq \tau /2>0$. We recall that these estimates are obtained
by multiplying the nonlocal Cahn-Hilliard by $\mu _{t}$ in $H$ and then by
differentiating the nonlocal Cahn-Hilliard with respect to time and
multiplying the resulting identity bu $\mu _{t}$. By doing so we are led to
a differential inequality of the following form
\begin{equation}
\frac{d}{ds}\log \Big(1+\int_{\Omega }\big(a+F^{\prime \prime }(\varphi )%
\big)\varphi _{t}^{2}\Big)\leq \Gamma _{\tau }\big(\sigma (s)+\Vert \varphi
_{t}\Vert ^{2}\big),\qquad \forall s\geq \tau /2,  \label{variant1}
\end{equation}%
where $\sigma =\Gamma _{\tau }\big(1+\Vert u\Vert _{H^{2}}^{2}+\Vert
u_{t}\Vert ^{2}$\big)and we have $\sigma \in L^{1}(\tau /2,T)$, for all $%
T>\tau /2$. At this point we argue a bit differently from the proof of \cite[%
Theorem 2]{FGK}. Indeed, here we want to avoid the $L^{2}$-norm of $\varphi
_{t}$ in $\tau /2$ which would require the initial condition $\varphi (\tau
/2)\in H^{2}(\Omega)$ and in addition would force us to make some further regularity
assumptions on the kernel $J$ (like, e.g., $J\in W^{2,1}(\mathbb{R}^2)$ {\color{black} or $J$ admissible)
in order to have $\varphi_t(\tau/2)\in H$}. Therefore, we
multiply \eqref{variant1} by $(s-\tau /2)$ and integrate with respect to $s$
between $\tau /2$ and $t\in (\tau /2,T)$. We get
\begin{align}
& \Big(t-\frac{\tau }{2}\Big)\log \Big(1+\int_{\Omega }\big(a+F^{\prime
\prime }(\varphi )\big)\varphi _{t}^{2}\Big)\leq \int_{\tau /2}^{T}\log \Big(%
1+\int_{\Omega }\big(a+F^{\prime \prime }(\varphi )\big)\varphi _{t}^{2}\Big)%
ds  \notag \\
& +\Gamma _{\tau }\Big(T-\frac{\tau }{2}\Big)\big(\Vert \sigma \Vert
_{L^{1}(\tau /2,T)}+\Vert \varphi _{t}\Vert _{L^{2}(\tau /2,T;H)}^{2}\big)
\notag \\
& \leq \Gamma _{\tau }\Vert \varphi _{t}\Vert _{L^{2}(\tau
/2,T;H)}^{2}+\Gamma _{\tau }\Big(T-\frac{\tau }{2}\Big)\big(\Vert \sigma
\Vert _{L^{1}(\tau /2,T)}+\Vert \varphi _{t}\Vert _{L^{2}(\tau /2,T;H)}^{2}%
\big),\qquad \forall t\in (\tau /2,T).  \notag
\end{align}%
From this inequality, on account of the fact that we have $\Vert \varphi
_{t}\Vert _{L^{2}(\tau /2,T;H)}\leq \Gamma _{\tau }$ (this was shown in the
first step of the proof of \cite[Theorem 2]{FGK}, before \eqref{variant1})
we deduce that
\begin{equation}
\varphi _{t}\in L^{\infty }(\tau ,T;H).  \label{variant2}
\end{equation}%
This bound, together with the following estimate (cf. proof of \cite[Theorem
2]{FGK})
\begin{equation*}
\Vert \nabla \mu \Vert _{L^{p}}\leq \Gamma _{\tau }\big(1+\Vert \varphi
_{t}\Vert ^{1-2/p}\big),\qquad 2<p<\infty ,
\end{equation*}%
yield
\begin{equation}
\varphi \in L^{\infty }\big(\tau ,T;W^{1,p}(\Omega )\big).  \label{variant3}
\end{equation}

Finally, arguing as in the proof of \cite[Proposition 1]{FGK} by applying
the uniform Gronwall's lemma, and taking \eqref{variant2}, \eqref{variant3}
(together with the bounds for $u$ on $(\tau ,T)$) into account, we get %
\eqref{pregtau1}, \eqref{pregtau2} and \eqref{pregtau3}, respectively.
\end{proof}

For the statements and proofs of the following lemmas we shall denote by $%
C_\tau=C_\tau\big(\mathcal{E}(z_{01}),\mathcal{E}(z_{02}),\eta\big)$ a
positive constant depending on a positive time $\tau$, on the energies $%
\mathcal{E}(z_{01})$, $\mathcal{E}(z_{02})$ of the initial data $%
z_{01},z_{02}\in\mathcal{X}_\eta$ of two weak solutions, and on $\eta$,
where $\eta>0$ is such that $|\overline{\varphi}_{01}|,|\overline{\varphi}%
_{02}|\leq\eta$ (of course, $C_\tau$ will generally depend also on $h$, $F$,
$J$, $\nu$ and $\Omega$). The value of $C_\tau$ may change even within the
same line. Furthermore, we shall always set $u:=u_2-u_1$, $%
\varphi:=\varphi_2-\varphi_1$.

\begin{lem}
\label{expattlem1} Let $d=2$. Assume that (H1)--(H5) are satisfied with {%
\color{black}$\nu$ constant} and that $F\in C^3(\mathbb{R})$. Let $u_{0i}\in
G_{div}$, $\varphi_{0i}\in H$ with $F(\varphi_{0i})\in L^1(\Omega)$ and $%
[u_i,\varphi_i] $ be the corresponding weak solutions, $i=1,2$. Then, for
every $\tau>0$ there exists $C_\tau>0$ 
such that we have {\small
\begin{align}
&\Vert u_2(t)-u_1(t)\Vert^2+\Vert\varphi_2(t)-\varphi_1(t)\Vert^2
+\int_\tau^t\Big(\frac{\nu}{4}\Vert\nabla(u_2(s)-u_1(s))\Vert^2 +\frac{c_0}{4%
}\Vert\nabla(\varphi_2(s)-\varphi_1(s))\Vert^2\Big)ds  \notag \\
&\leq e^{C_\tau t}\big(\Vert
u_2(\tau)-u_1(\tau)\Vert^2+\Vert\varphi_2(\tau)-\varphi_1(\tau)\Vert^2\big),
\qquad\forall t\geq\tau.  \label{expatt0}
\end{align}
}
\end{lem}

\begin{proof}
Let us multiply \eqref{CHdiff} by $\varphi $ in $L^{2}(\Omega )$. We get
\begin{equation}
\frac{1}{2}\frac{d}{dt}\Vert \varphi \Vert ^{2}=-(u\cdot \nabla \varphi
_{2},\varphi )-(\nabla \widetilde{\mu },\nabla \varphi )  \label{expatt1}
\end{equation}%
Taking the gradient of $\widetilde{\mu }$, on account of \eqref{chpotdiff}
we have
\begin{align*}
& (\nabla \widetilde{\mu },\nabla \varphi )=\int_{\Omega }\big(a+F^{\prime
\prime }(\varphi _{1})\big)|\nabla \varphi |^{2}+(\varphi \nabla a-\nabla
J\ast \varphi ,\nabla \varphi ) \\
& +\big((F^{\prime \prime }(\varphi _{2})-F^{\prime \prime }(\varphi
_{1}))\nabla \varphi _{2},\nabla \varphi \big)\geq c_{0}\Vert \nabla \varphi
\Vert ^{2}-c\Vert \varphi \Vert \Vert \nabla \varphi \Vert \\
& -\Vert F^{\prime \prime }(\varphi _{2})-F^{\prime \prime }(\varphi
_{1})\Vert _{L^{4}}\Vert \nabla \varphi _{2}\Vert _{L^{4}}\Vert \nabla
\varphi \Vert \geq \frac{c_{0}}{2}\Vert \nabla \varphi \Vert ^{2}-c\Vert
\varphi \Vert ^{2}-C_{\tau }\Vert \varphi \Vert _{L^{4}}\Vert \nabla \varphi
_{2}\Vert _{L^{4}}\Vert \nabla \varphi \Vert \\
& \geq \frac{c_{0}}{2}\Vert \nabla \varphi \Vert ^{2}-c\Vert \varphi \Vert
^{2}-C_{\tau }\big(\Vert \varphi \Vert +\Vert \varphi \Vert ^{1/2}\Vert
\nabla \varphi \Vert ^{1/2}\big)\Vert \nabla \varphi _{2}\Vert _{L^{4}}\Vert
\nabla \varphi \Vert \\
& \geq \frac{c_{0}}{4}\Vert \nabla \varphi \Vert ^{2}-C_{\tau }\big(1+\Vert
\nabla \varphi _{2}\Vert _{L^{4}}^{2}+\Vert \nabla \varphi _{2}\Vert
_{L^{4}}^{4}\big)\Vert \varphi \Vert ^{2}.
\end{align*}%
Observe that
\begin{equation}
(\nabla \widetilde{\mu },\nabla \varphi )\geq \frac{c_{0}}{4}\Vert \nabla
\varphi \Vert ^{2}-C_{\tau }\big(1+\Vert \nabla \varphi _{2}\Vert
_{L^{4}}^{4}\big)\Vert \varphi \Vert ^{2}.  \label{expatt2}
\end{equation}%
Furthermore, we have
\begin{equation}
|(u\cdot \nabla \varphi _{2},\varphi )|\leq \Vert u\Vert _{L^{4}}\Vert
\nabla \varphi _{2}\Vert _{L^{4}}\Vert \varphi \Vert \leq \frac{\nu }{4}%
\Vert \nabla u\Vert ^{2}+c\Vert \nabla \varphi _{2}\Vert _{L^{4}}^{2}\Vert
\varphi \Vert ^{2}.  \label{expatt3}
\end{equation}%
Therefore, plugging \eqref{expatt2} and \eqref{expatt3} into \eqref{expatt1}%
, we get
\begin{equation*}
\frac{1}{2}\frac{d}{dt}\Vert \varphi \Vert ^{2}+\frac{c_{0}}{4}\Vert \nabla
\varphi \Vert ^{2}\leq C_{\tau }\big(1+\Vert \nabla \varphi _{2}\Vert
_{L^{4}}^{4}\big)\Vert \varphi \Vert ^{2}+\frac{\nu }{4}\Vert \nabla u\Vert
^{2}.
\end{equation*}%
Adding this last differential inequality to \eqref{est4}, we obtain
\begin{equation}
\frac{1}{2}\frac{d}{dt}\big(\Vert u\Vert ^{2}+\Vert \varphi \Vert ^{2}\big)+%
\frac{\nu }{4}\Vert \nabla u\Vert ^{2}+\frac{c_{0}}{4}\Vert \nabla \varphi
\Vert ^{2}\leq \gamma (t)\big(\Vert u\Vert ^{2}+\Vert \varphi \Vert ^{2}\big)%
,  \label{expatt4}
\end{equation}%
where
\begin{equation*}
\gamma (t):=\alpha (t)+C_{\tau }\big(1+\Vert \nabla \varphi _{2}\Vert
_{L^{4}}^{4}\big).
\end{equation*}%
%
%
%
%
%
%
%
%
%
%
%
%
%
%
%
%
%
Then, thanks to Proposition \ref{regthmtau}, for every $\tau >0$ there
exists $C_{\tau }>0$ (always depending on $\tau $, $\eta $ and on the
energies $\mathcal{E}(z_{01})$, $\mathcal{E}(z_{02})$) such that the
following bounds for the solutions $z_{i}=[u_{i},\varphi _{i}]$
corresponding to $[u_{0i},\varphi _{0i}]$ hold 
\begin{align}
& \Vert u_{i}\Vert _{L^{\infty }(\tau ,\infty ;V_{div})}+\Vert \varphi
_{i}\Vert _{L^{\infty }(\tau ,\infty ;W^{1,p}(\Omega ))}\leq C_{\tau },
\label{expatt12} \\
& \Vert u_{i,t}\Vert _{L_{tb}^{2}(\tau ,\infty ;G_{div})}+\Vert \varphi
_{i,t}\Vert _{L^{\infty }(\tau ,\infty ;H)}\leq C_{\tau },  \label{expatt17}
\end{align}%
%
%
%
%
%
%
%
%
%
%
%
%
%
%
%
%
%
Thus we have $\gamma (t)\leq C_{\tau }$, for all $t\geq \tau $ and by
applying the standard Gronwall lemma to \eqref{expatt4} written for $t\geq
\tau $ we get
\begin{equation}
\Vert u(t)\Vert ^{2}+\Vert \varphi (t)\Vert ^{2}\leq \big(\Vert u(\tau
)\Vert ^{2}+\Vert \varphi (\tau )\Vert ^{2}\big)e^{C_{\tau }t},\qquad
\forall t\geq \tau .  \label{expatt5}
\end{equation}%
By integrating \eqref{expatt4} between $\tau $ and $t$ and using %
\eqref{expatt5} we get \eqref{expatt0}.
\end{proof}

\begin{lem}
\label{expattlem2} Let the assumptions of Lemma \ref{expattlem1} be
satisfied. Let $u_{0i}\in G_{div}$, $\varphi_{0i}\in H$ with $%
F(\varphi_{0i})\in L^1(\Omega)$ and $[u_i,\varphi_i]$ be the corresponding
weak solutions, $i=1,2$. Then, for every $\tau>0$ there exists $C_\tau>0$
such that we have {\small
\begin{align}
&\Vert u_2(t)-u_1(t)\Vert^2+\Vert\varphi_2(t)-\varphi_1(t)\Vert^2 +\Big|%
\int_{\Omega}F\big(\varphi_2(t)\big)-\int_{\Omega}F\big(\varphi_1(t)\big)%
\Big|^2  \notag \\
& \leq C_\tau \big(\Vert
u_2(\tau)-u_1(\tau)\Vert^2+\Vert\varphi_2(\tau)-\varphi_1(\tau)\Vert^2\big) %
e^{-kt}  \notag \\
&+C_\tau\int_\tau^t\big(\Vert
u_2(s)-u_1(s)\Vert^2+\Vert\varphi_2(s)-\varphi_1(s)\Vert^2\big)%
ds,\qquad\forall t\geq\tau.  \label{expatt8}
\end{align}
}
\end{lem}

\begin{proof}
By using the Poincaré inequality for $u$ and the Poincaré-Wirtinger inequality
for $\varphi $, i.e.,
\begin{equation}
\lambda _{1}\Vert u\Vert ^{2}\leq \Vert \nabla u\Vert ^{2},\qquad \Vert
\varphi -\overline{\varphi }\Vert ^{2}\leq c_{\Omega }\Vert \nabla \varphi
\Vert ^{2},  \label{PWineq}
\end{equation}%
from \eqref{expatt4} we have
\begin{equation*}
\frac{d}{dt}\big(\Vert u\Vert ^{2}+\Vert \varphi \Vert ^{2}\big)+\frac{\nu
\lambda _{1}}{2}\Vert u\Vert ^{2}+\frac{c_{0}}{2c_{\Omega }}\Vert \varphi
\Vert ^{2}\leq 2\gamma (t)\big(\Vert u\Vert ^{2}+\Vert \varphi \Vert ^{2}%
\big)+\frac{c_{0}|\Omega |}{2c_{\Omega }}\overline{\varphi }^{2},
\end{equation*}%
which yields
\begin{equation}
\frac{d}{dt}\big(\Vert u\Vert ^{2}+\Vert \varphi \Vert ^{2}\big)+k\big(\Vert
u\Vert ^{2}+\Vert \varphi \Vert ^{2}\big)\leq C_{\tau }\big(\Vert u\Vert
^{2}+\Vert \varphi \Vert ^{2}\big),  \label{expatt7}
\end{equation}%
where $k:=\min (\lambda _{1}\nu ,c_{0}/c_{\Omega })/2$ and $C_{\tau }$ is a
positive constant such that $2\gamma (t)+c_{0}/2c_{\Omega }\leq C_{\tau }$
for all $t\geq \tau $. 
By using Gronwall's lemma we immediately see from \eqref{expatt7} that $%
\Vert u\Vert ^{2}+\Vert \varphi \Vert ^{2}$ is controlled by the right-hand
side of \eqref{expatt8}. Furthermore, we also have
\begin{equation*}
\Big|\int_{\Omega }F\big(\varphi _{2}(t)\big)-\int_{\Omega }F\big(\varphi
_{1}(t)\big)\Big|\leq C_{\tau }\Vert \varphi (t)\Vert ,\qquad \forall t\geq
\tau .
\end{equation*}%
%
%
%
%
%
%
%
%
%
%
%
%
%
%
%
%
%
Hence, the proof of \eqref{expatt8} is complete.
\end{proof}

\begin{lem}
\label{expattlem3} Let the assumptions of Lemma \ref{expattlem1} be
satisfied.
Let $u_{0i}\in G_{div}$, $\varphi_{0i}\in H$ with $F(\varphi_{0i})\in
L^1(\Omega)$ and $[u_i,\varphi_i]$ be the corresponding weak solutions, $%
i=1,2$. Then, for every $\tau>0$ there exists $C_\tau>0$
such that
\begin{align}
&\Vert u_{2,t}-u_{1,t}\Vert_{L^2(\tau,t;V_{div}^{\prime})}^2
+\Vert\varphi_{2,t}-\varphi_{1,t}\Vert_{L^2(\tau,t;D(B_N)^{\prime})}^2
\notag \\
&\leq C_\tau e^{C_\tau t}\big(\Vert
u_2(\tau)-u_1(\tau)\Vert^2+\Vert\varphi_2(\tau)-\varphi_1(\tau)\Vert^2\big)%
,\qquad\forall t\geq\tau.  \label{expatt16}
\end{align}
\end{lem}

\begin{proof}
Consider the variational formulation of \eqref{CHdiff} and \eqref{chpotdiff}%
, namely,
\begin{equation}
\langle \varphi _{t},\psi \rangle =-(\nabla \widetilde{\mu },\nabla \psi
)-(u\cdot \nabla \varphi _{1},\psi )-(u_{2}\cdot \nabla \varphi ,\psi
),\qquad \forall \psi \in V,  \label{expatt11}
\end{equation}%
and take $\psi \in D(B_{N})$. Then, for every $\tau >0$ we see that there
exists $C_{\tau }>0$ such that
\begin{equation}
|(\nabla \widetilde{\mu },\nabla \psi )|=|(\widetilde{\mu },B_{N}\psi )|\leq
\Vert \widetilde{\mu }\Vert \Vert \psi \Vert _{D(B_{N})}\leq C_{\tau }\Vert
\varphi \Vert \Vert \psi \Vert _{D(B_{N})},\qquad \forall t\geq \tau .
\label{expatt9}
\end{equation}%
Moreover, we have
\begin{equation*}
|(u\cdot \nabla \varphi _{1},\psi )|=|(u\cdot \nabla \psi ,\varphi _{1})\leq
c\Vert \nabla u\Vert \Vert \varphi _{1}\Vert \Vert \psi \Vert
_{D(B_{N})}\leq C\Vert \nabla u\Vert \Vert \psi \Vert _{D(B_{N})},
\end{equation*}%
where in this case it is enough to use the dissipative estimate %
\eqref{dissest} and therefore the constant $C$ does not depend on $\tau $
but depends on $h$, $\mathcal{E}(z_{01})$ and $\eta $ only. Concerning the
last term on the right-hand side of \eqref{expatt11} we have
\begin{equation}
|(u_{2}\cdot \nabla \varphi ,\psi )|=|(u_{2}\cdot \nabla \psi ,\varphi
)|\leq c\Vert \nabla u_{2}\Vert \Vert \varphi \Vert \Vert \psi \Vert
_{D(B_{N})}\leq C_{\tau }\Vert \varphi \Vert \Vert \psi \Vert
_{D(B_{N})},\qquad \forall t\geq \tau .  \label{expatt10}
\end{equation}%
Plugging \eqref{expatt9}--\eqref{expatt10} into \eqref{expatt11}, we get
\begin{equation}
\Vert \varphi _{t}\Vert _{D(B_{N})^{\prime }}\leq C_{\tau }\big(\Vert
\varphi \Vert +\Vert \nabla u\Vert \big),\qquad \forall t\geq \tau .
\label{diss-est5}
\end{equation}%
Therefore, taking also \eqref{expatt0} into account, we have
\begin{equation}
\Vert \varphi _{t}\Vert _{L^{2}(\tau ,t;D(B_{N})^{\prime })}\leq C_{\tau
}e^{C_{\tau }t}\big(\Vert u(\tau )\Vert +\Vert \varphi (\tau )\Vert \big)%
,\qquad \forall t\geq \tau .  \label{expatt14bis}
\end{equation}

In order to obtain an estimate for $u_{2,t}-u_{1,t}$ let us consider the
difference of the Navier-Stokes equations written for two weak solutions in
the variational formulation, i.e.,
\begin{align}
& \langle u_{t},v\rangle =-\nu (\nabla u,\nabla
v)-b(u_{2},u_{2},v)+b(u_{1},u_{1},v)  \notag \\
& -\frac{1}{2}\big(\nabla a\varphi (\varphi _{1}+\varphi _{2}),v\big)-\big(%
(J\ast \varphi )\nabla \varphi _{2},v\big)-\big((J\ast \varphi _{2})\nabla
\varphi ,v\big),\qquad \forall v\in V_{div}.  \label{expatt13}
\end{align}%
Thanks to \eqref{expatt12} the last three terms on the right-hand side can
be easily estimated as follows
\begin{align*}
& \frac{1}{2}\big|\big(\nabla a\varphi (\varphi _{1}+\varphi _{2}),v\big)%
\big|\leq c\Vert \nabla a\Vert _{L^{\infty }}\Vert \varphi \Vert \Vert
\varphi _{1}+\varphi _{2}\Vert _{L^{\infty }}\Vert v\Vert \leq C_{\tau
}\Vert \varphi \Vert \Vert v\Vert _{V_{div}}, \\
& \big|\big((J\ast \varphi )\nabla \varphi _{2},v\big)\big|=\big|\big(%
(\nabla J\ast \varphi )\varphi _{2},v\big)\big|\leq c\Vert \nabla J\Vert
_{L^{1}}\Vert \varphi \Vert \Vert \varphi _{2}\Vert _{L^{\infty }}\Vert
v\Vert \leq C_{\tau }\Vert \varphi \Vert \Vert v\Vert _{V_{div}}, \\
& \big|\big((J\ast \varphi _{2})\nabla \varphi ,v\big)\big|=\big|\big(%
(\nabla J\ast \varphi _{2})\varphi ,v\big)\big|\leq c\Vert \nabla J\Vert
_{L^{1}}\Vert \varphi _{2}\Vert _{L^{\infty }}\Vert \varphi \Vert \Vert
v\Vert \leq C_{\tau }\Vert \varphi \Vert \Vert v\Vert _{V_{div}},
\end{align*}%
for all $t\geq \tau $. Furthermore, the trilinear form can be controlled by
using \eqref{standest2D}, that is, 
\begin{align*}
& |b(u_{2},u_{2},v)-b(u_{1},u_{1},v)|=|b(u_{2},u,v)+b(u,u_{1},v)| \\
& \leq c\big(\Vert \nabla u_{1}\Vert +\Vert \nabla u_{2}\Vert \big)\Vert
\nabla u\Vert \Vert \nabla v\Vert \leq C_{\tau }\Vert \nabla u\Vert \Vert
\nabla v\Vert ,\qquad \forall t\geq \tau .
\end{align*}%
Combining the last four estimates with \eqref{expatt13} we obtain
\begin{equation*}
\Vert u_{t}\Vert _{V_{div}^{\prime }}\leq C_{\tau }\big(\Vert \nabla u\Vert
+\Vert \varphi \Vert \big),\qquad \forall t\geq \tau ,
\end{equation*}%
Thus, recalling \eqref{expatt0}, we deduce 
\begin{equation}
\Vert u_{t}\Vert _{L^{2}(\tau ,t;V_{div}^{\prime })}\leq C_{\tau }e^{C_{\tau
}t}\big(\Vert u(\tau )\Vert +\Vert \varphi (\tau )\Vert \big),\qquad \forall
t\geq \tau .  \label{expatt15}
\end{equation}%
Finally, \eqref{expatt14bis} and \eqref{expatt15} yield \eqref{expatt16}.
\end{proof}

\begin{lem}
\label{expattlem4} Let the assumptions of Lemma \ref{expattlem1} be
satisfied.
Let $u_{0i}\in G_{div}$, $\varphi_{0i}\in H$ with $F(\varphi_{0i})\in
L^1(\Omega)$
$i=1,2$. Then, for every $\tau>0$ and every $T>0$ there exists $C_{\tau,T}>0$
depending also on $T$ such that
\begin{align}
&\rho_{\mathcal{X}_\eta}(S_\eta(t_2)z_{02},S_\eta(t_1)z_{01})\leq C_{\tau,T}%
\big(\rho_{\mathcal{X}_\eta}(S_\eta(\tau)z_{02},S_\eta(%
\tau)z_{01})+|t_2-t_1|^{1/2}\big),  \label{expatt18}
\end{align}
for all $t_1,t_2\in [\tau,\tau+T]$, where $z_{0i}:=[u_{0i},\varphi_{0i}]$, $%
i=1,2$.
\end{lem}

\begin{proof}
Setting $S_\eta(t)z_{0i}:=[u_i(t),\varphi_i(t)]$, $i=1,2$, we have {\small
\begin{align}
&\rho_{\mathcal{X}_\eta}(S_\eta(t_2)z_{01},S_\eta(t_1)z_{01})  \notag \\
&=\Vert u_1(t_2)-u_1(t_1)\Vert+\Vert\varphi_1(t_2)-\varphi_1(t_1)\Vert +\Big|%
\int_\Omega F(\varphi_1(t_2))-\int_\Omega F(\varphi_1(t_1))\Big|  \notag \\
&\leq\Vert u_{1,t}\Vert_{L^2(t_1,t_2;G_{div})}|t_2-t_1|^{1/2}
+\Vert\varphi_{1,t}\Vert_{L^\infty(\tau,\infty;H)}|t_2-t_1|+C_\tau\Vert%
\varphi_{1,t}\Vert_{L^\infty(\tau,\infty;H)}|t_2-t_1|  \notag \\
&\leq C_{\tau,T}|t_2-t_1|^{1/2},\qquad\forall t_1,t_2\in[\tau,\tau+T],
\label{expatt19}
\end{align}
} where we have used \eqref{expatt17}. Furthermore we have {\small
\begin{align}
&\rho_{\mathcal{X}_\eta}(S_\eta(t_2)z_{02},S_\eta(t_2)z_{01})  \notag \\
&= \Vert u_2(t_2)-u_1(t_2)\Vert+\Vert\varphi_2(t_2)-\varphi_1(t_2)\Vert +%
\Big|\int_\Omega F(\varphi_2(t_2))-\int_\Omega F(\varphi_1(t_2))\Big|  \notag
\\
&\leq C_\tau e^{C_\tau(\tau+T)}\big(\Vert
u_2(\tau)-u_1(\tau)\Vert+\Vert\varphi_2(\tau)-\varphi_1(\tau)\Vert\big) \leq
C_{\tau,T}\rho_{\mathcal{X}_\eta}(S_\eta(\tau)z_{02},S_\eta(\tau)z_{01}).
\label{expatt20}
\end{align}
} From \eqref{expatt19} and \eqref{expatt20} we get \eqref{expatt18}.
\end{proof}

We now recall the following abstract result on the existence of exponential
attractors \cite[Proposition 3.1]{EZ}. This result, together with the lemmas
above, will be used to prove Theorem \ref{expattthm}.

\begin{prop}
\label{expattr} Let $\mathcal{H}$ be a metric space (with metric $\rho_{%
\mathcal{H}}$) and let $\mathcal{V},\mathcal{V}_1$ be two Banach spaces such
that the embedding $\mathcal{V}_1 \hookrightarrow\mathcal{V}$
is compact. Let $\mathbb{B}$ be a bounded subset of $\mathcal{H}$ and let $%
\mathcal{S}:\mathbb{B}\to\mathbb{B}$ be a map such that
\begin{align}
& \rho_{\mathcal{H}}\big(\mathcal{S}w_{02},\mathcal{S}w_{01}\big)\leq \gamma
\rho_{\mathcal{H}}(w_{02},w_{01})+K\Vert\mathcal{T}w_{02}-\mathcal{T}%
w_{01}\Vert_{\mathcal{V}}, \qquad\forall w_{01},w_{02}\in\mathbb{B},
\label{contrass}
\end{align}
where $\gamma\in (0,\frac{1}{2})$, $K\geq 0$ and $\mathcal{T}:\mathbb{B}\to\mathcal{V}_1$
is a globally Lipschitz continuous map, i.e.,
\begin{align}
\Vert\mathcal{T}w_{02}-\mathcal{T}w_{01}\Vert_{\mathcal{V}_1}\leq L \rho_{%
\mathcal{H}}(w_{02},w_{01}), \qquad\forall w_{01},w_{02}\in\mathbb{B},
\label{lipschass}
\end{align}
for some $L\geq 0$. Then, there exists a (discrete) exponential attractor $%
\mathcal{M}_d\subset\mathbb{B}$ for the (time discrete) semigroup $\{%
\mathcal{S}^n\}_{n=0,1,2,\dots}$ on $\mathbb{B}$ (with the topology of $%
\mathcal{H}$ induced on $\mathbb{B}$).
\end{prop}

\begin{proof}[Proof of Theorem \protect\ref{expattthm}]

Let $\mathcal{B}_0$ be a bounded absorbing set in $\mathcal{X}_\eta$. The
existence of such a bounded absorbing set has been proven in \cite{FG1}.
Indeed, it is immediate to check that the argument of \cite[Proposition 4]%
{FG1} still applies with our choice for the metric $\rho_{\mathcal{X}_\eta}$%
. Let $t_0=t_0(\mathcal{B}_0)\geq 0$ be a time such that $S_\eta(t)\mathcal{B%
}_0\subset\mathcal{B}_0$ for all $t\geq t_0$. Due to
\eqref{pregtau3} we can fix $t^\ast=t^\ast(\mathcal{B}_0)\geq t_0$ such that
$S_\eta(t)\mathcal{B}_0\subset B_{\mathcal{Z}_\eta^p}(0,\Lambda(\eta))$ for
all $t\geq t^\ast$, where $B_{\mathcal{Z}_\eta^p}(0,\Lambda(\eta))$ is the
closed ball in $\mathcal{Z}_\eta^p$ with radius $\Lambda(\eta)$ and $%
\Lambda(\eta)$ a positive constant which depends only on $\eta$. The
(complete) metric space $\mathcal{Z}_\eta^p$ is given by
\begin{align}
&\mathcal{Z}_\eta^p:=V_{div}\times\{\varphi\in W^{1,p}(\Omega):\:\: |%
\overline{\varphi}|\leq\eta\},
\end{align}
endowed with the metric
\begin{equation*}
d_{\mathcal{Z}_\eta^p}(z_2,z_1)=\|\nabla u_2-\nabla
u_1\|+\|\varphi_2-\varphi_1\|_{W^{1,p}(\Omega)}, \quad\forall
z_i:=[u_i,\varphi_i]\in\mathcal{Z}_\eta^p,\qquad i=1,2.
\end{equation*}
Note that the terms in the integrals of $F(\varphi_1),F(\varphi_2)$ are
omitted in the metric since, for $p>2$, we have the embedding $%
W^{1,p}(\Omega)\hookrightarrow C(\overline{\Omega})$.

Let us now set
\begin{align}
&\mathcal{B}_1:=\bigcup_{t\geq t^\ast}S_\eta(t)\mathcal{B}_0.
\end{align}
Then, $\mathcal{B}_1$ is bounded in $\mathcal{Z}_\eta^p$ and positively
invariant for $S_\eta(t)$. It is easy to see that it is also absorbing in $%
\mathcal{X}_\eta$. Indeed, if $B$ is a bounded subset of $\mathcal{X}_\eta$
and $t_0=t_0(B)$ is such that $S_\eta(t_0)B\subset\mathcal{B}_0$, then we
have $S_\eta(t)B\subset\cup_{\tau\geq
t^\ast}S_\eta(\tau+t_0)B\subset\cup_{\tau\geq t^\ast}S_\eta(\tau)\mathcal{B}%
_0=:\mathcal{B}_1$, for all $t\geq t_0+t^\ast$. Furthermore, we set
\begin{equation*}
\mathbb{B}:=S_\eta(1)\mathcal{B}_1.
\end{equation*}
Then, $\mathbb{B}\subset B_{\mathcal{Z}_\eta^p}(0,\Lambda(\eta))$ is
positively invariant and still absorbing in $\mathcal{X}_\eta$.

By choosing $\tau=1$ in Lemma \ref{expattlem2}, then \eqref{expatt8} can be
written as follows
\begin{align}
&\rho_{\mathcal{X}_\eta}\big(S_\eta(t)z_{02},S_\eta(t)z_{01}\big) \leq C_1
e^{-kt/2}\rho_{\mathcal{X}_\eta}\big(S_\eta(1)z_{02},S_\eta(1)z_{01}\big)
\notag \\
& +C_1\Vert
S_\eta(\cdot)z_{02}-S_\eta(\cdot)z_{01}\Vert_{L^2(1,t;G_{div}\times H)},
\qquad\forall t\geq 1,\qquad\forall z_{01},z_{02}\in\mathcal{X}_\eta,
\label{expatt21}
\end{align}
where $C_1>0$ depends only on $\mathcal{E}(z_{01})$, $\mathcal{E}(z_{02})$
and $\eta$. From \eqref{expatt21} we therefore get
\begin{align}
&\rho_{\mathcal{X}_\eta}\big(S_\eta(t-1)w_{02},S_\eta(t-1)w_{01}\big) \leq
C_1 e^{-kt/2}\rho_{\mathcal{X}_\eta}\big(w_{02},w_{01}\big)  \notag \\
&+C_1\Vert
S_\eta(\cdot)w_{02}-S_\eta(\cdot)w_{01}\Vert_{L^2(0,t-1;G_{div}\times H)},
\qquad\forall t>1,\qquad\forall w_{01},w_{02}\in\mathbb{B}.  \label{expatt22}
\end{align}
Observe that, since $w_{0i}=S(1)z_{0i}$, with $z_{0i}\in\mathcal{B}_1$, $%
i=1,2$, and $\mathcal{B}_1$ is bounded in $\mathcal{Z}_\eta^p$, then $C_1$
does not depend on $w_{01},w_{02}$.

Choosing $\tau=1$ also in Lemma \ref{expattlem1} and in Lemma \ref%
{expattlem3}, and combining \eqref{expatt0} with \eqref{expatt16} we can
write
\begin{align}
&\Vert S_\eta(\cdot)z_{02}-S_\eta(\cdot)z_{01}\Vert_{L^2(1,t;V_{div}\times
V)}^2 +\Vert \partial_t S_\eta(\cdot)z_{02}-\partial_t
S_\eta(\cdot)z_{01}\Vert_{L^2(1,t;V_{div}^{\prime}\times D(B_N)^{\prime})}^2
\notag \\
&\leq C_1 e^{C_1 t}\rho_{\mathcal{X}_\eta}^2(S_\eta(1)z_{02},S_%
\eta(1)z_{01}), \qquad\forall t\geq 1,\qquad\forall z_{01},z_{02}\in\mathcal{%
X}_\eta.
\end{align}
Thus we find
\begin{align}
&\Vert S_\eta(\cdot)w_{02}-S_\eta(\cdot)w_{01}\Vert_{L^2(0,t-1;V_{div}\times
V)}^2 +\Vert \partial_t S_\eta(\cdot)w_{02}-\partial_t
S_\eta(\cdot)w_{01}\Vert_{L^2(0,t-1;V_{div}^{\prime}\times
D(B_N)^{\prime})}^2  \notag \\
&\leq C_1 e^{C_1 t}\rho_{\mathcal{X}_\eta}^2(w_{02},w_{01}), \qquad\forall
t\geq 1,\qquad\forall w_{01},w_{02}\in\mathbb{B},  \label{expatt23}
\end{align}
where, as pointed out above, the constant $C_1$ does not depend on $w_{01}$ and $w_{02}$.

Let us now introduce the following spaces
\begin{align*}
&\mathcal{H}:=\mathcal{X}_\eta=G_{div}\times\mathcal{Y}_\eta  \notag \\
&\mathcal{V}_1:=L^2(0,T;V_{div}\times V)\cap H^1(0,T;V_{div}^{\prime}\times
D(B_N)^{\prime}) \\
&\mathcal{V}:=L^2(0,T;G_{div}\times H),
\end{align*}
with $T>0$ fixed such that $C_1 e^{-k(T+1)/2}<1/2$, where $C_1$ and $k$ are
the same constants that appear in the first term on the right-hand side of %
\eqref{expatt22}.
Notice that, due to the Aubin-Lions lemma, $\mathcal{V}_1$ is compactly
embedded into $\mathcal{V}$.

Then, take $\mathcal{S}:=S_\eta(T)$ and define a map $\mathcal{T}:\mathbb{B}%
\to\mathcal{V}_1$ in the following way: for every $w_0\in\mathbb{B}$ we set $%
\mathcal{T}w_0:=w:=S_\eta(\cdot)w_0$, i.e., $w\in\mathcal{V}_1$ is the
(strong) solution corresponding to the initial datum $w_0$.

It is now easy to see that choosing the spaces $\mathcal{H},\mathcal{V},%
\mathcal{V}_1$, the set $\mathbb{B}$, and the maps $\mathcal{S}$, $\mathcal{T%
}$ as above, then the conditions of Proposition \ref{expattr}
are satisfied. Indeed, \eqref{contrass} and \eqref{lipschass} follow from %
\eqref{expatt8} and \eqref{expatt23}, respectively, both written for $t=T+1$%
.

Therefore, Proposition \ref{expattr} entails the existence of a (discrete)
exponential attractor $\mathcal{M}_\eta^d\subset\mathbb{B}$ for the (time
discrete) semigroup $\{\mathcal{S}^n\}_{n=0,1,2,\dots}$ on $\mathbb{B}$
(with the topology of $\mathcal{H}$ induced on $\mathbb{B}$). Since $\mathbb{%
B}$ is absorbing in $\mathcal{H}$, then the basin of attraction of $\mathcal{%
M}_\eta^d$ is the whole phase space $\mathcal{H}$.

In order to prove the existence of the exponential attractor $\mathcal{M}%
_{\eta }$ for $(\mathcal{X}_{\eta },S_{\eta }(t))$ with continuous time we
observe first that \eqref{expatt18} written with $\tau =1$ (the time $T$ is
chosen as above) yields
\begin{equation*}
\rho _{\mathcal{X}_{\eta }}(S_{\eta }(t_{2}-1)w_{02},S_{\eta
}(t_{1}-1)w_{01})\leq C_{1,T}\big(\rho _{\mathcal{X}_{\eta
}}(w_{02},w_{01})+|t_{2}-t_{1}|^{1/2}\big),
\end{equation*}%
for all $w_{01},w_{02}\in \mathbb{B}$ and for all $t_{1},t_{2}\in \lbrack
1,1+T]$. Hence
\begin{equation*}
\rho _{\mathcal{X}_{\eta }}(S_{\eta }(t^{\prime \prime })w_{02},S_{\eta
}(t^{\prime })w_{01})\leq C_{1,T}\big(\rho _{\mathcal{X}_{\eta
}}(w_{02},w_{01})+|t^{\prime \prime }-t^{\prime }|^{1/2}\big),
\end{equation*}%
for all $w_{01},w_{02}\in \mathbb{B}$ and for all $t^{\prime \prime
},t^{\prime }\in \lbrack 0,T]$. Therefore, the map $[t,z]\mapsto S_{\eta
}(t)z$ is uniformly Hölder continuous (with exponent $1/2$) on $[0,T]\times
\mathbb{B}$, where $\mathbb{B}$ is endowed with the $\mathcal{H}- $metric.
Therefore, the exponential attractor $\mathcal{M}_{\eta }$ for the
continuous time case can be obtained by the classical expression
\begin{equation*}
\mathcal{M}_{\eta }=\bigcup_{t\in \lbrack 0,T]}S_{\eta }(t)\mathcal{M}_{\eta
}^{d},
\end{equation*}%
and this concludes the proof of the theorem.
\end{proof}

We conclude by proving the existence of exponential attractors when the
viscosity $\nu $ depends on $\varphi$ and satisfies the assumption (\ref{H2bis})
in Remark \ref{weak-H2}. In view of Theorems \ref{thmncvisc} and \ref%
{weak-uniq} we can define a dynamical system by using strong solutions.
Indeed, taking $d=2$ and $h\in G_{div}$, we have that for every fixed $\eta
\geq 0$ system \eqref{sy1}--\eqref{sy5} generates a semigroup $\{Z_{\eta
}(t)\}_{t\geq 0}$ of {\itshape closed} operators on the metric space $%
\mathcal{K}_{\eta }$ given by
\begin{equation*}
\mathcal{K}_{\eta }:=V_{div}\times \{\varphi \in H^{2}\left( \Omega \right)
:|\overline{\varphi }|\leq \eta \}
\end{equation*}%
endowed with the (weaker) metric
\begin{equation*}
\varrho (z_{2},z_{1})=\Vert u_{2}-u_{1}\Vert+\Vert \varphi _{2}-\varphi
_{1}\Vert,\quad \forall z_{i}:=[u_{i},\varphi _{i}]\in \mathcal{K}_{\eta },%
\text{ }i=1,2.
\end{equation*}

We are now ready to state and prove the following.

\begin{thm}
\label{strong-expo}Assume (H1), (H3)-(H5) and (\ref{H2bis}). Consider either
$J\in W^{2,1}(B_{\delta })$ or $J$ admissible. The dynamical system $(%
\mathcal{K}_{\eta },Z_{\eta }(t))$ possesses an exponential attractor $%
\mathcal{E}_{\eta }$ which is bounded in $V_{div}\times H^{2}\left( \Omega
\right) $ such that the following properties are satisfied:

\begin{itemize}
\item positive invariance: $Z_{\eta }(t)\mathcal{E}_{\eta }\subseteq
\mathcal{E}_{\eta }$ for all $t\geq 0$;

\item finite dimensionality: dim$_{F}(\mathcal{E}_{\eta },G_{div}\times
H)<\infty $;

\item exponential attraction: $\exists$ $Q:\mathbb{R}%
^{+}\rightarrow \mathbb{R}^{+}$ increasing and $\kappa >0$ such that, for
all $R>0$ and for all $\mathcal{B}\subset \mathcal{K}_{\eta }$ with $%
\sup_{z\in \mathcal{B}}\rho (z,0)\leq R$ there holds
\begin{equation*}
dist_{\mathcal{K}_{\eta }}(Z_{\eta }(t)\mathcal{B},\mathcal{E}_{\eta })\leq
Q(R)e^{-\kappa t},\qquad \forall t\geq 0.
\end{equation*}
\end{itemize}
\end{thm}

\begin{proof}
\emph{Step 1.} We will briefly show that a dissipative estimate like (\ref%
{pregtau3}) still holds for the strong solution of \eqref{sy1}--\eqref{sy5}
under the assumptions of the theorem. More precisely, the following estimate
holds
\begin{equation}
\Vert \nabla u(t)\Vert +\Vert \varphi (t)\Vert _{H^{2}(\Omega
)}+\int_{t}^{t+1}\Vert u(s)\Vert _{H^{2}(\Omega )^{2}}^{2}ds\leq \Theta
(\eta ),\qquad \forall t\geq t^{\ast }.  \label{diss-strong}
\end{equation}%
for some positive constant $\Theta$ independent of the initial data and
time, and some time $t_{\#}>0$ which depends only $\mathcal{E}(z_{0})$. In
order to get this estimate, first we recall estimate (\ref{dissest}) by
Theorem \ref{thm} which also holds for nonconstant viscosity. The proof of (%
\ref{diss-strong}) follows immediately from the proof of Theorem \ref%
{thmncvisc}. Indeed, we observe preliminarily that (\ref{sup-norm}), (\ref%
{holdi}) and (\ref{variant2}) already hold uniformly with respect to time
and initial data in the nonconstant case, i.e., there exists a time $%
t_{\#}>0 $, depending only on $\mathcal{E}(z_{0}),$ such that%
\begin{equation}
\varphi \in L^{\infty }\left( t_{\#},\infty ;L^{\infty }\left( \Omega
\right) \cap V\right) \cap W^{1,2}\left( t_{\#},\infty ;H\right)
\label{sup-norm2}
\end{equation}%
and%
\begin{equation}
\sup_{t\geq t_{\#}}\left\Vert \varphi \right\Vert _{C^{\delta /2,\delta
}\left( \left[ t,t+1\right] \times \overline{\Omega }\right) }\leq \Theta
(\eta ).  \label{sup-norm3}
\end{equation}%
In particular, this regularity allows us to obtain $\mu \in L^{\infty
}\left( t_{\#},\infty ;L^{\infty }\left( \Omega \right) \cap V\right) $ and $%
l\in L^{2}(t_{\#},\infty ;\left( L^{2}\left( \Omega \right) \right) ^{2})$
uniformly. This can be done by arguing exactly in the same fashion as in the
derivation of estimates (\ref{est25})-(\ref{est25bis}), with the exception
that the constant $R>0$ is such that $\mathrm{ess}\sup_{t\in \left(
t_{\#},\infty \right) }\left\Vert \varphi \left( t\right) \right\Vert
_{L^{\infty }}\leq R$. Then, we can employ the same procedure as in the
proof of Theorem \ref{thmncvisc} (with a function $Q=Q%
\left( R\right) >0$ which is now independent of the initial data, by (\ref%
{sup-norm2})-(\ref{sup-norm3})) to deduce by virtue of the uniform Gronwall
lemma (see \cite[Chapter III, Lemma 1.1]{T}) that
\begin{equation}
u\in L^{\infty }(t_{\ast },\infty ;V_{div})\cap L^{2}(t_{\ast },\infty
;H^{2}(\Omega )^{2}),\quad u_{t}\in L^{2}(t_{\ast },\infty ;G_{div}),
\label{diss-strong2}
\end{equation}%
for some $t_{\ast }\geq 1$ depending only on $t_{\#}.$ Finally, arguing
exactly as in the proof of Theorem \ref{thmncvisc} we deduce $\varphi \in
L^{\infty }\big(t_{\ast },\infty ;H^{2}(\Omega )\big)$ uniformly with
respect to time and the data. Note that estimate (\ref{diss-strong}) entails
the existence of a bounded absorbing set $\mathcal{B}_{2}\subset \mathcal{K}%
_{\eta }$ for the semigroup $Z_{\eta }(t).$

\emph{Step 2.} As in the proof of Theorem \ref{expattthm}, it will be
sufficient to construct the exponential attractor for the restriction of $%
Z_{\eta }(t)$ on this set $\mathcal{B}_{2}$. Thus, it suffices to verify the
validity of Lemmas \ref{expattlem2} and \ref{expattlem3} for the difference
{\color{black}$u=u_{2}-u_{1},$ $\varphi =\varphi _{2}-\varphi _{1}$},
where $\left(u_{i},\varphi _{i}\right) $ is a (given) strong solution and $i=1,2.$ The
first one is an immediate consequence of estimate (\ref{vv12}) (see the
proof of Theorem \ref{weak-uniq}) and the application of Poincaré-type
inequalities (\ref{PWineq}) (see the proof of Lemma \ref{expattlem2}).
Indeed, in the nonconstant case we have {\small
\begin{align}
& \Vert u(t)\Vert ^{2}+\Vert \varphi (t)\Vert ^{2}  \notag \\
& \leq C\big(\Vert u(\tau ))\Vert ^{2}+\Vert \varphi (\tau )\Vert ^{2}\big)%
e^{-kt}+C\int_{\tau }^{t}\big(\Vert u(s)\Vert ^{2}+\Vert \varphi (s)\Vert
^{2}\big)ds,\qquad \forall t\geq \tau ,  \label{diss-est1}
\end{align}%
}for some constant $C=C_{\tau }>0$, where $\left( u_{i}\left( \tau \right)
,\varphi _{i}\left( \tau \right) \right) \in \mathcal{B}_{2}$ for each $%
i=1,2.$ For the second one, we observe that in order to estimate $%
u_{t}:=u_{2,t}-u_{1,t}$, we have%
\begin{align}
\langle u_{t},v\rangle & =-(\nu \left( \varphi _{2}\right) \nabla u,\nabla
v)-(\left( \nu \left( \varphi _{1}\right) -\nu \left( \varphi _{2}\right)
\right) \nabla u_{1},\nabla v)  \notag \\
& -b(u_{2},u_{2},v)+b(u_{1},u_{1},v)  \notag \\
& -\frac{1}{2}\big(\nabla a\varphi (\varphi _{1}+\varphi _{2}),v\big)-\big(%
(J\ast \varphi )\nabla \varphi _{2},v\big)-\big((J\ast \varphi _{2})\nabla
\varphi ,v\big),  \label{diss-est2}
\end{align}%
for all $v\in W:=\left( H^{2+\varepsilon }\left( \Omega \right) \right)
^{2}\cap V_{div}$ and some $\varepsilon >0$ (such that the embedding $%
H^{2+\varepsilon }\subset W^{1,\infty }$ holds). While all the terms on the
right-hand side of (\ref{diss-est2}), with the exception of the first two,
can be word by word estimated exactly as in the proof of Lemma \ref%
{expattlem3}, we notice that assumption (\ref{H2bis}) and the essential $%
L^{\infty }$-bound on $\varphi $ yield%
\begin{align}
\left\vert (\nu \left( \varphi _{2}\right) \nabla u,\nabla v)\right\vert &
\leq C\left\Vert \nabla u\right\Vert \left\Vert \nabla v\right\Vert ,
\label{diss-est3} \\
\left\vert (\left( \nu \left( \varphi _{1}\right) -\nu \left( \varphi
_{2}\right) \right) \nabla u_{1},\nabla v)\right\vert & \leq C\left\Vert
\varphi \right\Vert \left\Vert \nabla u_{1}\right\Vert \left\Vert
v\right\Vert _{H^{2+\varepsilon }}.  \notag
\end{align}%
Thus, we easily get%
\begin{equation}
\left\Vert u_{t}\right\Vert _{W^{\prime}}\leq C\left( \left\Vert \nabla
u\right\Vert +\left\Vert \varphi \right\Vert \right) ,\quad \forall t\geq
\tau ,  \label{diss-est4}
\end{equation}%
which together with (\ref{vv12}) and (\ref{diss-est5}) yields the following
estimate
\begin{equation}
\Vert u_{t}\left( t\right) \Vert _{L^{2}(\tau ,t;W^{\prime })}^{2}+\Vert
\varphi _{t}||_{L^{2}(\tau ,t;D(B_{N})^{\prime })}^{2}\leq Ce^{Ct}\big(\Vert
u(\tau )\Vert ^{2}+\Vert \varphi (\tau )\Vert ^{2}\big),\quad \forall t\geq
\tau .  \label{diss-est6}
\end{equation}%
Estimates (\ref{diss-est1}) and (\ref{diss-est6}) convey that a certain
smoothing property holds for the difference of any two strong solutions
associated with any two given initial data in $\mathcal{B}_{2}$.

\emph{Step 3}. It is now not difficult to finish the proof of the theorem,
using the abstract scheme of Proposition \ref{expattr} by arguing in a
similar fashion as in the proof of Theorem \ref{expattthm}. The differences
are quite minor and so we leave them to the interested reader.
\end{proof}

\begin{oss}
{\upshape
On account of \cite[Proofs of Proposition 1 and Lemma 3]{FGK} and {%
\color{black}\eqref{est19}}, using uniform Gronwall's lemma (see \cite[Chapter
III, Lemma 1.1]{T}), it is possible to show that any weak solution becomes a
strong solution in finite time. We remind that this property is based on the
validity of the energy identity \eqref{eniden}. Indeed, estimate %
\eqref{diss-strong} ensures that, given a weak trajectory $z$ starting from $%
z_{0}\in \mathcal{X}_{\eta }$ (cf. \eqref{phasespace}), there exists a time $%
t^{\ast }=t^{\ast }(z_{0})\geq 0$ such that $z(t)\in \mathcal{B}_{1}(\Lambda
(\eta ))$ for all $t\geq t^{\ast }$, where $\mathcal{B}_{1}(\Lambda (\eta ))$
is the closed ball in the space $V_{div}\times H^{2}(\Omega )$ with radius $%
\Lambda (\eta )$ and constraint $|\overline{\varphi }|\leq \eta $. Let us
briefly mention some consequences of this property. First, the global
attractor of the generalized semiflow on $\mathcal{X}_{\eta }$ generated by
the problem with nonconstant viscosity (see \cite{FG1}) is bounded in $%
V_{div}\times H^{2}(\Omega )$. Therefore we can show the validity of a
smoothing property (cf. \eqref{diss-est1} and \eqref{diss-est6}) on the
global attractor and deduce that it has finite fractal dimension. Moreover,
the regularizing effect also allows us to prove the precompactness of (weak)
trajectories (see \cite[Lemma 3]{FGK}). This is an essential ingredient to
establish the convergence of a weak solution to a single equilibrium which
can be done along the lines of \cite[Section 5]{FGK}. }
\end{oss}

{\color{black}
\section{Conclusions}

Uniqueness of a weak solution was proven for the nonlocal Cahn-Hilliard-Navier-Stokes in two dimensions with constant viscosity. This result holds either for a regular or a singular potential and also for singular potentials and degenerate mobility. Uniqueness of weak solutions seems out of reach if viscosity in the Navier-Stokes equations depends on $\varphi$.  Therefore we established first the existence of a strong solution, a nontrivial result in itself. Then we show weak-strong uniqueness. This was done by assuming constant mobility and regular potential. In the case of constant viscosity and singular potential, the existence of a strong solution seems difficult to obtain. However, this can be achieved when the mobility is
degenerate, provided some natural assumptions are satisfied (though we gave no proof here). On account of this, weak-strong uniqueness can also be demonstrated for nonconstant viscosity, degenerate mobility and singular potential. In the last section
we investigated the global longtime behavior of the corresponding dynamical system. Uniqueness of weak solutions allowed us
to prove the connectedness of the global attractor whose existence was obtained elsewhere. Then we established the existence
of an exponential attractors for weak solutions (constant mobility and regular potential). Finally, in the case of variable
viscosity, we showed that an exponential attractor can be still constructed by using strong solutions. These last two results essentially depend on the continuous dependence estimates which entail uniqueness.

\bigskip
\textbf{Acknowledgments}. The authors thank the reviewers for their remarks and suggestions. The first author was supported
by FP7-IDEAS-ERC-StG Grant $\sharp$256872 (EntroPhase). The first and third authors are members of the Gruppo Nazionale per l'Analisi Matematica, la Probabilit\`{a} e le loro Applicazioni (GNAMPA) of the Istituto Nazionale di Alta Matematica (INdAM).}


\begin{thebibliography}{99}
\bibitem{A1} H. Abels, {\itshape On a diffusive interface model for
two-phase flows of viscous, incompressible fluids with matched densities},
Arch. Ration. Mech. Anal. \textbf{194} (2009), 463-506.

\bibitem{A2} H. Abels, {\itshape Longtime behavior of solutions of a
Navier-Stokes/Cahn-Hilliard system}, Proceedings of the Conference
\textquotedblleft Nonlocal and Abstract Parabolic Equations and their
Applications\textquotedblright , Bedlewo, Banach Center Publ. \textbf{86}
(2009), 9-19.

{\color{black} \bibitem{AC} W. Arendt, R. Chill, Global existence for quasilinear diffusion
equations in isotropic nondivergence form, Ann. Sc. Norm. Super. Pisa Cl.
Sci. (5) \textbf{9} (2010), 523-539.}

\bibitem{Ba} J.M. Ball, {\itshape Continuity properties and global
attractors of generalized semiflows and the Navier-Stokes equation}, J.
Nonlinear Sci. \textbf{7} (1997), 475-502 (Erratum, J. Nonlinear Sci.
\textbf{8} (1998), 233).

\bibitem{BH1} P.W. Bates, J. Han, {\itshape The Neumann boundary problem for
a nonlocal Cahn-Hilliard equation}, J. Differential Equations \textbf{212}
(2005), 235-277.

\bibitem{BRB} J. Bedrossian, N. Rodríguez, A. Bertozzi, {\itshape Local and
global well-posedness for an aggregation equation and Patlak-Keller-Segel
models with degenerate diffusion}, Nonlinearity \textbf{24} (2011),
1683-1714.

\bibitem{B} F. Boyer, {\itshape Mathematical study of multi-phase flow under
shear through order parameter formulation}, Asymptot. Anal. \textbf{20}
(1999), 175-212.

\bibitem{CG} C. Cao, C.G. Gal, {\itshape Global solutions for the 2D NS-CH
model for a two-phase flow of viscous, incompressible fluids with mixed
partial viscosity and mobility}, Nonlinearity \textbf{25} (2012), 3211-3234.

\bibitem{CFG} P. Colli, S. Frigeri, M. Grasselli, {\itshape Global existence
of weak solutions to a nonlocal Cahn-Hilliard-Navier-Stokes system}, J.
Math. Anal. Appl. \textbf{386} (2012), 428-444.

{\color{black} \bibitem{EE}  D.E. Edmunds, W.D. Evans, Spectral theory and differential
operators, Oxford Mathematical Monographs, Oxford Science Publications, The
Clarendon Press, Oxford University Press, New York, 1987.}

\bibitem{EMZ} M. Efendiev, A. Miranville, S. Zelik, {\itshape Exponential
attractors for a nonlinear reaction-diffusion system in $\mathbb{R}^{3}$},
C. R. Acad. Sci. Paris, Ser. I \textbf{330} (2000), 713-718.

\bibitem{EZ} M. Efendiev, S. Zelik, {\itshape Finite-dimensional attractors
and exponential attractors for degenerate doubly nonlinear equations}, Math.
Methods Appl. Sci. \textbf{32} (2009), 1638-1668.

{\color{black} \bibitem{EG} C.M. Elliott, H. Garcke, {\itshape On the Cahn-Hilliard
equation with degenerate mobility},  SIAM J. Math. Anal. \textbf{27} (1996),
404-423.}

\bibitem{FG1} S. Frigeri, M. Grasselli, {\itshape Global and trajectories
attractors for a nonlocal Cahn-Hilliard-Navier-Stokes system}, J. Dynam.
Differential Equations \textbf{24} (2012), 827-856.

\bibitem{FG2} S. Frigeri, M. Grasselli, {\itshape Nonlocal
Cahn-Hilliard-Navier-Stokes systems with singular potentials}, Dyn. Partial
Differ. Equ. \textbf{9} (2012), 273-304.

\bibitem{FGK} S. Frigeri, M. Grasselli, P. Krej\v{c}í, {\itshape Strong
solutions for two-dimensional nonlocal Cahn-Hilliard-Navier-Stokes systems},
J. Differential Equations \textbf{255} (2013), 2597-2614.

\bibitem{FGR} S. Frigeri, M. Grasselli, E. Rocca, {\itshape A diffuse
interface model for two-phase incompressible flows with nonlocal
interactions and nonconstant mobility}, {\color{black} Nonlinearity \textbf{28}
(2015), 1257-1293}. 

\bibitem{GG1} C.G. Gal, M. Grasselli, {\itshape Asymptotic behavior of a
Cahn-Hilliard-Navier-Stokes system in 2D}, Ann. Inst. H. Poincaré Anal. Non
Linéaire \textbf{27} (2010), 401-436.

\bibitem{GG2} C.G. Gal, M. Grasselli, {\itshape Trajectory attractors for
binary fluid mixtures in 3D}, Chinese Ann. Math. Ser. B \textbf{31} (2010),
655-678.

\bibitem{GG3} C.G. Gal, M. Grasselli, {\itshape Instability of two-phase
flows: a lower bound on the dimension of the global attractor of the
Cahn-Hilliard-Navier-Stokes system,} Phys. D \textbf{240} (2011), 629-635.

\bibitem{GG4} C.G. Gal, M. Grasselli, {\itshape Longtime behavior of
nonlocal Cahn-Hilliard equations}, Discrete Contin. Dyn. Syst. Ser. A
\textbf{34} (2014), 145-179.

\bibitem{GZ} H. Gajewski, K. Zacharias, {\itshape On a nonlocal phase
separation model}, J. Math. Anal. Appl. \textbf{286} (2003), 11-31.

\bibitem{GL0} G. Giacomin, J.L. Lebowitz, {\itshape Exact macroscopic
description of phase segregation in model alloys with long range interactions%
}, Phys. Rev. Lett. \textbf{76} (1996), 1094-1097.

\bibitem{GL1} G. Giacomin, J.L. Lebowitz, {\itshape Phase segregation
dynamics in particle systems with long range interactions. I. Macroscopic
limits}, J. Statist. Phys. \textbf{87} (1997), 37-61.

\bibitem{GL2} G. Giacomin, J.L. Lebowitz, {\itshape Phase segregation
dynamics in particle systems with long range interactions. II. Phase motion}%
, SIAM J. Appl. Math. \textbf{58} (1998), 1707-1729.

\bibitem{Kim12} J.S. Kim, {\itshape Phase-field models for multi-component
fluid flows}, Commun. Comput. Phys. \textbf{12} (2012), 613-661.

\bibitem{LS} C. Liu, J. Shen, {\itshape A phase field model for the mixture
of two incompressible fluids and its approximation by a Fourier spectral
method}, Phys. D \textbf{179} (2003), 211-228.

\bibitem{LP1} S.-O. Londen, H. Petzeltová, {\itshape Convergence of
solutions of a non-local phase-field system}, Discrete Contin. Dyn. Syst.
Ser. S \textbf{4} (2011), 653-670.

\bibitem{LP2} S.-O. Londen, H. Petzeltová, {\itshape Regularity and
separation from potential barriers for a non-local phase-field system}, J.
Math. Anal. Appl. \textbf{379} (2011), 724-735.

{\color{black} \bibitem{LSU} O.A. Lady\v{z}enskaja, V.A. Solonnikov, N.N. Ural'ceva, Linear
and quasilinear equations of parabolic type, AMS Transl. Monographs \textbf{23}, AMS, Providence, R.I. 1968.

\bibitem{NU} A.I. Nazarov, N.N. Ural'tseva, The Harnack inequality and
related properties of solutions of elliptic and parabolic equations with
divergence-free lower-order coefficients, St. Petersburg Math. J. \textbf{23} (2012),
93-115.}

\bibitem{PZ} V. Pata, S. Zelik, {\itshape A result on the existence of
global attractors for semigroups of closed operators}, Commun. Pure Appl.
Anal. \textbf{6} (2007), 481-486.

\bibitem{S} V.N. Starovoitov, {\itshape The dynamics of a two-component
fluid in the presence of capillary forces}, Math. Notes \textbf{62} (1997),
244-254.

\bibitem{T} R. Temam, Navier-Stokes equations and nonlinear
functional analysis, Second edition, CBMS-NSF Reg. Conf. Ser. Appl. Math.
\textbf{66}, SIAM, Philadelphia, PA, 1995.

\bibitem{ZWH} L. Zhao, H. Wu, H. Huang, {\itshape Convergence to equilibrium
for a phase-field model for the mixture of two viscous incompressible fluids}%
, Commun. Math. Sci. \textbf{7} (2009), 939-962.

{\color{black} \bibitem{ZY} Y. Sun, Z. Zhang, Global regularity for the initial-boundary
value problem of the 2-D Boussinesq system with variable viscosity and
thermal diffusivity, J. Differential Equations \textbf{255} (2013), 1069-1085.}
\end{thebibliography}
\end{document}